\documentclass[reqno]{amsart}
\usepackage[T1]{fontenc}
\usepackage[utf8]{inputenc}
\usepackage{amssymb}
\usepackage[mathscr]{eucal}
\usepackage{mathtools}
\usepackage{microtype}
\usepackage{enumerate}
\usepackage{xcolor}
\usepackage{graphicx}
\usepackage[colorlinks=true]{hyperref}
\hypersetup{urlcolor=blue,citecolor=red,linkcolor=blue}
\usepackage[initials]{amsrefs}
\graphicspath{{merge-views/}}
\numberwithin{equation}{section}
\theoremstyle{plain}
\newtheorem{theorem}{Theorem}[section]
\newtheorem{corollary}[theorem]{Corollary}

\newtheorem{lemma}[theorem]{Lemma}
\theoremstyle{remark}
\newtheorem{comments}[theorem]{Comments}
\newtheorem{remark}[theorem]{Remark}
\newtheorem*{remark*}{Remark}
\newtheorem*{remarks*}{Remarks}
\theoremstyle{definition}

\newtheorem*{acknowledgements*}{Acknowledgements}
\newtheorem*{assumption*}{Assumption}
\newtheorem*{definition*}{Definition}
\newtheorem*{notation*}{Notation}
\newtheorem*{notations*}{Notations}

\providecommand{\BS}[1]{\boldsymbol{#1}}
\providecommand{\C}[1]{\mathcal{#1}}
\providecommand{\D}[1]{\mathbb{#1}}

\newcommand{\dd}{\mathrm{d}}
\newcommand{\eul}{\mathrm{e}}
\newcommand{\ii}{\mathrm{i}}

\providecommand{\abs}[1]{\lvert#1\rvert}
\providecommand{\accol}[1]{\lbrace#1\rbrace}
\providecommand{\croch}[1]{\lbrack#1\rbrack}
\providecommand{\norm}[1]{\lVert#1\rVert}
\renewcommand{\Im}{\operatorname{Im}}
\newcommand{\model}{\operatorname{mod}}

\newcommand{\appr}{\mathrm{app}}

\newcommand{\cut}{\mathrm{cut}}
\newcommand{\diff}{\mathrm{dif}}

\newcommand{\ord}{\mathrm{O}}

\newcommand{\painl}{\mathrm{P}_{\text{IV}}}
\renewcommand{\Re}{\operatorname{Re}}

\allowdisplaybreaks
\begin{document}
\title[NLS with step-like background: transition zone]{The focusing NLS equation with step-like oscillating background: asymptotics in a transition zone}
\author[A. Boutet de Monvel]{Anne Boutet de Monvel}
\address{AB: Institut de Math\'ematiques de Jussieu-Paris Rive Gauche, Universit\'e de Paris, 75205 Paris Cedex 13, France.}
\email{anne.boutet-de-monvel@imj-prg.fr}
\author[J. Lenells]{Jonatan Lenells}
\address{JL: Department of Mathematics, KTH Royal Institute of Technology, 100 44 Stockholm, Sweden.}
\email{jlenells@kth.se}
\author[D. Shepelsky]{Dmitry Shepelsky}
\address{DS: B.~Verkin Institute for Low Temperature Physics and Engineering, 61103 Kharkiv, Ukraine.} 
\email{shepelsky@yahoo.com}
\date{}
\begin{abstract}
In a recent paper, we presented scenarios of long-time asymptotics for a solution of the focusing nonlinear Schr\"odinger equation whose initial data approach two different plane waves $A_j\eul^{\ii\phi_j}\eul^{-2\ii B_jx}$, $j=1,2$ at minus and plus infinity. In the shock case $B_1<B_2$ some scenarios include sectors of genus $3$, that is sectors $\xi_1<\xi<\xi_2$, $\xi\coloneqq\frac{x}{t}$ where the leading term of the asymptotics is expressed in terms of hyperelliptic functions attached to a Riemann surface $M(\xi)$ of genus $3$. The long-time asymptotic analysis in such a sector is performed in another recent paper. The present paper deals with the asymptotic analysis in a transition zone between two genus $3$ sectors $\xi_1<\xi<\xi_0$ and $\xi_0<\xi<\xi_2$. The leading term is expressed in terms of elliptic functions attached to a Riemann surface $\tilde M$ of genus $1$. A central step in the derivation is the construction of a local parametrix in a neighborhood of two merging branch points. We construct this parametrix by solving a model problem which is similar to the Riemann--Hilbert problem associated with the Painlev\'e IV equation.
\end{abstract}
\maketitle
\section{Introduction}  \label{sec:intro}

We consider the focusing nonlinear Schr\"odinger (NLS) equation
\begin{equation} \label{eq:nls}
\ii q_t+q_{xx}+2\abs{q}^2q=0,\qquad x\in\D{R},\quad t\geq 0,
\end{equation}
and a solution $q(x,t)$ of \eqref{eq:nls} such that
\begin{equation} \label{q0-limits}
q(x,0)\sim 
\begin{cases}
A_1\eul^{\ii\phi_1}\eul^{-2\ii B_1 x},&x\to -\infty,\\
A_2\eul^{\ii\phi_2}\eul^{-2\ii B_2 x},&x\to +\infty,
\end{cases}
\end{equation}
where $\accol{A_j,B_j,\phi_j}_1^2$ are real constants, $A_j>0$. Our general objective is the study of the long-time behavior of the solution $q(x,t)$ of \eqref{eq:nls}--\eqref{q0-limits}.

Equation \eqref{eq:nls} is a nonlinear PDE integrable by the Inverse Scattering Transform method \cite{AS}. The Riemann--Hilbert (RH) problem formalism of this method has proved to be highly efficient in studying various properties of solutions of integrable PDE, particularly, the long time asymptotics of the solution of the Cauchy problem \cites{DZ93, DVZ94} and the small dispersion limit \cites{DVZ97,KLM03}. The RH method, put into a rigorous shape by Deift and Zhou in \cite{DZ93},
was further developed in \cites{DVZ94,DVZ97} with the introduction of the so-called ``$g$-function mechanism''. This mechanism helps in the realization of the main idea of the method, which is to find a series of transformations of the original RH problem representation of the solution of the nonlinear equation in question leading ultimately to a model RH problem that can be solved explicitly. Moreover, the method makes it possible to obtain not only the main asymptotic term but also provides a way to derive rigorous error estimates, by solving appropriate local (``parametrix'') RH problems.

Expressing the $(x,t)$ dependence of the data for the associated RH problem (jump matrices and, if appropriate, residue conditions) in terms of $t$ (large parameter) and $\xi\coloneqq\frac{x}{t}$, it is natural to obtain large-$t$ asymptotics for a fixed $\xi$ or for $\xi$ varying in certain intervals, characterized by the same qualitative form of the asymptotic pattern and uniform error estimates for any compact subset of each interval. In terms of the model RH problem, such an interval is characterized by the same structure of the jump contour and jump condition, only some parameters (as functions of $\xi$) being varying. On the other hand, the interval's end points correspond to changes in this structure: either some parts of the contour shrink to single points or new parts emerge from specific points. Accordingly, uniform error estimates break down when approaching the interval's end points. Thus two other natural questions arise: do the asymptotics obtained for adjacent sectors (in the $(x,t)$ plane) match, and how can we describe the asymptotics when approaching (with $\xi$) the border of an asymptotic sector characterized by certain degeneration/restructuring of the RH problem input (contours and jumps).

It turns out that already for problems with ``zero boundary conditions'', i.e., when the solution is assumed to decay to $0$ as $\abs{x}\to\infty$, there exist various narrow transition regions in the $(x,t)$ plane, where the asymptotics has a qualitatively different form (related to a qualitatively different model RH problem) \cite{DVZ94}. So it is natural to expect transition regions also in the more complicated situation of ``non-zero boundary conditions'' described by \eqref{q0-limits} (for other types of non-zero boundary conditions, see \cites{BK14,BM16,BM17}).

Scenarios of long-time asymptotics of $q(x,t)$, where the $(x,t)$ half-plane is divided into sectors $\xi_1<\xi<\xi_2$ with qualitatively different asymptotics, are presented in \cite{BLS21}. That work was motivated by \cite{BV07}, where the asymptotic scenario was given corresponding to a particular range of parameters involved in \eqref{q0-limits}. In the ``shock case'' $B_1<B_2$ some scenarios include genus $3$ sectors where the leading term of the asymptotics of $q(x,t)$ is given in terms of theta functions attached to a hyperelliptic Riemann surface $M(\xi)$ of genus $3$. The asymptotic analysis in a genus $3$ sector is performed in \cite{BLS22}. The Riemann surface $M(\xi)$ is given by the equation (see \cite{BLS22}*{Section 3.2})
\begin{equation}  \label{M-def}
w^2=(k-E_1)(k-\bar E_1)(k-E_2)(k-\bar E_2)(k-\alpha)(k-\bar\alpha)(k-\beta)(k-\bar\beta)
\end{equation}
where $E_j=B_j+\ii A_j$, $B_1<B_2$, and $\alpha(\xi)$, $\beta(\xi)$ are critical points of the $g$-function $g(\xi,k)$, that is, zeros of $\dd g/\dd k$ (see \cite{BLS22}*{Section 3.3}). All branch points are distinct, in particular $\alpha(\xi)\neq\beta(\xi)$.

In this paper we consider a sector $\xi_1<\xi<\xi_2$ consisting of two genus $3$ sectors $\xi_1<\xi<\xi_0$ and $\xi_0<\xi<\xi_2$ separated by a point $\xi_0$ at which $\alpha(\xi)$ and $\beta(\xi)$ merge (in the symmetric case with $A_1=A_2$ and $B_1=-B_2$, we have $\xi_0=0$). Notice that a degeneration of genus $2$ (two-phase) solutions of \eqref{eq:nls} corresponding to the merging of two pairs of spectral points is considered in \cite{BG15}, where it is shown that it is possible to present the limiting behavior of the solution directly in terms of the solution of the degenerated RH problem, without resorting to the genus $2$ solution before degeneration.

The asymptotic theorem for the genus $3$ sector \cite{BLS22}*{Theorem 2.1} gives information on the asymptotics for $\xi\in(\xi_1+\delta,\xi_0-\delta)\cup(\xi_0+\delta,\xi_2-\delta)$. In fact, by choosing the $\epsilon$-disks centered at $\alpha$, $\bar\alpha$, $\beta$, $\bar\beta$ to have $\xi$-dependent radius $\epsilon=\epsilon(\xi)$ in the proof of the genus $3$ theorem, we see that the genus $3$ result is valid uniformly for all $\xi_0<\xi\leq\xi_2$ if we include the additional error term $\ord((t\epsilon^{3/2})^{-N})$. By choosing $\epsilon=\frac{\abs{\alpha-\beta}}{3}$, we infer that the genus $3$ theorem gives information on the solution whenever $t^{2/3}\abs{\alpha-\beta}\to\infty$. But the genus $3$ theorem does not give the asymptotics in a narrow region containing the line $\xi=\xi_0$.

Our goal here is to determine an asymptotic formula valid in a region containing the line $\xi=\xi_0$. Let $r(k)$ be the reflection coefficient (see \cite{BLS21}*{(2.36)}) defined by
\[
r(k)\coloneqq\frac{b^*(k)}{a(k)}
\]
where $a(k)$, $b(k)$ are the scattering coefficients (see \cite{BLS21}*{(2.15)}), and we write $f^*(k)\coloneqq\overline{f(\bar k)}$ for a complex-valued function $f(k)$. As in \cite{BLS21} and \cite{BLS22}, in order to avoid technicalities in the long-time analysis related to analytical approximations of spectral functions, we assume that $a(k)$ and $b(k)$ (and thus $r(k)$) can be analytically continued from the real line into the whole complex plane (see \eqref{eq:ic} below) and that $a(k)$ and $b(k)$ do not vanish for $\Im k\geq 0$. Define the complex-valued function $\tilde\nu(\xi)$ by
\begin{equation}  \label{eq:nutilde}
\tilde\nu(\xi)=\frac{1}{2\pi}\ln\left(1+r(\beta)r^*(\beta)\right),
\end{equation}
where $\beta\equiv\beta(\xi)$ and the branch of $\ln(1+rr^*)$ is fixed by requiring that $\ln(1+r(k)r^*(k))$ is a continuous function of $k\in\gamma_{(\mu,\beta)}$ whose value at $k=\mu$ is strictly positive ($\gamma_{(\mu,\beta)}$ is the contour from $\mu$ to $\beta$ in Figure~\ref{fig:jump-contour-3}). Let $\tilde\nu_0\coloneqq\tilde\nu(\xi_0)$ denote the value of $\tilde\nu$ at $\xi=\xi_0$. Since for $\abs{\Im\tilde\nu_0}\geq\frac{1}{2}$ we are not able to give any asymptotics we make the following assumption:

\begin{assumption*}
We suppose
\[ 
\Im\tilde\nu_0\in\left(-\frac{1}{2},\frac{1}{2}\right),
\]
i.e.\ $\arg\bigl(1+r(\beta(\xi_0))r^*(\beta(\xi_0))\bigr)\in(-\pi,\pi)$.
\end{assumption*}

\begin{remark*}
This assumption is at least satisfied in the symmetric case $A_1=A_2=A>0$ and $B_2=-B_1=B>0$ with $A>B$ provided $B$ is close to $A$. In that case $\xi_0=0$ and we have $\alpha(0)=\beta(0)=\ii\sqrt{A^2-B^2}$. As $B\uparrow A$ we have $\beta(0)\downarrow 0$ and $\abs{\Im\tilde\nu_0}\downarrow 0$. Indeed, if $\beta(\xi)$ is real $\tilde\nu(\xi)$ is also real. So, for $B$ close enough to $A$ we have $\Im\tilde\nu_0$ close to $0$, hence $\abs{\Im\tilde\nu_0}<1/2$.
\end{remark*}

\subsection{The transition zone}

We will consider the asymptotics in a sector $\D{S}$ defined by
\begin{equation}  \label{eq:wedge}
\D{S}\coloneqq\left\lbrace(x,t)\in\D{R}^2\,\Bigm\vert\,\xi\geq\xi_0,\ t\geq T,\ \abs{\alpha-\beta}<Mt^{-\frac{1+\abs{\Im\tilde\nu_0}}{2}-\delta}\right\rbrace,
\end{equation}
where the constants $\delta$, $M$, $T$ are strictly positive. The sector $\D{S}$ is characterized by the fact that the distance between $\alpha$ and $\beta$ shrinks faster than $t^{-\frac{1+\abs{\Im\tilde\nu_0}}{2}}$ as $t\to+\infty$. As $\xi\to\xi_0$, the branch points $\alpha(\xi)$ and $\beta(\xi)$ converge to a point $\alpha(\xi_0)=\beta(\xi_0)\in\D{C}^+$. Numerical computations suggest that $\abs{\alpha-\alpha(\xi_0)}\sim\abs{\xi-\xi_0}^{1/2}$ and $\abs{\beta-\beta(\xi_0)}\sim\abs{\xi-\xi_0}^{1/2}$ as $\xi\downarrow\xi_0$ (at least in the symmetric case). If this is correct, then
\begin{equation}  \label{eq:wedgebis}
\D{S}=\left\lbrace(x,t)\in\D{R}^2\,\Bigm\vert\,t\geq T,\ 0\leq\xi-\xi_0<Mt^{-\abs{\Im\tilde\nu_0}-1-2\delta}\right\rbrace. 
\end{equation}

\begin{remarks*}
The region $\D{S}$ is specified, somewhat implicitly (involving $\alpha(\xi)$ and $\beta(\xi)$), by \eqref{eq:wedge}. The more explicit description \eqref{eq:wedgebis} is based on the assumption that $\alpha(\xi)-\beta(\xi)\sim(\xi-\xi_0)^{1/2}$ as $\xi\downarrow\xi_0$. 

The region $\D{S}$ depends on $4$ parameters: $|\Im\tilde\nu_0|$, $\delta$, $M$, and $T$. The first one, $\tilde\nu_0$, reflects the influence (via the reflection coefficient) of the initial data on the Cauchy problem while the others are at our choice, $\delta$ being the most important one. A smaller choice of $\delta>0$ increases the size of the asymptotic sector $\D{S}$, but it also makes the asymptotic formula less precise.
\end{remarks*}

\subsection{Organization of the paper}

The main theorem is stated in Section~\ref{sec:main}. Our proof of this theorem is based on a steepest descent analysis of a RH problem which is described in Section~\ref{sec:prelim}. In Section~\ref{sec:prelim}, we also introduce some necessary notation. In Section~\ref{sec:transforms}, we implement a number of transformations of the RH problem that are required for the steepest descent analysis. In Section~\ref{sec:model}, we construct a global parametrix by solving a model RH problem in terms of theta functions associated to a genus $1$ Riemann surface. The global parametrix eventually gives rise to the leading term of $\ord(1)$ in the final asymptotic formula for $q(x,t)$. In Section~\ref{sec:locmod}, we construct local parametrices; in particular, we construct a local parametrix in a neighborhood of the two merging branch points $\alpha$ and $\beta$. This is achieved by relating the original RH problem to an exactly solvable RH problem whose solution is presented in the appendix. The local parametrices eventually give rise to the subleading terms in the final asymptotic formula for $q(x,t)$. The proof of the main theorem is finalized in Section~\ref{sec:final}.
\section{Main result}  \label{sec:main}

In the shock case $B_1<B_2$ we can assume $B_1=-1$, $B_2=1$, and $\phi_2=0$ (see \cite{BLS21}*{Section~2.2}). Suppose $q\colon\D{R}\times\lbrack 0,\infty)\to\D{C}$ is a smooth solution of \eqref{eq:nls} whose initial data $q_0(x)=q(x,0)$ satisfy
\begin{equation}  \label{eq:ic}
q_0(x)=
\begin{cases}
A_1\eul^{\ii\phi}\eul^{2\ii x},&x<-C,\\
A_2\eul^{-2\ii x},&x>C,
\end{cases}
\end{equation}
for some constants $C>0$, $A_1>0$, $A_2>0$, and $\phi\in\D{R}$. Let $E_1 =-1+\ii A_1$ and $E_2=1+\ii A_2$. Under conditions \eqref{eq:ic}, the  spectral functions $a(k)$ and $b(k)$ are entire functions, which, as we mentioned above, allows us to work directly with these functions thus avoiding analytical approximations, which would make the realization of the main ideas of the asymptotic analysis less transparent.

The asymptotics of $q$ in $\D{S}$ can be expressed in terms of quantities defined on the genus $1$ Riemann surface $\tilde M$ with branch cuts along $\Sigma_1$ and $\Sigma_2$, where $\Sigma_j\coloneqq\croch{\bar E_j,E_j}$, $j=1,2$. 

\begin{theorem}  \label{thm:main}
The asymptotics in the sector $\D{S}$ is given by
\begin{equation}  \label{eq:main}
q(x,t)=Q_0(\xi,t)+\frac{Q_1(\xi,t)}{\sqrt{t}}+\ord\left(F(\xi,t)^2+t\abs{\alpha-\beta}^2\right),\quad t\to+\infty,\ (x,t)\in\D{S}.
\end{equation}
\begin{enumerate}[\textbullet]
\item
The first term $Q_0$ is the leading order term. It is given by 
\[
Q_0(\xi,t)=\eul^{2\ii(tg^{(0)}(\xi)+\tilde h(\xi,\infty))}(A_1+A_2)\frac{\tilde\Theta(\tilde\varphi(\infty^+)+\tilde d)\tilde\Theta(\tilde\varphi(\infty^+)-\tilde v(\xi,t)-\tilde d)}{\tilde\Theta(\tilde\varphi(\infty^+)+\tilde v(\xi,t)+\tilde d)\tilde\Theta(\tilde\varphi(\infty^+)-\tilde d)}\,.
\]
\item
The coefficient $Q_1$ of the second term is given by
\[
Q_1(\xi,t)=2\ii\eul^{2\ii(tg^{(0)}(\xi)+\tilde h(\xi,\infty))}\left((T_{\mu}(x,t))_{12}+(T_{\beta}(x,t))_{12}-\overline{(T_{\beta}(x,t))_{21}}\right).
\]
\item
Regarding the last term the estimate is uniform with respect to $x$ and the function $F$ is given by
\begin{equation}  \label{eq:F}
F(\xi,t)\coloneqq t^{\abs{\Im\tilde\nu}}(\ln t)^{\abs{\Im\tilde\nu}+2}\left(t\abs{\alpha-\beta}^2\abs{\ln\abs{\alpha-\beta}}+t^{-\frac{1}{2}}\right).
\end{equation}
\end{enumerate}
$\tilde{\Theta}$ is the Riemann theta function associated with the genus $1$ Riemann surface $\tilde M$, the Abel map $\tilde\varphi$ is defined in \eqref{eq:tildephi}. The functions $\tilde{v}(t)\equiv\tilde{v}(\xi,t)$ and $\tilde h(k)\equiv\tilde h(\xi,k)$ are defined in \eqref{eq:vtilde} and \eqref{eq:htilde}, respectively. The constants $\tilde{\nu}\equiv\tilde{\nu}(\xi)$, $\tilde d\equiv\tilde d(\xi)$, and $g^{(0)}\equiv g^{(0)}(\xi)$ are defined in \eqref{eq:nutilde}, \eqref{eq:dtilde}, and \cite{BLS22}*{(3.28)}, respectively. The functions $T_{\mu}(x,t)$ and $T_{\beta}(x,t)$ are defined in \eqref{eq:Tmu} and \eqref{eq:Tbeta}, respectively.
\end{theorem}

To illustrate the statement of Theorem~\ref{thm:main} we make some comments. It will be useful to note that the functions $T_{\mu}(x,t)$ and $T_{\beta}(x,t)$ satisfy the uniform estimates 
\begin{equation} \label{eq:TmuTbeta}
\begin{array}{l}
\abs{T_{\mu}(x,t)}\leq C,\\[1mm]
\abs{T_{\beta}(x,t)}\leq Ct^{\abs{\Im\tilde\nu_0}},
\end{array}\ (x,t)\in\D{S}.
\end{equation}
All factors involved in \eqref{eq:Tmu} and \eqref{eq:Tbeta} are indeed bounded, with the exception of the factors $t^{\pm\frac{\ii\tilde\nu}{2}\sigma_3}$. Moreover, if $\abs{\beta(\xi)-\beta(\xi_0)}\sim\abs{\xi-\xi_0}^{1/2}$ as $\xi\downarrow\xi_0$ and \eqref{eq:wedgebis} are correct, then $\abs{\Im\tilde\nu}=\abs{\Im\tilde\nu_0}+\ord(t^{-\frac{1+\abs{\Im\tilde\nu_0}}{2}-\delta})$ which implies $t^{\abs{\Im\tilde\nu}}=\ord(t^{\abs{\Im\tilde\nu_0}})$.

\begin{comments}  \label{comms:asymptotics}
We compare the orders of the three terms in \eqref{eq:main}.
\begin{enumerate}[\textbullet]
\item
The first term $Q_0(\xi,t)$ is a bounded and oscillating term obtained
by solving the ``model'' RH problem (see Section~\ref{sec:model}). It dominates the other two terms (for any $\delta>0$) which decay because we assumed $\abs{\Im\tilde\nu_0}<1/2$.
\item
An important feature is that $\tilde\nu_0$ is complex-valued. This implies that the coefficient $Q_1$ can grow with $t$ since by \eqref{eq:TmuTbeta} it is of order $\ord(t^{\abs{\Im\tilde\nu_0}})$, while the second term itself $Q_1/\sqrt{t}$ is of order $\ord(t^{\abs{\Im\tilde\nu_0}-1/2})$ and decays.
\item
If $\delta>\delta_0\coloneqq\max\accol{\frac{1}{4}-\abs{\Im\tilde\nu_0},\frac{1}{8}-\frac{1}{4}\abs{\Im\tilde\nu_0}}$ and $c_0\coloneqq\frac{1}{2}-|\Im\tilde\nu_0|$, then the last term in \eqref{eq:main} is $\ord(t^{-c})$ for any $c<c_0$. Since the second term in \eqref{eq:main} can be written
\[
\frac{Q_1}{\sqrt{t}}=G(\xi,t)t^{-c_0},
\]
where $G(\xi,t)$ is a bounded, non-decaying, oscillating function, then in this case $Q_1/\sqrt{t}$ can be viewed as the subleading term. Moreover, for $\delta>\delta_1$ with $\delta_1\coloneqq\max\accol{\frac{1}{4}-\frac{1}{2}\abs{\Im\tilde\nu_0},\,\frac{1}{2}-\frac{3}{2}\abs{\Im\tilde\nu_0}}$, e.g., for $\delta>1/2$, the last term is
\[
\ord\bigl(t^{2\abs{\Im\tilde\nu_0}-1}(\ln t)^{2\abs{\Im\tilde\nu_0}+4}\bigr).
\]
\end{enumerate}
\end{comments}

On the line $\xi=\xi_0$, we can take $\delta$ arbitrarily large and Theorem~\ref{thm:main} reduces to the following. 

\begin{corollary} \label{cor:xi0}
The asymptotics on the line $\xi=\xi_0$ is given by
\begin{equation}   \label{eq:xi0}
q(x,t)=Q_0(\xi_0,t)+\frac{Q_1(\xi_0,t)}{\sqrt{t}}+\ord\left(t^{2\abs{\Im\tilde\nu_0}-1}(\ln t)^{2\abs{\Im\tilde\nu_0}+4}\right),\quad t\to+\infty,\ x=\xi_0t.
\end{equation}
\end{corollary}

On the other hand, keeping only the leading order term, Theorem~\ref{thm:main} reduces to the following.

\begin{corollary}[Leading order asymptotics]  \label{cor:lead}
There exists a $c>0$ such that
\[
q(x,t)=Q_0(\xi,t)+\ord(t^{-c}),\quad t\to+\infty,\ (x,t)\in\D{S},
\]
uniformly with respect to $x$.
\end{corollary}

\section{Preliminaries} \label{sec:prelim}

\subsection{Notations}

Let $E_j\coloneqq B_j+\ii A_j$ in the complex $k$-plane $\D{C}$. We denote by $\Sigma_j$, $j=1,2$ the vertical segment $\croch{\bar E_j,E_j}$ oriented upward. See Figure~\ref{fig:basic-contour}. 

Let $\D{C}^\pm=\accol{\pm\Im k>0}$ denote the open upper and lower halves of the complex plane. The Riemann sphere will be denoted by $\hat{\D{C}}=\D{C}\cup\accol{\infty}$. We write $\ln k$ for the logarithm with the principal branch, that is, $\ln k=\ln\abs{k}+\ii\arg k$ where $\arg k\in(-\pi,\pi\rbrack$. Unless specified otherwise, all complex powers will be defined using the principal branch, i.e., $z^{\alpha}=\eul^{\alpha\ln z}$. We let $f^*(k)\coloneqq\overline{f(\bar k)}$ denote the Schwarz conjugate of a complex-valued function $f(k)$.

Given an open subset $D\subset\hat{\D{C}}$ bounded by a piecewise smooth contour $\Sigma$, we let $\dot E_2(D)$ denote the Smirnoff class consisting of all functions $f(k)$ analytic in $D$ with the property that for each connected component $D_j$ of $D$ there exist curves $\accol{C_n}_1^{\infty}$ in $D_j$ such that the $C_n$ eventually surround each compact subset of $D_j$ and $\sup_{n\geq 1}\norm{f}_{L^2(C_n)}<\infty$. We let $E^{\infty}(D)$ denote the space of bounded analytic functions $D\to\D{C}$. RH problems in the paper are generally $2\times 2$ matrix-valued. They are formulated in the $L^2$-sense using Smirnoff classes (see \cites{Le17,Le18}): 
\begin{equation}   \label{rhp}
\begin{cases}
m\in I+\dot E^2(\D{C}\setminus\Sigma),&\\
m_+(k)=m_-(k)J(k)&\text{for a.e. }k\in\Sigma,
\end{cases}
\end{equation}
where $m_+$ and $m_-$ denote the boundary values of $m$ from the left and right sides of the contour $\Sigma$. We let $\sigma_1\coloneqq\left(\begin{smallmatrix}0&1\\1&0\end{smallmatrix}\right)$ and $\sigma_3\coloneqq\left(\begin{smallmatrix}1&0\\0&-1\end{smallmatrix}\right)$ denote the first and third Pauli matrices.
\subsection{Set-up}

We have \cite{BLS22}*{Proposition 3.1}
\begin{equation}
q(x,t)=2\ii\lim_{k\to\infty}k\left(\hat m(x,t,k)\right)_{12},\quad x\in\D{R},\ t\in\lbrack 0,\infty),
\end{equation}
where $\hat m(x,t,\,\cdot\,)$ is the unique solution of the basic RH problem (see \cite{BLS22}*{(3.12)})
\begin{equation}   \label{eq:rhp0}
\begin{cases}
\hat m(x,t,\,\cdot\,)\in I+\dot E^2\left(\D{C}\setminus(\D{R}\cup\Sigma_1\cup\Sigma_2)\right),&\\
\hat m_+(x,t,k)=\hat m_-(x,t,k)\hat J(x,t,k)&\text{for a.e. }k\in\D{R}\cup\Sigma_1\cup\Sigma_2,
\end{cases}
\end{equation}
with $\hat J$ defined as in \cite{BLS22}*{(3.11)}:
\begin{subequations}  \label{basic-jump}
\begin{equation}  \label{jump}
\hat J(x,t,k)\coloneqq\eul^{-\ii t\theta(k)\sigma_3}\hat J_0(k)\eul^{\ii t\theta(k)\sigma_3},\quad k\in\Sigma\coloneqq\D{R}\cup\Sigma_1\cup\Sigma_2,
\end{equation}
where the phase $\theta(k)$ is
\begin{equation}   \label{phase}
\theta(k)\equiv\theta(\xi,k)\coloneqq 2k^2+\xi k,\qquad\xi\coloneqq\frac{x}{t}\,,
\end{equation}
and 
\begin{equation}  \label{jump0}
\hat J_0(k)=\begin{cases}
\begin{pmatrix}1&\hat r^*\\
0&1\end{pmatrix}\begin{pmatrix}1&0\\
\hat r&1\end{pmatrix}=\begin{pmatrix}1+\hat r\hat r^*&\hat r^*\\
\hat r&1\end{pmatrix},&k\in\D{R},\\
\begin{pmatrix}-\ii&0\\\frac{\ii\eul^{-\ii\phi}}{\hat a_+\hat a_-}&\ii\end{pmatrix},&k\in\Sigma_1\cap\D{C}^+,\\
\begin{pmatrix}\frac{\hat a_-}{\hat a_+}&\ii\nu_1^2\\0&\frac{\hat a_+}{\hat a_-}\end{pmatrix},&k\in\Sigma_2\cap\D{C}^+,\\
\begin{pmatrix}-\ii&\frac{\ii\eul^{\ii\phi}}{\hat a_+^*\hat a_-^*}\\0&\ii\end{pmatrix},&k\in\Sigma_1\cap\D{C}^-,\\
\begin{pmatrix}\frac{\hat a_+^*}{\hat a_-^*}&0\\\ii\nu_1^{-2}&-\frac{\hat a_-^*}{\hat a_+^*}\end{pmatrix},&k\in\Sigma_2\cap\D{C}^-.
\end{cases}
\end{equation}
\end{subequations}
Recall \cite{BLS22}*{(3.10)} that $\hat a$, $\hat b$, and $\hat r$ are slight modifications of the scattering data $a$, $b$, and $r$:
\begin{equation}  \label{hat-abr}
\hat a\coloneqq a\nu_1,\quad\hat b\coloneqq b\nu_1,\quad\hat r\coloneqq\frac{\hat b^*}{\hat a}=r\nu_1^{-2},
\end{equation}
where $\nu_1(k)\coloneqq\bigl(\frac{k-E_1}{k-\bar E_1}\bigr)^{\frac{1}{4}}$ with $\nu_1(\infty)=1$, as in \cite{BLS22}*{(3.7)}. Note that $\hat a\hat a^*=aa^*$, $\hat b\hat b^*=bb^*$, and $\hat r\hat r^*=rr^*$.

\begin{figure}[ht]
\centering\includegraphics[scale=.8]{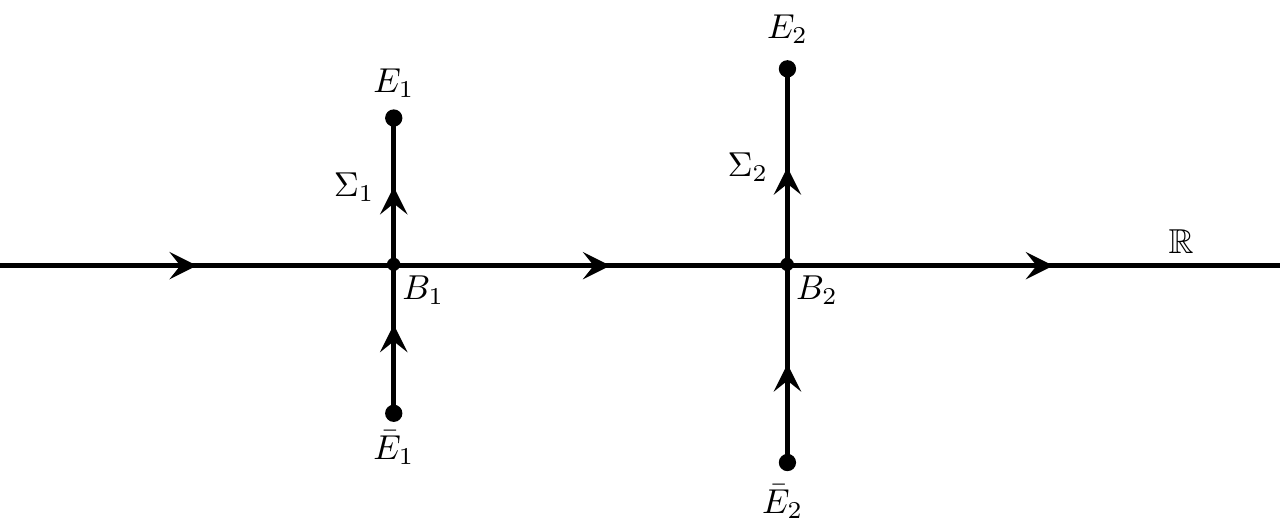}
\caption{The contour $\Sigma=\D{R}\cup\Sigma_1\cup\Sigma_2$ for the basic RH problem.} 
\label{fig:basic-contour}
\end{figure}

Let $\tilde M$ denote the genus $1$ Riemann surface with branch cuts along $\Sigma_1$ and $\Sigma_2$. We think of $\tilde M$ as the limit of the genus $3$ surface $M(\xi)$ (see \eqref{M-def}) as $\xi\to\xi_0$. We introduce a canonical homology basis $\accol{a_1,b_1}$ on $\tilde{M}$ by letting the cycles $a_1$ and $b_1$ be as in \cite{BLS22}*{Figure 4}. We let $\tilde\varphi\colon\tilde M\to\D{C}$ be the limit of the Abel map $\varphi$ (see \cite{BLS22}*{(3.21)}) as $\xi\to\xi_0$, i.e.,
\begin{equation}  \label{eq:tildephi}
\tilde\varphi(P)=\int_{\bar E_2}^P\tilde\zeta,\quad P\in\tilde M,
\end{equation}
where $\accol{\tilde\zeta}$ is the normalized basis of $\C{H}^1(\tilde M)$ dual to the canonical basis $\accol{a_1,b_1}$, i.e., $\tilde\zeta$ is a holomorphic differential form such that
\[
\int_{a_1}\tilde\zeta=1.
\]
Let $\tau$ denote the period $\int_{b_1}\tilde\zeta$ and let $\tilde\Theta$ denote the associated theta function which is defined by
\begin{equation} \label{eq:tildetheta}
\tilde\Theta(z)\coloneqq\sum_{n\in\D{Z}}\eul^{2\pi\ii(\frac{1}{2}n^2\tau+nz)},\quad z\in\D{C}.
\end{equation}
As in \cite{BLS22}*{(3.24)}, we let $g(k)\equiv g(\xi,k)$ denote the $g$-function on $M(\xi)$ with derivative
\[
\frac{\dd g}{\dd k}=\frac{4(k-\mu)(k-\alpha)(k-\bar\alpha)(k-\beta)(k-\bar\beta)}{w(k)},
\]
where $\mu\equiv\mu(\xi)$ is a real number and
\[
w(k)\coloneqq\sqrt{(k-E_1)(k-\bar E_1)(k-E_2)(k-\bar E_2)(k-\alpha)(k-\bar\alpha)(k-\beta)(k-\bar\beta)}.
\]
We define $\tilde w(k)$ by
\begin{equation}  \label{eq:wtilde}
\tilde w(k)\coloneqq\sqrt{(k-E_1)(k-\bar E_1)(k-E_2)(k-\bar E_2)},\quad k\in\D{C}\setminus(\Sigma_1\cup\Sigma_2),
\end{equation}
where the branch of the square root is such that $\tilde w(k)>0$ for $k\gg 0$.

Throughout the paper $C,\,c>0$ will denote generic constants independent of $k$ and of $(x,t)\in\D{S}$, which may change within a computation.

\section{Transformations of the RH problem}  \label{sec:transforms}

In order to determine the long-time asymptotics of the solution $\hat m$ of the RH problem \eqref{eq:rhp0}, we will perform a series of transformations of the RH problem. More precisely, starting with $\hat m$, we define functions $\hat m^{(j)}(x,t,k)$, $j=1,\dots,6$, such that each $\hat m^{(j)}$ satisfies an RH problem which is equivalent to the original RH problem \eqref{eq:rhp0}. The RH problem for $\hat m^{(j)}$ has the form
\begin{equation}  \label{eq:mj}
\begin{cases}
\hat m^{(j)}(x,t,\,\cdot\,)\in I+\dot E^2(\D{C}\setminus\Sigma^{(j)}),&\\
\hat m_+^{(j)}(x,t,k)=\hat m_-^{(j)}(x,t,k)\hat v^{(j)}(x,t,k)&\text{for a.e. }k\in\Sigma^{(j)},
\end{cases}
\end{equation}
where the contours $\Sigma^{(j)}$ and jump matrices $\hat v^{(j)}$ are specified below.

At each stage of the transformations, $\hat m^{(j)}$ and $\hat v^{(j)}$ will satisfy the symmetries
\begin{equation}  \label{eq:vjsym}
\hat v^{(j)}(x,t,k)=\begin{cases}
\sigma_3\sigma_1\overline{\hat v^{(j)}(x,t,\bar k)}\sigma_1\sigma_3,&k\in\Sigma^{(j)}\setminus\D{R},\\
\sigma_3\sigma_1\overline{\hat v^{(j)}(x,t,\bar k)}^{-1}\sigma_1\sigma_3,&k\in\Sigma^{(j)}\cap\D{R},
\end{cases}
\quad j=1,\dots,6,
\end{equation}
and
\begin{equation}   \label{eq:mjsym}
\hat m^{(j)}(x,t,k)=\sigma_3\sigma_1\overline{\hat m^{(j)}(x,t,\bar k)}\sigma_1\sigma_3,\quad k\in\D{C}\setminus\Sigma^{(j)},\quad j=1,\dots,6.
\end{equation}
\subsection{First three transformations}

The first three transformations are the same as in \cite{BLS22}*{Sections 4.1-4.3}. This leads to a function $\hat m^{(3)}$ which satisfies the RH problem \eqref{eq:mj} for $j=3$ with jump contour $\Sigma^{(3)}$ displayed in Figure~\ref{fig:jump-contour-3} 
\begin{figure}[ht]
\centering\includegraphics[scale=.7]{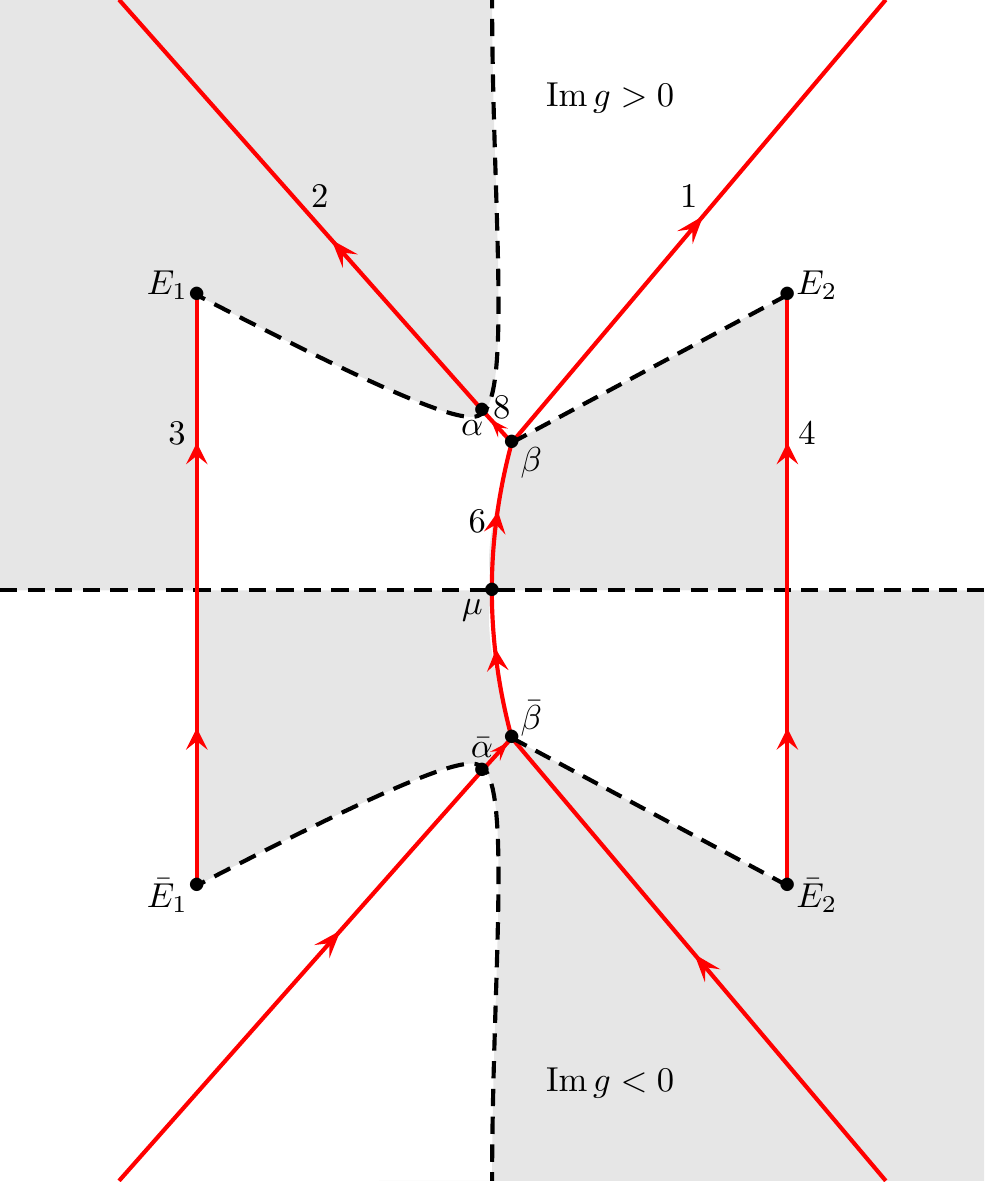}
\caption{The jump contour $\Sigma^{(3)}$ (red with arrows indicating orientation) together with the set where $\Im g=0$ (dashed black). The region where $\Im g<0$ is shaded and the region where $\Im g>0$ is white.} 
\label{fig:jump-contour-3}
\end{figure}
and jump matrix $\hat v^{(3)}$ given by the following formulas \cite{BLS22}*{Section 4.3} ($\hat v_l^{(3)}$ denotes the restriction of $\hat v^{(3)}$ to the contour labeled by $l$ in Figure~\ref{fig:jump-contour-3}):
\begin{alignat*}{2}
&\hat v_1^{(3)}=
\begin{pmatrix}
1&0\\
\hat r\delta^{-2}\eul^{2\ii tg}&1\end{pmatrix},&&
\hat v_2^{(3)}=
\begin{pmatrix}
1&-\hat a\hat b\delta^2\eul^{-2\ii tg}\\
0&1\end{pmatrix},\\
&\hat v_3^{(3)}=
\begin{pmatrix}
0&\ii\hat a_+\hat a_-\eul^{\ii\phi}\delta^2\eul^{-\ii t(g_++g_-)}\\
\frac{\ii\eul^{-\ii\phi}}{\hat a_+\hat a_-}\delta^{-2}\eul^{\ii t(g_++g_-)}&0\end{pmatrix},&&\\
&\hat v_4^{(3)}=
\begin{pmatrix}
0&\ii\nu_1^2\delta^2\eul^{-\ii t(g_++g_-)}\\
\ii\nu_1^{-2}\delta^{-2}\eul^{\ii t(g_++g_-)}&0\end{pmatrix},&&\\
&\hat v_6^{(3)}=
\eul^{-\ii tg_-\sigma_3}\begin{pmatrix}
1&-\hat a\hat b\delta^2\\
\hat r\delta^{-2}&\hat a\hat a^*\end{pmatrix}\eul^{\ii tg_+\sigma_3},&&
\hat v_8^{(3)}=
\eul^{-\ii tg_-\sigma_3}\begin{pmatrix}
1&-\hat a\hat b\delta^2\\
0&1\end{pmatrix}\eul^{\ii tg_+\sigma_3},
\end{alignat*}
and extended to the lower half-plane by means of the symmetry \eqref{eq:vjsym}. Recall \cite{BLS22}*{Section 4.2} that the complex-valued function $\delta(k)\equiv\delta(\xi,k)$ is defined by
\[
\delta(k)\coloneqq\eul^{\frac{1}{2\pi\ii}\int_{-\infty}^{\mu}\frac{\ln(1+\abs{r(s)}^2)}{s-k}\,\dd s},\quad k\in\D{C}\setminus(-\infty,\mu\rbrack,
\]
where $\mu$ is the only real critical point of the $g$-function $g(k)$.

\begin{remark*}
$(\hat v_2^{(3)})_{12}$ has the opposite sign compared to that in \cite{BLS22} because the orientation of the contour $2$ is opposite. The path labeled by $8$ in Figure~\ref{fig:jump-contour-3} is labeled by $5$ in \cite{BLS22}*{Figure 10}.
\end{remark*}

\subsection{Fourth transformation}

The fourth transformation is also similar to the analogous transformation in \cite{BLS22}*{Section 4.4}. However, since $\alpha$ and $\beta$ are now merging, we do not need to factor the jump across $\gamma_{(\beta,\alpha)}$, where $\gamma_{(\beta,\alpha)}$ is the contour from $\beta$ to $\alpha$ in Figure~\ref{fig:jump-contour-3}. As in \cite{BLS22}*{Figure~11}, we let $V_3$ and $V_4$ be open sets that form a lens around $\gamma_{(\mu,\beta)}$. Thus we define $\hat m^{(4)}$ for $k$ in the upper-half plane by
\[
\hat m^{(4)}\coloneqq\hat m^{(3)}\times
\begin{cases}
\begin{pmatrix}
1&0\\-\frac{\hat b^*}{\hat a^2\hat a^*}\delta^{-2}\eul^{2\ii tg}&1\end{pmatrix},&k\in V_3,\\
\begin{pmatrix}
1&-\frac{\hat b}{\hat a^*}\delta^2\eul^{-2\ii tg}\\0&1\end{pmatrix},&k\in V_4,\\
I,&\text{elsewhere in }\D{C}^+,
\end{cases}
\]
and extend the definition to the lower half-plane by means of the symmetry \eqref{eq:mjsym}. We note that the function $\delta(\xi,k)$ obeys the uniform bound (see \cite{BLS22}*{Lemma 4.1})
\[
\abs{\delta(\xi,k)^{\pm1}}\leq C,\quad k\in\D{C}\setminus(-\infty,\mu\rbrack,\ (x,t)\in\D{S}.
\]

Then $\hat m$ satisfies the RH problem \eqref{eq:rhp0} iff $\hat m^{(4)}$ satisfies the RH problem \eqref{eq:mj} for $j=4$, where $\Sigma^{(4)}$ is the jump contour displayed in Figure~\ref{fig:jump-contour-4} and the jump matrix $\hat v^{(4)}$ is given by ($\hat v_l^{(4)}$ denotes the restriction of $\hat v^{(4)}$ to the contour labeled by $l$ in Figure~\ref{fig:jump-contour-4}):
\begin{alignat*}{2}
&\hat v_1^{(4)}=\begin{pmatrix}1&0\\\frac{\hat b^*}{\hat a}\delta^{-2}\eul^{2\ii tg}&1\end{pmatrix},&&\hat v_2^{(4)}=\begin{pmatrix}1&-\hat a\hat b\delta^2\eul^{-2\ii tg}\\0&1\end{pmatrix},\\
&\hat v_3^{(4)}=\begin{pmatrix}0&\ii\hat a_+\hat a_-\eul^{\ii\phi}\delta^2\eul^{-\ii t(g_++g_-)}\\\frac{\ii\eul^{-\ii\phi}}{\hat a_+\hat a_-}\delta^{-2}\eul^{\ii t(g_++g_-)}&0\end{pmatrix},&&\\
&\hat v_4^{(4)}=\begin{pmatrix}0&\ii\nu_1^2\delta^2\eul^{-\ii t(g_++g_-)}\\\ii\nu_1^{-2}\delta^{-2}\eul^{\ii t(g_++g_-)}&0\end{pmatrix},&&\hat v_5^{(4)}=\begin{pmatrix}1&0\\\frac{\hat b^*}{\hat a^2\hat a^*}\delta^{-2}\eul^{2\ii tg}&1\end{pmatrix},\\
&\hat v_6^{(4)}=\begin{pmatrix}\frac{\eul^{\ii t(g_+-g_-)}}{\hat a\hat a^*}&0\\0&\hat a\hat a^*\eul^{-\ii t(g_+-g_-)}\end{pmatrix},&&\hat v_7^{(4)}=\begin{pmatrix}1&-\frac{\hat b}{\hat a^*}\delta^2\eul^{-2\ii tg}\\0&1\end{pmatrix},\\
&\hat v_8^{(4)}=\eul^{-\ii tg_-\sigma_3}\begin{pmatrix}1&-\hat a\hat b\delta^2\\0&1\end{pmatrix}\eul^{\ii tg_+\sigma_3},&&
\end{alignat*}
and is extended to the part of $\Sigma^{(4)}$ that lies in the lower half-plane by means of the symmetry \eqref{eq:vjsym}.
\begin{figure}[ht]
\centering\includegraphics[scale=.7]{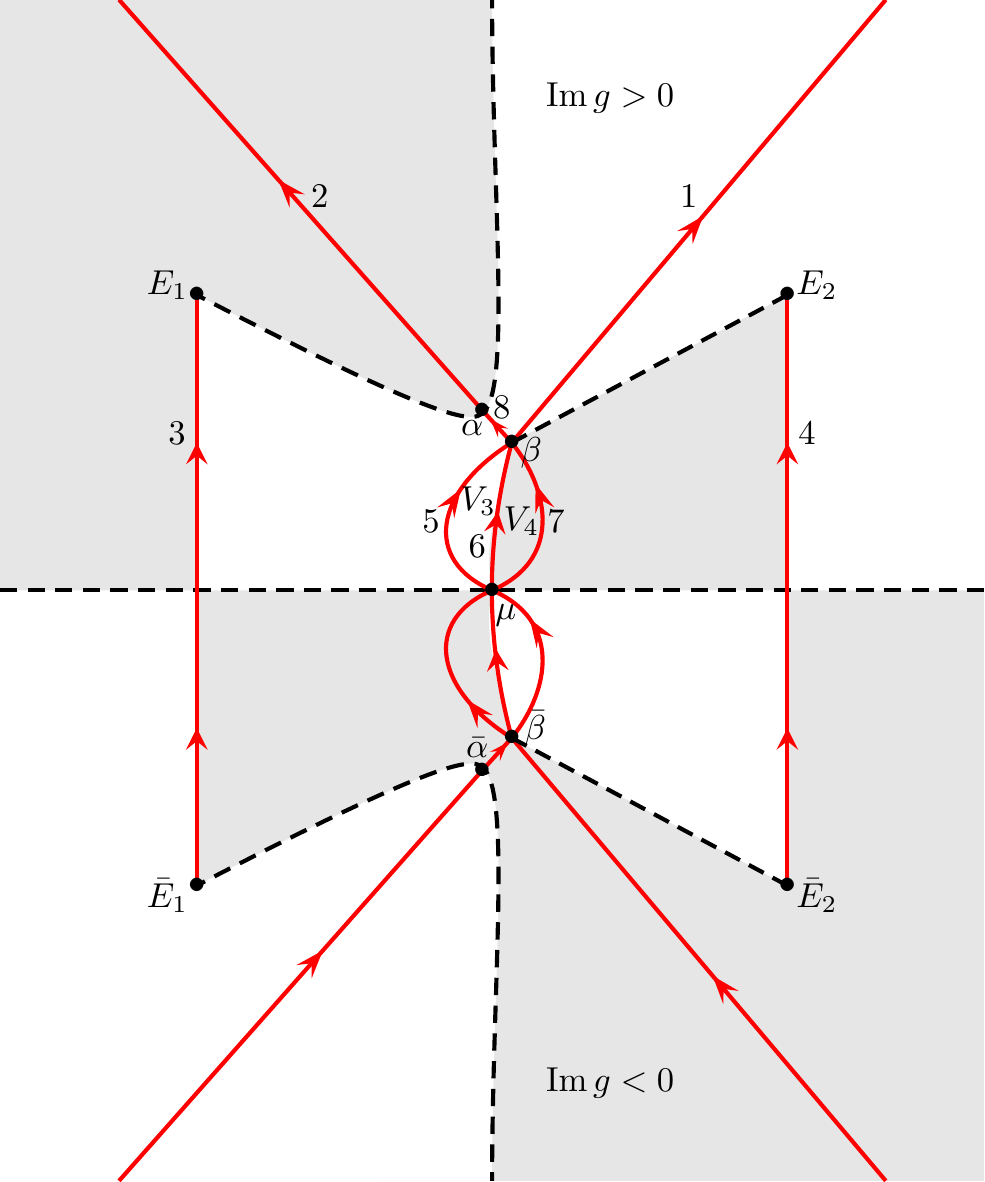}
\caption{The jump contour $\Sigma^{(4)}$ (red with arrows indicating orientation) together with the set where $\Im g=0$ (dashed black). The region where $\Im g<0$ is shaded and the region where $\Im g>0$ is white.} 
\label{fig:jump-contour-4}
\end{figure}

\begin{remark*}
The paths labeled by $5$, $6$, $7$, and $8$ in Figure~\ref{fig:jump-contour-4} are labeled by $12$, $11$, $10$ in \cite{BLS22}*{Figure 12}, and $5$ in \cite{BLS22}*{Figure 10}. Moreover, $(\hat v_2^{(4)})_{12}$ has the opposite sign compared to that in \cite{BLS22}.
\end{remark*}

\subsection{Fifth transformation}

The purpose of the fifth transformation is to remove the factors $\hat a\hat a^*=aa^*$ from the jump across $\gamma_{(\bar\beta,\beta)}$, i.e., from $v_6^{(4)}$ in the upper half-plane.

As we did with $\delta$ in \cite{BLS22}*{Section 4.2}, we introduce a complex-valued function $\tilde\delta(k)\equiv\tilde\delta(\xi,k)$ by
\begin{equation}  \label{eq:deltatilde}
\tilde\delta(\xi,k)=\eul^{-\frac{1}{2\pi\ii}\int_{\gamma_{(\mu,\beta)}}\frac{\ln(a(s)a^*(s))}{s-k}\dd s}\eul^{\frac{1}{2\pi\ii}\int_{\gamma_{(\bar\beta,\mu)}}\frac{\ln(a(s)a^*(s))}{s-k}\dd s},\quad k\in\D{C}\setminus\gamma_{(\bar\beta,\beta)},\ (x,t)\in\D{S},
\end{equation}
where the branch of the logarithm is chosen so that $\ln(a(s)a^*(s))$ is continuous on each contour and strictly positive for $s=\mu$. Since $a(k)$ has no zeros or poles, the function $aa^*$ is nonzero and finite everywhere on the contours. The function $\tilde\delta$ is in general singular at $\beta$ and $\bar\beta$. Hence we cut out small neighborhoods of these points. Thus let
\begin{equation}  \label{eq:R}
R\equiv R(\xi,t)\coloneqq c(\xi)t^{-1/2},
\end{equation}
where $c(\xi)>c>0$ is independent of $t$ (we will fix $c(\xi)$ below). 

For $r>0$ and $z\in\D{C}$, we let $D_r(z)$ denote the open disk of radius $r$ centered at $z$. Since $\abs{\alpha-\beta}=\ord(t^{-1/2-\delta})$, by increasing $T$ in the definition \eqref{eq:wedge} of $\D{S}$ if necessary, we may assume that $\alpha\in D_{R/2}(\beta)$ for all $(x,t)\in\D{S}$. 

The following lemma shows that $\tilde\delta$ is nonsingular in the complement of $D_R(\beta)\cup D_R(\bar\beta)$.

\begin{lemma} \label{lem:deltatilde}
For each $(x,t)\in\D{S}$, the function $\tilde\delta(\xi,k)$ has the following properties:
\begin{enumerate}[\rm(a)]
\item
$\tilde\delta(\xi,k)$ and $\tilde\delta(\xi,k)^{-1}$ are bounded and analytic functions of $k\in\D{C}\setminus(D_R(\beta)\cup D_R(\bar\beta)\cup\gamma_{(\bar\beta,\beta)})$.
\item
$\tilde\delta(\xi,k)$ obeys the symmetry
\[
\tilde\delta=(\tilde\delta^*)^{-1},\quad k\in\D{C}\setminus\gamma_{(\bar\beta,\beta)}.
\]
\item
Across $\gamma_{(\bar\beta,\beta)}$, $\tilde\delta(\xi,k)$ satisfies the jump condition
\begin{equation}  \label{eq:vdeltatilde}
\tilde\delta_+=\tilde\delta_-\times\begin{cases}
\frac{1}{aa^*},&k\in\gamma_{(\mu,\beta)},\\
aa^*,&k\in\gamma_{(\bar\beta,\mu)}.
\end{cases}
\end{equation}
\end{enumerate}
\end{lemma}

\begin{proof}
Since $\ln(\abs{a(\mu)}^2)>0$, \cite{BLS22}*{Lemma C.1} shows that $\tilde\delta(\xi,k)$ is bounded as $k$ approaches $\mu$. The other properties follow easily from the definition \eqref{eq:deltatilde}. We can indeed write $\tilde\delta=\eul^{-f_{\beta}+f_{\bar\beta}}$ with
\[
f_{\beta}(k)\coloneqq\frac{1}{2\pi\ii}\int_{\gamma_{(\mu,\beta)}}\frac{\ln(a(s)a^*(s))}{s-k}\dd s,\qquad f_{\bar\beta}(k)\coloneqq\frac{1}{2\pi\ii}\int_{\gamma_{(\bar\beta,\mu)}}\frac{\ln(a(s)a^*(s))}{s-k}\dd s,
\]
and $f_{\beta}=f_{\bar\beta}^*$.
\end{proof}

Let $\epsilon\in(0,1/2)$ be such that the $\xi$-dependent disks 
\begin{equation} \label{eq:threedisks}
D_{\epsilon}(\beta),\quad D_{\epsilon}(\mu),\quad D_{\epsilon}(\bar\beta)
\end{equation}
are disjoint from each other and from the cuts $\Sigma_1$ and $\Sigma_2$ for all $(x,t)\in\D{S}$. Increasing $T$ in \eqref{eq:wedge} if necessary, we may assume that $R(\xi,t)<\epsilon/2$ for all $(x,t)\in\D{S}$.

\begin{figure}[ht]
\centering\includegraphics[scale=.7]{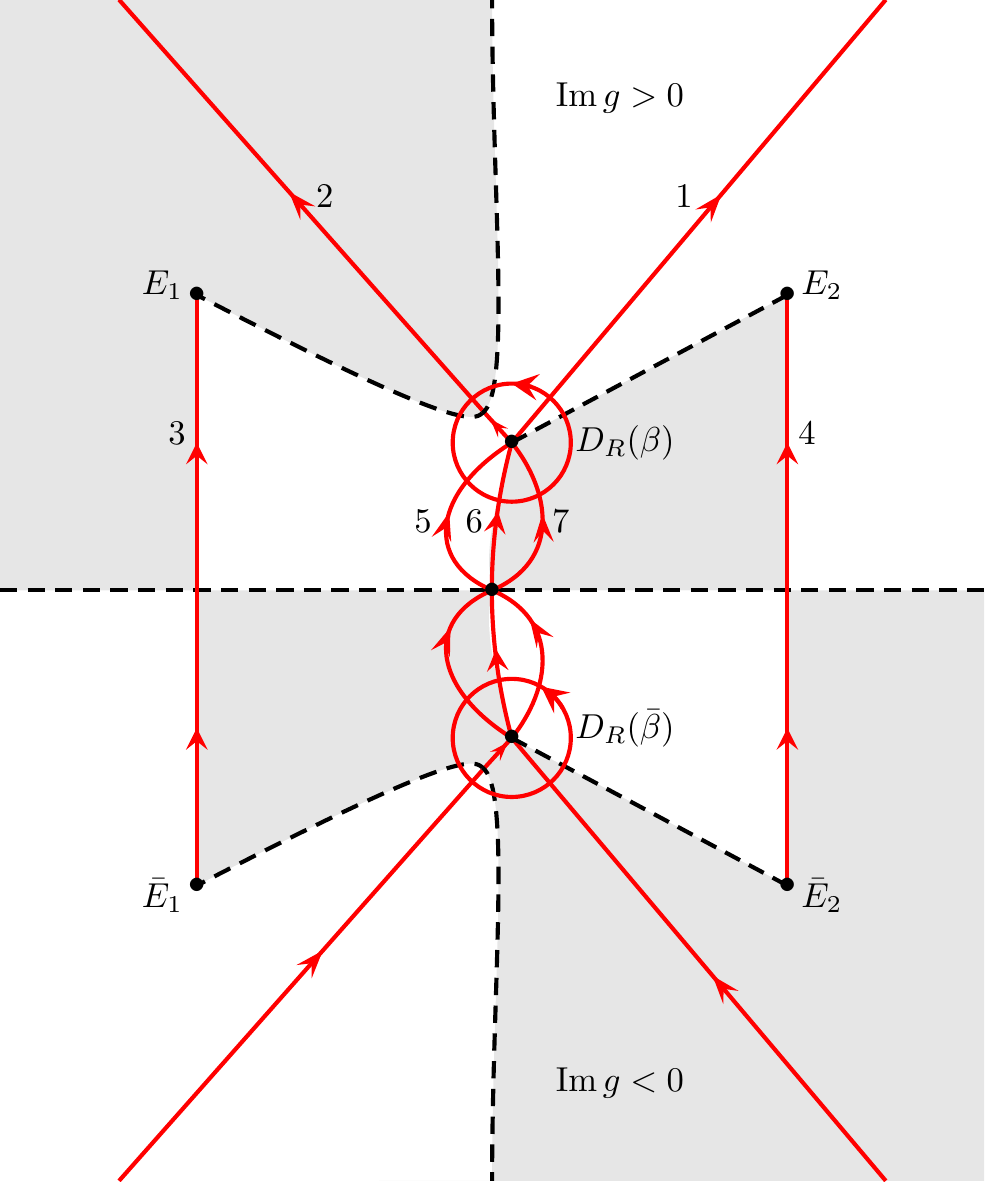}
\caption{The jump contour $\Sigma^{(5)}=\Sigma^{(6)}$ (red with arrows indicating orientation) together with the set where $\Im g=0$ (dashed black). The region where $\Im g<0$ is shaded and the region where $\Im g>0$ is white.} 
\label{fig:jump-contour-6}
\end{figure}

We define $\hat m^{(5)}\equiv\hat m^{(5)}(x,t,k)$ by
\[
\hat m^{(5)}=\hat m^{(4)}\times\begin{cases}
\tilde\delta^{-\sigma_3},&k\in\D{C}\setminus(D_R(\beta)\cup D_R(\bar\beta)),\\
I,&k\in D_R(\beta)\cup D_R(\bar\beta).
\end{cases}
\]
We define the complex-valued function $\delta_2(k)\equiv\delta_2(\xi,k)$ by
\[
\delta_2\coloneqq\delta\tilde\delta.
\]
Using Lemma~\ref{lem:deltatilde} we see that $\hat m$ satisfies the RH problem \eqref{eq:rhp0} iff $\hat m^{(5)}$ satisfies the RH problem \eqref{eq:mj} with $j=5$, where $\Sigma^{(5)}\coloneqq\Sigma^{(4)}\cup\partial D_R(\beta)\cup\partial D_R(\bar\beta)$ and the jump matrix $\hat v^{(5)}$ is given in $\D{C}^+\setminus\overline{D_R(\beta)}$ (see Figure~\ref{fig:jump-contour-6}) by
\begin{alignat*}{2}
&\hat v_1^{(5)}=\begin{pmatrix}1&0\\\frac{\hat b^*}{\hat a}\delta_2^{-2}\eul^{2\ii tg}&1\end{pmatrix},&\quad&\hat v_2^{(5)}=\begin{pmatrix}1&-\hat a\hat b\delta_2^2\eul^{-2\ii tg}\\0&1\end{pmatrix},\\
&\hat v_3^{(5)}=\begin{pmatrix}0&\ii\hat a_+\hat a_-\eul^{\ii\phi}\delta_2^2\eul^{-\ii t(g_++g_-)}\\\frac{\ii\eul^{-\ii\phi}}{\hat a_+\hat a_-}\delta_2^{-2}\eul^{\ii t(g_++g_-)}&0\end{pmatrix},&&\\
&\hat v_4^{(5)}=\begin{pmatrix}0&\ii\nu_1^2\delta_2^2\eul^{-\ii t(g_++g_-)}\\\ii\nu_1^{-2}\delta_2^{-2}\eul^{\ii t(g_++g_-)}&0\end{pmatrix},&&\hat v_5^{(5)}=\begin{pmatrix}1&0\\\frac{\hat b^*}{\hat a^2\hat a^*}\delta_2^{-2}\eul^{2\ii tg}&1\end{pmatrix},\\
&\hat v_6^{(5)}=\begin{pmatrix}\eul^{\ii t(g_+-g_-)}&0\\0&\eul^{-\ii t(g_+-g_-)}\end{pmatrix},&&\hat v_7^{(5)}=\begin{pmatrix}1&-\frac{\hat b}{\hat a^*}\delta_2^2\eul^{-2\ii tg}\\0&1\end{pmatrix},
\end{alignat*}
where $\hat v_l^{(5)}$ denotes the restriction of $\hat v^{(5)}$ to the contour labeled by $l$ in Figure~\ref{fig:jump-contour-6}, outside of $\overline{D_R(\beta)}$.

In order to give an expression for $\hat v^{(5)}$ on the part of the contour that lies in $\overline{D_R(\beta)}$, we let $\C{Y}\coloneqq\Sigma^{(5)}\cap D_{\epsilon}(\beta)$ denote the restriction of $\Sigma^{(5)}$ to $D_{\epsilon}(\beta)$ and write $\C{Y}=\cup_{l=1}^{16}\C{Y}_l$ where the curves $\accol{\C{Y}_l}_1^{16}$ are as in Figure \ref{fig:jump-contour-Y}. 
\begin{figure}[ht]
\centering\includegraphics[scale=.8]{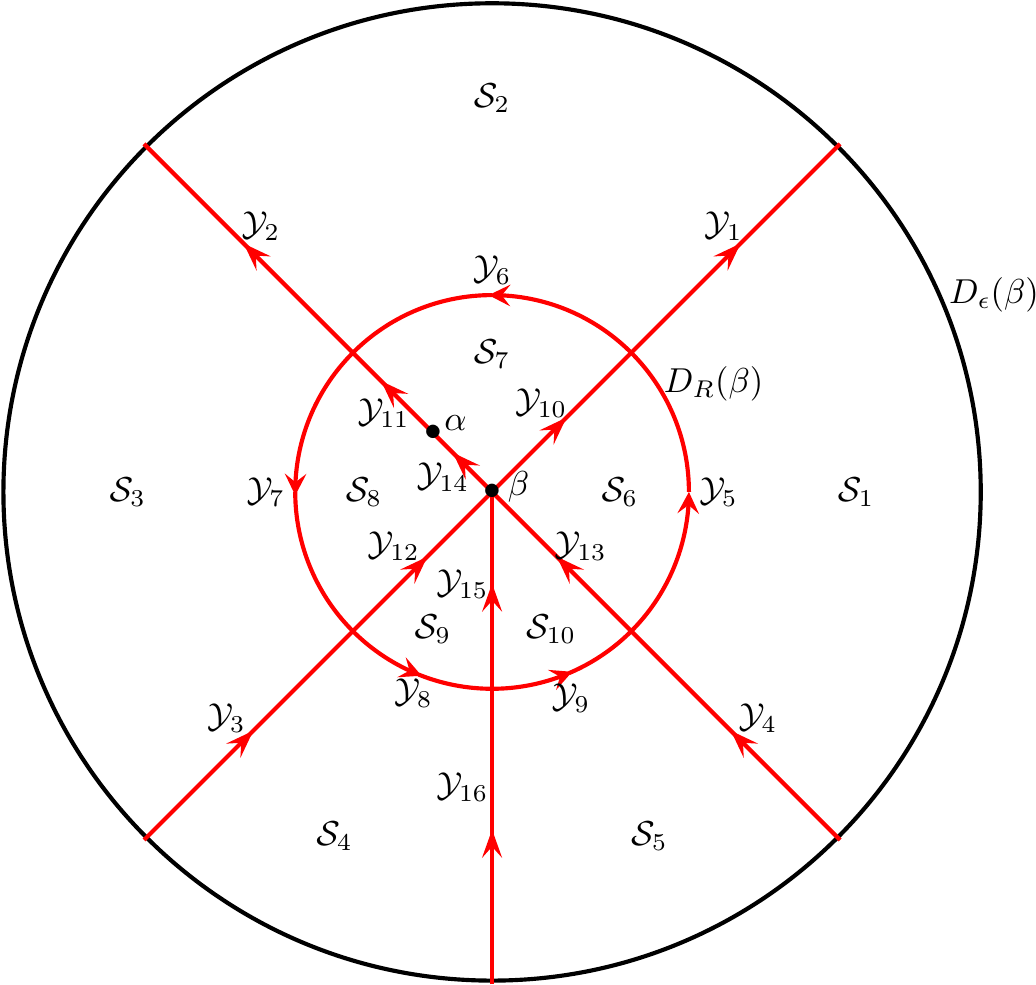}
\caption{The contour $\C{Y}=\Sigma^{(5)}\cap D_{\epsilon}(\beta)=\cup_{j=1}^{16}\C{Y}_j$ in the disk of radius $\epsilon$ centered at $\beta$ and the domains $\accol{S_j}_1^{10}$.} 
\label{fig:jump-contour-Y}
\end{figure}

Then $\hat v^{(5)}$ is given in $D_{\epsilon}(\beta)$ by
\begin{alignat*}{2}
&\hat v_{\C{Y}_1}^{(5)}=\begin{pmatrix}1&0\\\frac{\hat b^*}{\hat a}\delta_2^{-2}\eul^{2\ii tg}&1\end{pmatrix},&\qquad&\hat v_{\C{Y}_2}^{(5)}=\begin{pmatrix}1&-\hat a\hat b\delta_2^2\eul^{-2\ii tg}\\0&1\end{pmatrix},\\
&\hat v_{\C{Y}_3}^{(5)}=\begin{pmatrix}1&0\\\frac{\hat b^*}{\hat a^2\hat a^*}\delta_2^{-2}\eul^{2\ii tg}&1\end{pmatrix},&&\hat v_{\C{Y}_4}^{(5)}=\begin{pmatrix}1&-\frac{\hat b}{\hat a^*}\delta_2^2\eul^{-2\ii tg}\\0&1\end{pmatrix},\\
&\hat v_{\C{Y}_5}^{(5)}=\hat v_{\C{Y}_6}^{(5)}=\hat v_{\C{Y}_7}^{(5)}=\hat v_{\C{Y}_8}^{(5)}=\hat v_{\C{Y}_9}^{(5)}=\tilde\delta^{\sigma_3},&&\\
&\hat v_{\C{Y}_{10}}^{(5)}=\begin{pmatrix}1&0\\\frac{\hat b^*}{\hat a}\delta^{-2}\eul^{2\ii tg}&1\end{pmatrix},&&\hat v_{\C{Y}_{11}}^{(5)}=\begin{pmatrix}1&-\hat a\hat b\delta^2\eul^{-2\ii tg}\\0&1\end{pmatrix},\\
&\hat v_{\C{Y}_{12}}^{(5)}=\begin{pmatrix}1&0\\\frac{\hat b^*}{\hat a^2\hat a^*}\delta^{-2}\eul^{2\ii tg}&1\end{pmatrix},&&\hat v_{\C{Y}_{13}}^{(5)}=\begin{pmatrix}1&-\frac{\hat b}{\hat a^*}\delta^2\eul^{-2\ii tg}\\0&1\end{pmatrix},\\
&\hat v_{\C{Y}_{14}}^{(5)}=\eul^{-\ii tg_-\sigma_3}\begin{pmatrix}1&-\hat a\hat b\delta^2\\0&1\end{pmatrix}\eul^{\ii tg_+\sigma_3},&&\\
&\hat v_{\C{Y}_{15}}^{(5)}=\begin{pmatrix}\frac{\eul^{\ii t(g_+-g_-)}}{\hat a\hat a^*}&0\\0&\hat a\hat a^*\eul^{-\ii t(g_+-g_-)}\end{pmatrix},&&\hat v_{\C{Y}_{16}}^{(5)}=\begin{pmatrix}\eul^{\ii t(g_+-g_-)}&0\\0&\eul^{-\ii t(g_+-g_-)}\end{pmatrix},
\end{alignat*}
where $\hat v_{\C{Y}_l}^{(5)}$ denotes the restriction of $\hat v^{(5)}$ to $\C{Y}_l$. We extend the definition of $\hat v^{(5)}$ to the lower half-plane by means of the symmetry \eqref{eq:vjsym}.
\subsection{Sixth transformation}

The purpose of the sixth transformation is to make the jumps across the branch cuts $\Sigma_1$ and $\Sigma_2$ constant in $k$. Let $\Sigma^{\model}\coloneqq\Sigma_1\cup\Sigma_2$ denote the union of the branch cuts $\Sigma_1$ and $\Sigma_2$ oriented as in Figure~\ref{fig:jump-contour-model}.

Given a real number $\omega_1\equiv\omega_1(\xi)$, we define the function $\C{H}\equiv\C{H}(\xi,k)$ for $k\in\Sigma^{\model}\cap\D{C}^+$ by
\begin{equation}   \label{eq:Hcal}
\C{H}=\begin{cases}
\frac{2\omega_1-\ii\ln(\hat a_+\hat a_-\delta_2^2\eul^{\ii\phi})}{\tilde w_+}, &k\in\Sigma_1\cap\D{C}^+,\\
-\frac{\ii\ln(\nu_1^2\delta_2^2)}{\tilde w_+},&k\in\Sigma_2\cap\D{C}^+,
\end{cases}\quad(x,t)\in\D{S},
\end{equation}
where $\tilde w$ is defined in \eqref{eq:wtilde}, and extend it to $\Sigma^{\model}\cap\D{C}^-$ by the symmetry
\begin{equation}  \label{eq:Hcalsym}
\C{H}=\C{H}^*,
\end{equation}
i.e., $\C{H}(\xi,k)=\overline{\C{H}(\xi,\bar k)}$. The branches of the logarithms in \eqref{eq:Hcal} can be chosen arbitrarily as long as $\C{H}(\xi,k)$ is a continuous function of $k$ on each of the three segments $\Sigma_1\cap\D{C}^+$, $\Sigma_1\cap\D{C}^-$, and $\Sigma_2$ (such a choice exists because $|\nu_1^2 \delta_2^2| = 1$ at the point where $\Sigma_2$ crosses the real line).

We define $\tilde h\equiv\tilde h(\xi,k)$ by
\begin{equation}   \label{eq:htilde}
\tilde h(\xi,k)=\frac{\tilde w(k)}{2\pi\ii}\int_{\Sigma^{\model}}\frac{\C{H}(\xi,s)}{s-k}\dd s,\quad k\in\D{C}\setminus\Sigma^{\model}.
\end{equation}
In general, the function $\tilde h$ has a pole at $k=\infty$. The following lemma shows that by choosing $\omega_1(\xi)$ appropriately, the pole at $\infty$ can be removed.

\begin{lemma} \label{omega-1}
There is a unique choice of $\omega_1(\xi)$ such that the function $\tilde h(\xi,k)$ defined in \eqref{eq:htilde} has the following properties:
\begin{enumerate}[\rm(a)]
\item
$\tilde h$ obeys the symmetry $\tilde h=\tilde h^*$, \emph{i.e.},
\begin{equation}  \label{eq:htildesym}
\tilde h(\xi,k)=\overline{\tilde h(\xi,\bar k)},\quad k\in\hat{\D{C}}\setminus\Sigma^{\model},\ (x,t)\in\D{S},
\end{equation}
\item
As $k$ goes to infinity,
\begin{equation}  \label{eq:htildeas}
\tilde h(\xi,k)=\tilde h(\xi,\infty)+\ord(k^{-1}),\quad k\to\infty,\ (x,t)\in\D{S},
\end{equation}
where
\begin{equation}   \label{eq:htildeinfty}
\tilde h(\xi,\infty)=-\frac{1}{2\pi\ii}\int_{\Sigma^{\model}}s\C{H}(\xi,s)\dd s
\end{equation}
is a real-valued bounded function of $\xi$.
\item
For each $(x,t)\in\D{S}$, $\eul^{\ii\tilde h(\xi,k)\sigma_3}$ is an analytic function of $k\in\hat{\D{C}}\setminus\Sigma^{\model}$.
\item
$\eul^{\ii\tilde h(\xi,k)\sigma_3}$ satisfies the uniform bound
\[
\abs{\eul^{\ii\tilde h(\xi,k)\sigma_3}}\leq C,\quad k\in\hat{\D{C}}\setminus\Sigma^{\model},\ (x,t)\in\D{S}.
\]
\item
$\tilde h(\xi,k)$ satisfies the following jump conditions across $\Sigma^{\model}$:
\begin{equation}   \label{eq:vhtilde}
\tilde h_++\tilde h_-=\begin{cases}
2\omega_1-\ii\ln(\hat a_+\hat a_-\delta_2^2\eul^{\ii\phi}),&k\in\Sigma_1\cap\D{C}^+,\\
2\omega_1+\ii\ln(\hat a_+^*\hat a_-^*\delta_2^{-2}\eul^{-\ii\phi}),&k\in\Sigma_1\cap\D{C}^-,\\
-\ii\ln(\nu_1^2\delta_2^2),&k\in\Sigma_2.
\end{cases}
\end{equation}
\end{enumerate}
\end{lemma}

\begin{proof}
The proof is similar to the analogous proof in the genus $3$ sector \cite{BLS22}*{Lemma 4.2}.
\end{proof}

We henceforth fix $\omega_1(\xi)$ to be the unique choice which ensures that $\tilde h(\xi,k)$ has the properties of Lemma~\ref{omega-1}. We define $\hat m^{(6)}\equiv\hat m^{(6)}(x,t,k)$ by
\[
\hat m^{(6)}(x,t,k)\coloneqq\eul^{-\ii\tilde h(\xi,\infty)\sigma_3}\hat m^{(5)}(x,t,k)\eul^{\ii\tilde h(\xi,k)\sigma_3}.
\]
By Lemma~\ref{omega-1}, $\hat m$ satisfies the RH problem \eqref{eq:rhp0} iff $\hat m^{(6)}$ satisfies the RH problem \eqref{eq:mj} with $j=6$, where $\Sigma^{(6)}=\Sigma^{(5)}$ and
\[
\hat v^{(6)}=\eul^{-\ii\tilde h_-\sigma_3}\hat v^{(5)}\eul^{\ii\tilde h_+\sigma_3}.
\]
The jump matrix $\hat v^{(6)}$ is given explicitly in $\D{C}^+\setminus \overline{D_R(\beta)}$ by
\begin{alignat*}{2}
&\hat v_1^{(6)}=\begin{pmatrix}1&0\\\frac{\hat b^*}{\hat a}\delta_2^{-2}\eul^{2\ii tg}\eul^{2\ii\tilde h}&1\end{pmatrix},&\qquad&\hat v_2^{(6)}=\begin{pmatrix}1&-\hat a\hat b\delta_2^2\eul^{-2\ii tg}\eul^{-2\ii\tilde h}\\0&1\end{pmatrix},\\
&\hat v_3^{(6)}=\begin{pmatrix}0&\ii\eul^{-2\ii(t\Omega_1+\omega_1)}\\\ii\eul^{2\ii(t\Omega_1+\omega_1)}&0\end{pmatrix},&&\hat v_4^{(6)}=\begin{pmatrix}0&\ii\\\ii&0\end{pmatrix},\\
&\hat v_5^{(6)}=\begin{pmatrix}1&0\\\frac{\hat b^*}{\hat a^2\hat a^*}\delta_2^{-2}\eul^{2\ii tg}\eul^{2\ii\tilde h}&1\end{pmatrix},&&\hat v_6^{(6)}=\begin{pmatrix}\eul^{\ii t(g_+-g_-)}&0\\0&\eul^{-\ii t(g_+-g_-)}\end{pmatrix},\\
&\hat v_7^{(6)}=\begin{pmatrix}1&-\frac{\hat b}{\hat a^*}\delta_2^2\eul^{-2\ii tg}\eul^{-2\ii\tilde h}\\0&1\end{pmatrix},&&
\end{alignat*}
where we used \eqref{eq:vhtilde} and the relations $g_++g_-=2\Omega_1$ along $\Sigma_1$ and $g_++g_-=0$ along $\Sigma_2$ \cite{BLS22}*{Lemma 3.2 (c)}. Here, $\omega_1\equiv\omega_1(\xi)$ and $\Omega_1\equiv\Omega_1(\xi)$ are two real constants and $\hat v_l^{(6)}$ denotes the restriction of $\hat v^{(6)}$ to the contour labeled by $l$ in Figure~\ref{fig:jump-contour-6}, outside of $\overline{D_R(\beta)}$. 

On the part of the contour that lies in $D_{\epsilon}(\beta)$ the jump matrix $\hat v^{(6)}$ is given by
\[
\hat v_{\C{Y}_l}^{(6)}=\eul^{-\ii\tilde h\sigma_3}\hat v_{\C{Y}_l}^{(5)}\eul^{\ii\tilde h\sigma_3},\quad l=1,\dots,16.
\]
That is
\begin{alignat*}{2}
&\hat v_{\C{Y}_1}^{(6)}=\begin{pmatrix}1&0\\\frac{\hat b^*}{\hat a}\delta_2^{-2}\eul^{2\ii tg}\eul^{2\ii\tilde h}&1\end{pmatrix},&\quad&\hat v_{\C{Y}_2}^{(6)}=\begin{pmatrix}1&-\hat a\hat b\delta_2^2\eul^{-2\ii tg}\eul^{-2\ii\tilde h}\\0&1\end{pmatrix},\\
&\hat v_{\C{Y}_3}^{(6)}=\begin{pmatrix}1&0\\\frac{\hat b^*}{\hat a^2\hat a^*}\delta_2^{-2}\eul^{2\ii tg}\eul^{2\ii\tilde h}&1\end{pmatrix},&&\hat v_{\C{Y}_4}^{(6)}=\begin{pmatrix}1&-\frac{\hat b}{\hat a^*}\delta_2^2\eul^{-2\ii tg}\eul^{-2\ii\tilde h}\\0&1\end{pmatrix},\\
&\hat v_{\C{Y}_5}^{(6)}=\hat v_{\C{Y}_6}^{(6)}=\hat v_{\C{Y}_7}^{(6)}=\hat v_{\C{Y}_8}^{(6)}=\hat v_{\C{Y}_9}^{(6)}=\tilde\delta^{\sigma_3},&&\\
&\hat v_{\C{Y}_{10}}^{(6)}=\begin{pmatrix}1&0\\\frac{\hat b^*}{\hat a}\delta^{-2}\eul^{2\ii tg}\eul^{2\ii\tilde h}&1\end{pmatrix},&&\hat v_{\C{Y}_{11}}^{(6)}=\begin{pmatrix}1&-\hat a\hat b\delta^2\eul^{-2\ii tg}\eul^{-2\ii\tilde h}\\0&1\end{pmatrix},\\
&\hat v_{\C{Y}_{12}}^{(6)}=\begin{pmatrix}1&0\\\frac{\hat b^*}{\hat a^2\hat a^*}\delta^{-2}\eul^{2\ii tg}\eul^{2\ii\tilde h}&1\end{pmatrix},&&\hat v_{\C{Y}_{13}}^{(6)}=\begin{pmatrix}1&-\frac{\hat b}{\hat a^*}\delta^2\eul^{-2\ii tg}\eul^{-2\ii\tilde h}\\0&1\end{pmatrix},\\
&\hat v_{\C{Y}_{14}}^{(6)}=\eul^{-\ii tg_-\sigma_3}\begin{pmatrix}1&-\hat a\hat b\delta^2\eul^{-2\ii\tilde h}\\0&1\end{pmatrix}\eul^{\ii tg_+\sigma_3},&&\\
&\hat v_{\C{Y}_{15}}^{(6)}=\begin{pmatrix}\frac{\eul^{\ii t(g_+-g_-)}}{\hat a\hat a^*}&0\\0&\hat a\hat a^*\eul^{-\ii t(g_+-g_-)}\end{pmatrix},&&\hat v_{\C{Y}_{16}}^{(6)}=\begin{pmatrix}\eul^{\ii t(g_+-g_-)}&0\\0&\eul^{-\ii t(g_+-g_-)}\end{pmatrix}.
\end{alignat*}
\section{Global parametrix}  \label{sec:model}

\begin{figure}[ht]
\centering\includegraphics[scale=.7]{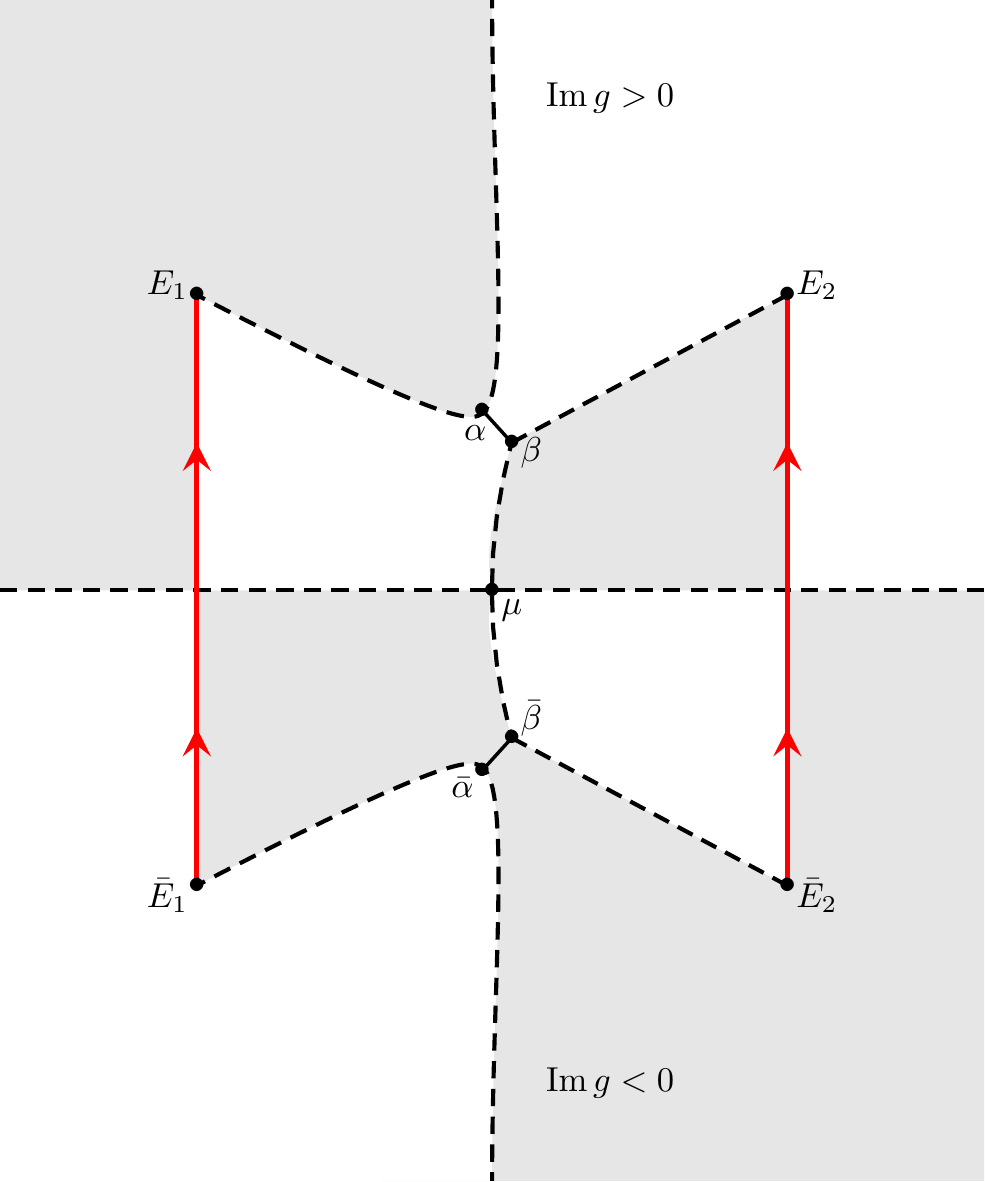}
\caption{The jump contour $\Sigma^{\model}=\Sigma_1\cup\Sigma_2$ (red with arrows indicating orientation) together with the set where $\Im g=0$ (dashed black).} 
\label{fig:jump-contour-model}
\end{figure}
Away from $\Sigma^{\model}$ and the critical points, the jump matrix $\hat v^{(6)}$ approaches the identity matrix as $t\to+\infty$ (since $\alpha\to\beta$ as $t\to+\infty$, this is true also for $\hat{v}_6^{(6)}$, see Section \ref{sec:final}). This leads us to expect that in the limit $t\to+\infty$, the solution $\hat m^{(6)}$ approaches the solution $m^{\model}$ of the RH problem
\begin{equation}   \label{eq:mmod}
\begin{cases}
m^{\model}(x,t,\,\cdot\,)\in I+\dot E^2(\hat{\D{C}}\setminus\Sigma^{\model}),&\\
m_+^{\model}(x,t,k)=m_-^{\model}(x,t,k)v^{\model}(x,t,k)&\text{for a.e. }k\in\Sigma^{\model},
\end{cases}
\end{equation}
where $v^{\model}$ denotes the restriction of $\hat v^{(6)}$ to $\Sigma^{\model}$, i.e.,
\begin{equation}   \label{eq:vmod}
v^{\model}=\begin{cases}
\begin{pmatrix}0&\ii\eul^{-2\ii(t\Omega_1(\xi)+\omega_1(\xi))}\\\ii\eul^{2\ii(t\Omega_1(\xi)+\omega_1(\xi))}&0\end{pmatrix},&k\in\Sigma_1,\\
\begin{pmatrix}0&\ii\\\ii&0\end{pmatrix},&k\in\Sigma_2.
\end{cases}
\end{equation}
The jump matrix $v^{\model}$ is off-diagonal and independent of $k$. This implies that we can write down an explicit solution of the RH problem \eqref{eq:mmod} in terms of theta functions.

Define the function $\nu\equiv\nu(k)$ for $k\in\D{C}\setminus\Sigma^{\model}$ by
\[
\nu\coloneqq\left(\frac{(k-E_1)(k-E_2)}{(k-\bar E_1)(k-\bar E_2)}\right)^{\frac{1}{4}},\quad k\in\D{C}\setminus\Sigma^{\model},
\]
where the branch of the fourth root is fixed by requiring that $\nu=1+\ord(k^{-1})$ as $k\to\infty$. Let $\hat\nu\colon\tilde M\to\hat{\D{C}}$ denote the function which is given by $\nu^2$ on the upper sheet and by $-\nu^2$ on the lower sheet of $\tilde M$, that is, $\hat\nu(k^\pm)=\pm\nu(k)^2$ for $k\in\D{C}\setminus\Sigma^{\model}$. Then $\hat\nu$ is a meromorphic function on $\tilde M$. Noting, for example, that $\hat\nu$ has two simple zeros at $E_1$, $E_2$, we see that $\hat\nu\colon\tilde M\to\hat{\D{C}}$ has degree two. Hence the function $\hat\nu-1$ has two zeros on $\tilde M$ counting multiplicity; we denote these zeros by $\infty^+,\,P_1\in\tilde M$. We define the complex constant $\tilde{\C{K}}$ by
\begin{equation}   \label{eq:calKtilde}
\tilde{\C{K}}\coloneqq\frac{1}{2}(1+\tau),
\end{equation}
where $\tau$ denotes the period $\int_{b_1}\tilde\zeta$.

We define the constant $\tilde d\in\D{C}$ by
\begin{equation}  \label{eq:dtilde}
\tilde d\coloneqq\tilde\varphi(P_1)+\tilde{\C{K}},
\end{equation}
where $\tilde\varphi\colon\tilde M\to\D{C}$ is the Abel map defined in \eqref{eq:tildephi}; if $P_1$ has projection in $\Sigma^{\model}$, then we fix the value of $\tilde d$ by letting $\tilde\varphi(P_1)$ denote the boundary value $\tilde\varphi_+(P_1)$ from the left say. We also define the complex-valued function $\tilde v\equiv\tilde v(\xi,t)$ by
\begin{equation}  \label{eq:vtilde}
\tilde v(\xi,t)\coloneqq-\frac{1}{\pi}\left(t\Omega_1(\xi)+\omega_1(\xi)\right).
\end{equation}

\begin{theorem}[Solution of the model RH problem] \label{thm:mmod}
For each choice of the real constants $\Omega_1$ and $\omega_1$ and for each $t\geq 0$, the RH problem \eqref{eq:mmod} has a unique solution $m^{\model}(x,t,k)$. Moreover, this solution satisfies
\begin{equation}   \label{eq:mmodinfty}
\lim_{k\to\infty}k(m^{\model}(x,t,k))_{12}=-\frac{\ii}{2}\Im(E_1+E_2)\times\frac{\tilde\Theta\left(\tilde\varphi(\infty^+)+\tilde d\right)\tilde\Theta\left(\tilde\varphi(\infty^+)-\tilde v(\xi,t)-\tilde d\right)}{\tilde\Theta\left(\tilde\varphi(\infty^+)+\tilde v(\xi,t)+\tilde d\right)\tilde\Theta\left(\tilde\varphi(\infty^+)-\tilde d\right)},
\end{equation}
where the limit is uniform with respect to $\arg k\in\croch{0,2\pi}$ and $\tilde\Theta$ is the Riemann theta function associated to the genus $1$ Riemann surface $\tilde M$, see \eqref{eq:tildetheta}.
\end{theorem}

\begin{proof}
The proof is similar to the analogous proof in the genus $3$ sector \cite{BLS22}*{Theorem 5.1}.
\end{proof}

\section{Local parametrices}   \label{sec:locmod}

The solution $m^{\model}$ is a good approximation of $\hat m^{(6)}$ as $t\to+\infty$ except for $k$ near the five critical points $\accol{\alpha,\beta,\mu,\bar\alpha,\bar\beta}$. In this section, we introduce local model solutions which are good approximations of $\hat m^{(6)}$ near these critical points. More precisely, we define two local solutions, denoted by $m^{\beta}(x,t,k)$ and $m^{\mu}(x,t,k)$, which are good approximations of $\hat m^{(6)}$ for $k$ in the disks $D_{\epsilon}(\beta)$ and $D_{\epsilon}(\mu)$, respectively.
   
\subsection{Local model near $\BS\mu$}  \label{sec:locmu}

The local model $m^{\mu}$ near $\mu$ is defined as in the case of the genus $3$ sector \cite{BLS22}*{Section 6.3}.

\subsection{Local model near $\BS\beta$} \label{sec:locbeta}

\begin{remark}[Motivating remark]  \label{rem:motiv}
The derivative $\dd g/\dd k$ of the $g$-function for the genus $3$ sector behaves like $(k-\alpha)^{1/2}$ and $(k-\beta)^{1/2}$ as $k$ approaches the critical points $\alpha$ and $\beta$, respectively. If $\alpha$ and $\beta$ stay separated, we have $g(k)\sim(k-\alpha)^{3/2}$ as $k\to\alpha$ and $g(k)\sim(k-\beta)^{3/2}$ as $k\to\beta$. The contributions to the solution $q(x,t)$ from $\alpha$ and $\beta$ can then be computed using the Airy local model \cite{BLS22}*{Appendix B} and each contribution is of order $\ord(t^{-N})$ for any $N\geq 1$. However, if $\alpha$ approaches $\beta$, $\dd g/\dd k$ behaves roughly like $(k-\alpha)^{1/2}(k-\beta)^{1/2}\approx k-\beta$ for $k$ in $D_{\epsilon}(\beta)\setminus D_R(\beta)$, where $R$ is some radius such that $\abs{\alpha-\beta}\ll R<\epsilon$. Hence $g(k)$ behaves as if it had a double zero at $\beta$. This suggests the construction of a local model near $\beta$ as follows.
\end{remark}

Let $\accol{\C{S}_j}_1^{10}$ denote the open subsets of $D_{\epsilon}(\beta)$ shown in Figure~\ref{fig:jump-contour-Y} which separate the $\C{Y}_j$. Let $\gamma_{(\alpha,\infty)}$ denote the contour from $\alpha$ to infinity labeled by $2$ in Figure~\ref{fig:jump-contour-4}. Let $\gamma_{\cut}\coloneqq\gamma_{(\beta,\alpha)}\cup\gamma_{(\alpha,\infty)}$. Let $\C{S}_+\coloneqq\C{S}_3\cup\C{S}_4\cup\C{S}_8\cup\C{S}_9$ and $\C{S}_-\coloneqq\C{S}_1\cup\C{S}_2\cup\C{S}_5\cup\C{S}_6\cup\C{S}_7\cup\C{S}_{10}$ denote the parts of $D_{\epsilon}(\beta)$ that lie to the left and right of the curve $\gamma_{(\mu,\beta)}\cup\gamma_{\cut}$, respectively.
\subsubsection{First transformation}  \label{sec:beta1}

We define the function $g_{\beta}(\xi,k)$ for $k$ near $\beta$ by
\[
g_{\beta}(\xi,k)=\int_{\beta}^k\dd g=
\begin{cases}
g(\xi,k)-g_+(\xi,\beta),&k\in\C{S}_+,\\
g(\xi,k)-g_-(\xi,\beta),&k\in\C{S}_-,
\end{cases}\quad(x,t)\in\D{S}.
\]
For each $\xi$, $g_{\beta}(\xi,k)$ is an analytic function of $k\in D_{\epsilon}(\beta)\setminus\gamma_{\cut}$ with jump across $\gamma_{\cut}$.

Let
\[
m^{(\beta0)}(x,t,k)\coloneqq
\hat m^{(6)}(x,t,k)\eul^{-\ii\tilde h(\xi,k)\sigma_3}A(x,t,k),\quad k\in D_{\epsilon}(\beta)\setminus\C{Y},\ (x,t)\in\D{S},
\]
where $A(x,t,k)$ denotes the sectionally holomorphic function
\begin{equation}  \label{eq:A}
A(x,t,k)\coloneqq
\begin{cases}
(\delta^2\hat a\hat b)^{\frac{\sigma_3}{2}}\eul^{-\ii tg_+(\xi,\beta)\sigma_3},&k\in\C{S}_+,\\
(\delta^2\hat a\hat b)^{\frac{\sigma_3}{2}}\eul^{-\ii tg_-(\xi,\beta)\sigma_3},&k\in\C{S}_-,
\end{cases}\quad(x,t)\in\D{S}.
\end{equation}
Since the function $\delta^2\hat a\hat b$ is nonzero and analytic in $D_{\epsilon}(\beta)$, the square root $(\delta^2\hat a\hat b)^{\frac{1}{2}}$ is well-defined and analytic; the branch of $(\delta^2\hat a\hat b)^{\frac{1}{2}}$ can be chosen arbitrarily. 

Then (recall that $\hat a\hat a^*=aa^*$, $\hat b\hat b^*=bb^*$, and $\hat r\hat r^*=rr^*$)
\begin{alignat*}{2}
&v_{\C{Y}_1}^{(\beta0)}=\begin{pmatrix}1&0\\bb^*\tilde\delta^{-2}\eul^{2\ii tg_{\beta}}&1\end{pmatrix},&\quad&v_{\C{Y}_2}^{(\beta0)}=\begin{pmatrix}1&-\tilde\delta^2\eul^{-\ii t(g_{\beta+}+g_{\beta-})}\\0&1\end{pmatrix},\\
&v_{\C{Y}_3}^{(\beta0)}=\begin{pmatrix}1&0\\\frac{bb^*}{aa^*}\tilde\delta^{-2}\eul^{2\ii tg_{\beta}}&1\end{pmatrix},&&v_{\C{Y}_4}^{(\beta0)}=\begin{pmatrix}1&-\frac{1}{aa^*}\tilde\delta^2\eul^{-2\ii tg_{\beta}}\\0&1\end{pmatrix},\\
&v_{\C{Y}_5}^{(\beta0)}=v_{\C{Y}_6}^{(\beta0)}=v_{\C{Y}_7}^{(\beta0)}=v_{\C{Y}_8}^{(\beta0)}=v_{\C{Y}_9}^{(\beta0)}=\tilde\delta^{\sigma_3},&&\\
&v_{\C{Y}_{10}}^{(\beta0)}=\begin{pmatrix}1&0\\bb^*\eul^{2\ii tg_{\beta}}&1\end{pmatrix},&&v_{\C{Y}_{11}}^{(\beta0)}=\begin{pmatrix}1&-\eul^{-\ii t(g_{\beta+}+g_{\beta-})}\\0&1\end{pmatrix},\\
&v_{\C{Y}_{12}}^{(\beta0)}=\begin{pmatrix}1&0\\\frac{bb^*}{aa^*}\eul^{2\ii tg_{\beta}}&1\end{pmatrix},&&v_{\C{Y}_{13}}^{(\beta0)}=\begin{pmatrix}1&-\frac{1}{aa^*}\eul^{-2\ii tg_{\beta}}\\0&1\end{pmatrix},\\
&v_{\C{Y}_{14}}^{(\beta0)}=\eul^{-\ii tg_{\beta-}\sigma_3}\begin{pmatrix}1&-1\\0&1\end{pmatrix}\eul^{\ii tg_{\beta+}\sigma_3},&&v_{\C{Y}_{15}}^{(\beta0)}=\begin{pmatrix}\frac{1}{aa^*}&0\\0&aa^*\end{pmatrix},\\
&v_{\C{Y}_{16}}^{(\beta0)}=I,&&
\end{alignat*}
or, in terms of $r$,
\begin{alignat*}{2}
&v_{\C{Y}_1}^{(\beta0)}=\begin{pmatrix}1&0\\\frac{rr^*}{1+rr^*}\tilde\delta^{-2}\eul^{2\ii tg_{\beta}}&1\end{pmatrix},&\quad&v_{\C{Y}_2}^{(\beta0)}=\begin{pmatrix}1&-\tilde\delta^2\eul^{-\ii t(g_{\beta+}+g_{\beta-})}\\0&1\end{pmatrix},\\
&v_{\C{Y}_3}^{(\beta0)}=\begin{pmatrix}1&0\\rr^*\tilde\delta^{-2}\eul^{2\ii tg_{\beta}}&1\end{pmatrix},&&v_{\C{Y}_4}^{(\beta0)}=\begin{pmatrix}1&-(1+rr^*)\tilde\delta^2\eul^{-2\ii tg_{\beta}}\\0&1\end{pmatrix},\\
&v_{\C{Y}_5}^{(\beta0)}=v_{\C{Y}_6}^{(\beta0)}=v_{\C{Y}_7}^{(\beta0)}=v_{\C{Y}_8}^{(\beta0)}=v_{\C{Y}_9}^{(\beta0)}=\tilde\delta^{\sigma_3},&&\\
&v_{\C{Y}_{10}}^{(\beta0)}=\begin{pmatrix}1&0\\\frac{rr^*}{1+rr^*}\eul^{2\ii tg_{\beta}}&1\end{pmatrix},&&v_{\C{Y}_{11}}^{(\beta0)}=\begin{pmatrix}1&-\eul^{-\ii t(g_{\beta+}+g_{\beta-})}\\0&1\end{pmatrix},\\
&v_{\C{Y}_{12}}^{(\beta0)}=\begin{pmatrix}1&0\\rr^*\eul^{2\ii tg_{\beta}}&1\end{pmatrix},&&v_{\C{Y}_{13}}^{(\beta0)}=\begin{pmatrix}1&-(1+rr^*)\eul^{-2\ii tg_{\beta}}\\0&1\end{pmatrix},\\
&v_{\C{Y}_{14}}^{(\beta0)}=\eul^{-\ii tg_{\beta-}\sigma_3}\begin{pmatrix}1&-1\\0&1\end{pmatrix}\eul^{\ii tg_{\beta+}\sigma_3},&&v_{\C{Y}_{15}}^{(\beta0)}=\begin{pmatrix}1+rr^*&0\\0&\frac{1}{1+rr^*}\end{pmatrix},\\
&v_{\C{Y}_{16}}^{(\beta0)}=I.&&
\end{alignat*}
\subsubsection{Second transformation}  \label{sec:beta2}

We would next like to remove all jumps within $D_R(\beta)$. This is not quite possible, because $g_{\beta}$ has a jump across $\C{Y}_{11}\cup\C{Y}_{14}$. However, the following transformation removes all jumps within $D_R(\beta)$ except those across $\C{Y}_{11}$ and $\C{Y}_{14}$. The remaining jumps across $\C{Y}_{11}$ and $\C{Y}_{14}$ go away in the limit $\alpha\to\beta$.

Let $m^{(\beta1)}\coloneqq m^{(\beta0)}B$, where $B(x,t,k)$ denotes the sectionally holomorphic function
\begin{equation}   \label{eq:B}
B(x,t,k)\coloneqq\begin{cases}
\begin{pmatrix}1&0\\-\frac{rr^*}{1+rr^*}\eul^{2\ii tg_{\beta}}&1\end{pmatrix},&k\in\C{S}_7,\\
\begin{pmatrix}1&0\\-rr^*\eul^{2\ii tg_{\beta}}&1\end{pmatrix}\begin{pmatrix}\frac{1}{1+rr^*}&\eul^{-2\ii tg_{\beta}}\\0&1+rr^*\end{pmatrix},&k\in\C{S}_8,\\
\begin{pmatrix}\frac{1}{1+rr^*}&\eul^{-2\ii tg_{\beta}}\\0&1+rr^*\end{pmatrix},&k\in\C{S}_9,\\
\begin{pmatrix}1&(1+rr^*)\eul^{-2\ii tg_{\beta}}\\0&1\end{pmatrix},&k\in\C{S}_{10},\\
\;I,&\text{else},
\end{cases}\quad(x,t)\in\D{S}.
\end{equation}
Then
\begin{align*}
&v_{\C{Y}_1}^{(\beta1)}=\begin{pmatrix}1&0\\\frac{rr^*}{1+rr^*}\tilde\delta^{-2}\eul^{2\ii tg_{\beta}}&1\end{pmatrix},\qquad\qquad v_{\C{Y}_2}^{(\beta1)}=\begin{pmatrix}1&-\tilde\delta^2\eul^{-\ii t(g_{\beta+}+g_{\beta-})}\\0&1\end{pmatrix},\\
&v_{\C{Y}_3}^{(\beta1)}=\begin{pmatrix}1&0\\rr^*\tilde\delta^{-2}\eul^{2\ii tg_{\beta}}&1\end{pmatrix},\qquad\qquad\quad v_{\C{Y}_4}^{(\beta1)}=\begin{pmatrix}1&-(1+rr^*)\tilde\delta^2\eul^{-2\ii tg_{\beta}}\\0&1\end{pmatrix},\\
&v_{\C{Y}_5}^{(\beta1)}=\tilde\delta^{\sigma_3},\qquad\qquad\qquad\qquad\qquad\qquad v_{\C{Y}_6}^{(\beta1)}=\tilde\delta^{\sigma_3}\begin{pmatrix}1&0\\-\frac{rr^*}{1+rr^*}\eul^{2\ii tg_{\beta}}&1\end{pmatrix},\\
&v_{\C{Y}_7}^{(\beta1)}=\tilde\delta^{\sigma_3}\begin{pmatrix}\frac{1}{1+rr^*}&\eul^{-2\ii tg_{\beta}}\\-\frac{rr^*}{1+rr^*}\eul^{2\ii tg_{\beta}}&1\end{pmatrix},\quad v_{\C{Y}_8}^{(\beta1)}=\tilde\delta^{\sigma_3}\begin{pmatrix}\frac{1}{1+rr^*}&\eul^{-2\ii tg_{\beta}}\\0&1+rr^*\end{pmatrix},\\
&v_{\C{Y}_9}^{(\beta1)}=\tilde\delta^{\sigma_3}\begin{pmatrix}1&(1+rr^*)\eul^{-2\ii tg_{\beta}}\\0&1\end{pmatrix},\\
&v_{\C{Y}_{11}}^{(\beta1)}=\begin{pmatrix}\frac{1+rr^*\eul^{\ii t(g_{\beta+}-g_{\beta-})}}{1+rr^*}&\eul^{-2\ii tg_{\beta+}}(1-\eul^{\ii t(g_{\beta+}-g_{\beta-})})\\\frac{rr^*\eul^{2\ii tg_{\beta+}}(\eul^{-2\ii t(g_{\beta+}-g_{\beta-})}-1+rr^*(\eul^{-\ii t(g_{\beta+}-g_{\beta-})}-1))}{(1+rr^*)^2}&1-\frac{rr^*\eul^{-\ii t(g_{\beta+}-g_{\beta-})}(1-\eul^{-\ii t(g_{\beta+}-g_{\beta-})})}{1+rr^*}\end{pmatrix},\\
&v_{\C{Y}_{14}}^{(\beta1)}=\begin{pmatrix}\eul^{\ii t(g_{\beta+}-g_{\beta-})}&0\\0&\eul^{-\ii t(g_{\beta+}-g_{\beta-})}\end{pmatrix},\qquad v_{\C{Y}_{10}}^{(\beta1)}=v_{\C{Y}_{12}}^{(\beta1)}=v_{\C{Y}_{13}}^{(\beta1)}=v_{\C{Y}_{15}}^{(\beta1)}=v_{\C{Y}_{16}}^{(\beta1)}=I.
\end{align*}

We note that in the limit as $\alpha$ approaches $\beta$, we have $g_{\beta+}-g_{\beta-}\to 0$ on $\C{Y}_{11}$ and $\C{Y}_{14}$. Hence the jumps across $\C{Y}_{11}$ and $\C{Y}_{14}$ vanish in this limit.
\subsubsection{Analytic approximation of $g_{\beta}$}  \label{sec:bapprox}

We want to map the RH problem for $m^{(\beta1)}$ to an RH problem which, in the large $t$ limit, can be approximated by the exactly solvable RH problem of Appendix~\ref{sec:app}. Thus, we want to introduce new variables $\lambda\equiv\lambda(x,t,k)$ and $y\equiv y(x,t)$ such that
\begin{equation}  \label{eq:ylambda}
2\ii tg_{\beta}(k)\approx-\lambda^2-2y\lambda.
\end{equation}
For each $(x,t)\in\D{S}$, we would like the map $k\mapsto\lambda$ to be a biholomorphism from $D_{\epsilon}(\beta)$ to a neighborhood of $0$ in the complex $\lambda$-plane. However, $g_{\beta}$ is not analytic at $k=\beta$. Our next goal is to prove Lemma~\ref{lem:bapprox}, which overcomes this problem by showing that $g_{\beta}$ can be well-approximated in $D_{\epsilon}(\beta)\setminus D_R(\beta)$ by an analytic function.

In this subsection, we will come across square roots of the form $\sqrt{s-k}$ with $k\in\gamma_{(\beta,\alpha)}$. For $k\in\gamma_{(\beta,\alpha)}$, let $\gamma_{(k,\infty)}$ denote the part of $\gamma_{\cut}$ that stretches from $k$ to $\infty$. Given $k\in\gamma_{(\beta,\alpha)}$, we then fix the branch of the square root $\sqrt{s-k}$ (or, more generally, of the complex power $(s-k)^a$) so that $\sqrt{s-k}$ is an analytic function of $s\in\D{C}\setminus\gamma_{(k,\infty)}$ and $\sqrt{s-k}\sim\sqrt{s}$ as $s\to+\infty$. In particular, the functions $\sqrt{s-\beta}$ and $\sqrt{s-\alpha}$ are analytic for $s\in\D{C}\setminus\gamma_{\cut}$ and $s\in\D{C}\setminus\gamma_{(\alpha,\infty)}$, respectively.

We define the function $g_1(s)\equiv g_1(\xi,s)$ by 
\[
g_1(s)\coloneqq\frac{g'(s)}{\sqrt{s-\beta}\sqrt{s-\alpha}},\quad s\in D_{\epsilon}(\beta)\setminus\gamma_{\cut}.
\]
The definition of $g(s)$ implies that, for each $\xi$, $g_1(\xi,s)$ is an analytic function of $s\in D_{\epsilon}(\beta)$. Moreover,
\begin{equation}  \label{eq:g1-estims}
\sup_{(x,t)\in\D{S}}\sup_{s\in D_{\epsilon}(\beta)}\left\lvert\frac{\partial^ng_1}{\partial s^n}(\xi,s)\right\rvert<\infty,\quad n=0,1,2,\dots
\end{equation}

\begin{lemma} \label{lem:bapprox}
We have
\begin{align*}
g_{\beta}(k)&=\int_{\beta}^kg_1(\xi,s)\left(s-\beta+\frac{\beta-\alpha}{2}\right)\dd s+\ord\left(\abs{\alpha-\beta}^2\ln\abs{\alpha-\beta}\right),\quad\ t\to+\infty,\\
&\qquad\qquad\qquad\qquad\qquad\qquad\quad k\in D_{\epsilon}(\beta)\setminus(D_R(\beta)\cup\gamma_{\cut}),\ (x,t)\in\D{S},
\end{align*}
where the error term is uniform with respect to $k$ and $x$ in the given ranges.
\end{lemma}

\begin{proof}
For each $s\in D_{\epsilon}(\beta)\setminus\gamma_{\cut}$, $\sqrt{s-k}$ is a smooth function of $k\in\gamma_{(\beta,\alpha)}$. Hence integration by parts gives the following Taylor expansion of $\sqrt{s-\alpha}$ as $\alpha\to\beta$:
\begin{equation}   \label{eq:sqrt-taylor}
\sqrt{s-\alpha}=\sqrt{s-\beta}+\frac{\beta-\alpha}{2\sqrt{s-\beta}}+\int_{\gamma_{(\beta,\alpha)}}\frac{u-\alpha}{4(s-u)^{3/2}}\dd u,\quad s\in D_{\epsilon}(\beta)\setminus\gamma_{\cut}.
\end{equation}
Using \eqref{eq:sqrt-taylor} we can write
\begin{align*}
g_{\beta}(k)&=\int_{\beta}^k\dd g=\int_{\beta}^kg_1(\xi,s)\sqrt{s-\alpha}\sqrt{s-\beta}\,\dd s\\
&=\int_{\beta}^kg_1(\xi,s)\left(s-\beta+\frac{\beta-\alpha}{2}\right)\dd s+E(\xi,k),\quad k\in D_{\epsilon}(\beta)\setminus(D_R(\beta)\cup\gamma_{\cut}),
\end{align*}
where the error term $E(\xi,k)$ is given by
\begin{equation}   \label{eq:error}
E(\xi,k)=\int_{\beta}^kg_1(\xi,s)\left(\int_{\gamma_{(\beta,\alpha)}}\frac{u-\alpha}{4(s-u)^{3/2}}\dd u\right)\sqrt{s-\beta}\,\dd s.
\end{equation}
Given $k\in D_{\epsilon}(\beta)\setminus(D_R(\beta)\cup\gamma_{\cut})$, we let the contour from $\beta$ to $k$ in \eqref{eq:error} consist of the straight segment from $\beta$ to $\beta+\frac{\abs{k-\beta}}{\abs{\alpha-\beta}}(\beta-\alpha)$ followed by an arc of the cut circle $\partial D_{\abs{k-\beta}}(\beta)\setminus\gamma_{\cut}$, see Figure~\ref{fig:contour-beta-k}.
\begin{figure}[ht]
\centering\includegraphics[scale=.6]{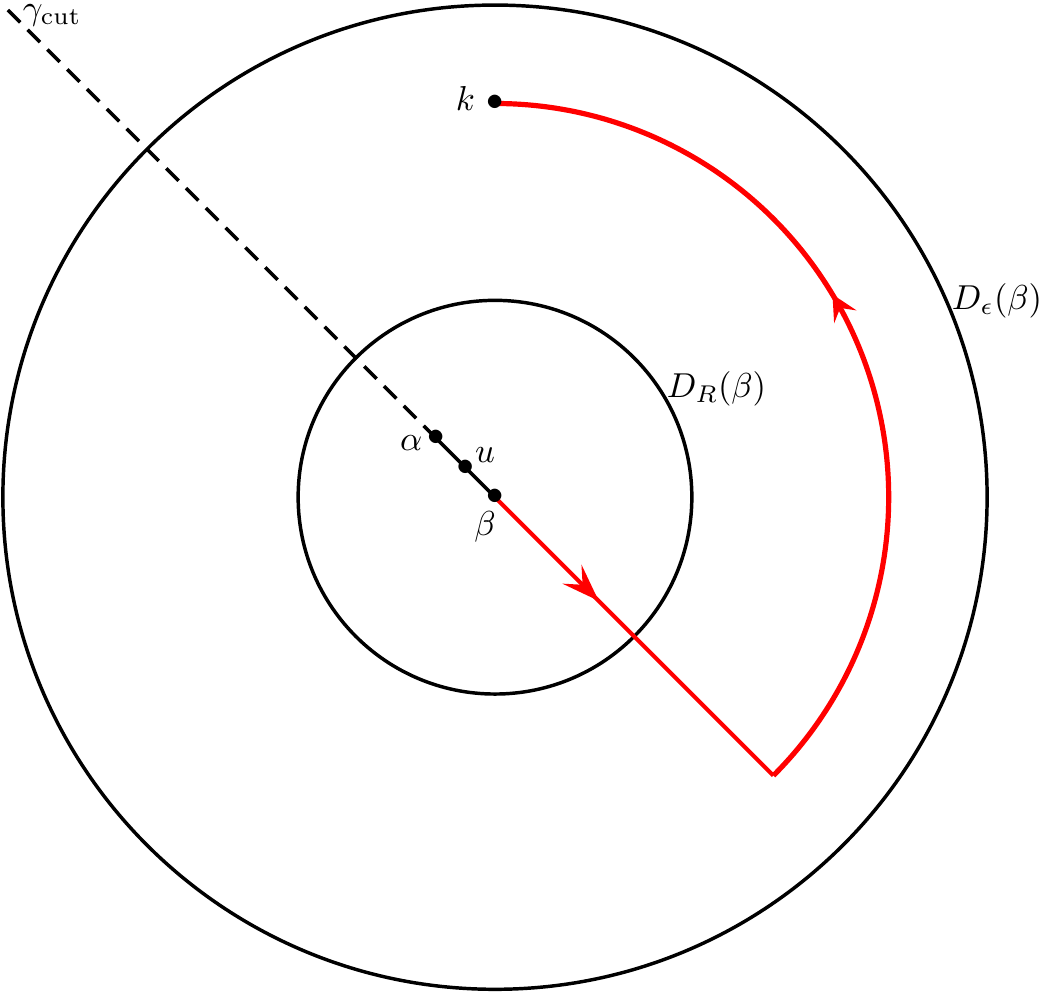}
\caption{The integration contour from $\beta$ to $k$ in equation \eqref{eq:error}} 
\label{fig:contour-beta-k}
\end{figure}

Increasing $T$ in \eqref{eq:wedge} if necessary, we may assume that $t$ is so large that $\gamma_{(\beta,\alpha)}$ deviates only slightly from the straight line segment $\croch{\alpha,\beta}$ from $\alpha$ to $\beta$. Then $\abs{s-\beta}\leq C\abs{s-u}$ for all $s$, $u$ on the given contours and for all $(x,t)\in\D{S}$. Hence, using \eqref{eq:g1-estims}, we can estimate
\begin{align*}
\abs{E(\xi,k)}
&\leq C\int_{\beta}^k\int_{\gamma_{(\beta,\alpha)}}\frac{\abs{u-\alpha}}{\abs{s-u}}\abs{\dd u}\abs{\dd s}\\
&\leq C\int_{\beta}^k\int_0^{\abs{\alpha-\beta}}\frac{\abs{\alpha-\beta}-v}{v+\abs{s-\beta}}\dd v\abs{\dd s}\\
&\leq C\int_0^{\abs{k-\beta}}\int_0^{\abs{\alpha-\beta}}\frac{\abs{\alpha-\beta}-v}{v+r}\dd v\dd r+2\pi C\abs{k-\beta}\int_0^{\abs{\alpha-\beta}}\frac{\abs{\alpha-\beta}-v}{v+\abs{k-\beta}}\dd v\\
&\leq C\int_0^{\abs{\alpha-\beta}}(\abs{\alpha-\beta}-v)\ln\left(\frac{v+\abs{k-\beta}}{v}\right)\dd v+C\int_0^{\abs{\alpha-\beta}}(\abs{\alpha-\beta}-v)\dd v\\
&\leq C\int_0^{\abs{\alpha-\beta}}(\abs{\alpha-\beta}-v)(1+\ln\abs{v})\dd v+C\abs{\alpha-\beta}^2\\
&\leq C\abs{\alpha-\beta}^2\abs{\ln\abs{\alpha-\beta}},\qquad k\in D_{\epsilon}(\beta)\setminus(D_R(\beta)\cup\gamma_{\cut}),\ (x,t)\in\D{S}.
\end{align*}
This completes the proof.
\end{proof}

\subsubsection{A local change of variables}  \label{sec:bchange}

Taylor expanding $g_1(\xi,s)$ around $s=\beta$ and using
\eqref{eq:g1-estims}, we find
\[
\abs{g_1(\xi,s)-g_1(\xi,\beta)-g_{1s}(\xi,\beta)(s-\beta)}\leq C\abs{s-\beta}^2,\quad s\in D_{\epsilon}(\beta),\ (x,t)\in\D{S},
\]
and hence
\begin{align*}
&\left\lvert\int_{\beta}^kg_1(\xi,s)\left(s-\beta+\frac{\beta-\alpha}{2}\right)\dd s-\left\lbrace g_1(\xi,\beta)\frac{\beta-\alpha}{2}(k-\beta)\right.\right.\\
&\qquad\qquad\left.\left.+\left(g_1(\xi,\beta)+g_{1s}(\xi,\beta)\frac{\beta-\alpha}{2}\right)\frac{(k-\beta)^2}{2}\right\rbrace\right\rvert\leq C(k-\beta)^3,\quad k\in D_{\epsilon}(\beta),\ (x,t)\in\D{S}.
\end{align*}
It follows that if we define the functions $\lambda\equiv\lambda(x,t,k)$ and $y\equiv y(x,t)$ by
\begin{align}   \label{eq:lambda}
\lambda&\coloneqq\sqrt{t}\,\psi_{\beta}(\xi)(k-\beta),\\
\label{eq:y}
y&\coloneqq-\sqrt{t}\,\frac{\ii g_1(\xi,\beta)(\beta-\alpha)}{2\psi_{\beta}(\xi)},
\end{align}
where
\begin{equation}   \label{eq:psibeta}
\psi_{\beta}(\xi)\coloneqq\eul^{-\frac{\ii\pi}{4}}\sqrt{g_1(\xi,\beta)+g_{1s}(\xi,\beta)\frac{\beta-\alpha}{2}},
\end{equation}
then
\begin{align}   \label{eq:g1estim}
&\Biggl\lvert 2\ii t\int_{\beta}^kg_1(\xi,s)\left(s-\beta+\frac{\beta-\alpha}{2}\right)\dd s-(-\lambda^2-2y\lambda)\Biggr\rvert\leq Ct\abs{k-\beta}^3,\notag\\
&\qquad\qquad\qquad\qquad\qquad\qquad\qquad\qquad\qquad\quad\ k\in D_{\epsilon}(\beta),\ (x,t)\in\D{S}.
\end{align}

For each $(x,t)\in\D{S}$, the map $k\mapsto\lambda$ is a biholomorphism of the disk $D_{\epsilon}(\beta)$ onto the open disk of radius $\sqrt{t}\abs{\psi_{\beta}(\xi)}\epsilon$ centered at the origin. By deforming the contour $\Sigma^{(6)}$ slightly and fixing the branch of the square root in \eqref{eq:psibeta} appropriately, we may assume that $\C{Y}_j$ is mapped into the ray $Y_j$ defined in \eqref{eq:rays} for $j=1,\dots,4$. Furthermore, by choosing 
\[
c(\xi)=\abs{\psi_{\beta}(\xi)}^{-1}
\]
in the definition \eqref{eq:R} of $R(\xi,t)$, we can arrange so that the circle $\partial D_R(\beta)$ is mapped onto the unit circle. Finally, we choose the branch cut $\C{C}$ in Figure~\ref{fig:raysY} so that $\C{Y}_{15}$ and $\C{Y}_{16}$ are mapped into $\C{C}$.

\begin{remark}   \label{rem:psibeta}
For $\xi=\xi_0=0$ we have
\[
\psi_{\beta}(0)=\eul^{-\frac{\ii\pi}{4}}\sqrt{g_1(0,\beta(0))},
\]
and, by symmetry, $g_1(0,\beta(0))>0$. So in this case we can fix the branch of the square root in \eqref{eq:psibeta} by requiring that 
\[
\arg\psi_{\beta}(\xi)\in\left(-\frac{\pi}{2},0\right)
\]
for all small enough $\xi$.
\end{remark}

\begin{lemma}   \label{lem:gbestims}
We have
\begin{align}   \label{eq:gbestim1}
2\ii tg_{\beta}(\xi,k)&=-\lambda^2-2y\lambda+\ord(t\abs{\alpha-\beta}^2\abs{\ln\abs{\alpha-\beta}})+\ord(t\abs{k-\beta}^3),\notag\\
&\qquad\qquad t\to+\infty,\ k\in D_{\epsilon}(\beta)\setminus(D_R(\beta)\cup\gamma_{\cut}),\ (x,t)\in\D{S},
\end{align}
and
\begin{equation}  \label{eq:gbestim2}
\abs{g_{\beta+}(\xi,k)-g_{\beta-}(\xi,k)}=\ord(\abs{\alpha-\beta}^2),\quad t\to+\infty,\ k\in\gamma_{\cut},\ (x,t)\in\D{S},
\end{equation}
where the error terms are uniform with respect to $k$ and $x$ in the given ranges.
\end{lemma}

\begin{proof}
The estimate \eqref{eq:gbestim1} follows immediately from Lemma~\ref{lem:bapprox} and \eqref{eq:g1estim}.

In order to prove \eqref{eq:gbestim2}, we note that
\begin{align*}
g_{\beta+}(\xi,k)-g_{\beta-}(\xi,k)
&=\int_{\beta}^k\left(\Bigl(\frac{\dd g}{\dd k}\Bigr)_+\!(\xi,s)-\Bigl(\frac{\dd g}{\dd k}\Bigr)_-\!(\xi,s)\right)\dd s\\
&=\int_{\beta}^kg_1(\xi,s)\left(\left(\sqrt{s-\alpha}\sqrt{s-\beta}\right)_+-\left(\sqrt{s-\alpha}\sqrt{s-\beta}\right)_-\right)\dd s\\
&=2\int_{\beta}^kg_1(\xi,s)\left(\sqrt{s-\alpha}\sqrt{s-\beta}\right)_+\dd s,\quad k\in\gamma_{(\beta,\alpha)},\ (x,t)\in\D{S},
\end{align*}
where the integration contour follows the curve $\gamma_{(\beta,\alpha)}$ from $\beta$ to $k$. Hence, by \eqref{eq:g1-estims},
\begin{align*}
\abs{g_{\beta+}(\xi,k)-g_{\beta-}(\xi,k)}
&\leq C\int_0^{\abs{k-\beta}}\sqrt{\abs{\alpha-\beta}-u}\,\sqrt{u}\,\dd u\\
&\leq C\int_0^{\abs{\alpha-\beta}}\sqrt{\abs{\alpha-\beta}-u}\,\sqrt{u}\,\dd u\\
&=C\frac{\pi\abs{\alpha-\beta}^2}{8}\,,\quad k\in\gamma_{(\beta,\alpha)},\ (x,t)\in\D{S}.
\end{align*}
This establishes \eqref{eq:gbestim2} for $k\in\gamma_{(\beta,\alpha)}$. Since
\[
g_{\beta+}(\xi,k)-g_{\beta-}(\xi,k)=g_{\beta+}(\xi,\alpha)-g_{\beta-}(\xi,\alpha),\quad k\in\gamma_{(\alpha,\infty)},\ (x,t)\in\D{S},
\]
the estimate \eqref{eq:gbestim2} holds also for $k\in\gamma_{(\alpha,\infty)}$.
\end{proof}

\subsubsection{Behavior of $\tilde\delta$ as $k\to\beta$}

We next consider the behavior of $\tilde\delta$ as $k$ approaches $\beta$. Let $\ln_{\mu}(k-\beta)$ denote the function $\ln(k-\beta)$ with branch cut along $\gamma_{(\mu,\beta)}$. More precisely,
\[
\ln_{\mu}(k-\beta)=\ln(k-\beta),\quad k\in D_{\epsilon}(\beta)\setminus\gamma_{(\mu,\beta)},
\]
where the branch is fixed by requiring that
\begin{equation}   \label{eq:lnmu}
\ln_{\C{C}}(\lambda)=\ln_{\mu}(k-\beta)+\ln(\sqrt{t}\psi_{\beta}(\xi)),\quad k\in D_{\epsilon}(\beta)\setminus\gamma_{(\mu,\beta)},\ (x,t)\in\D{S},
\end{equation}
where $\ln_{\C{C}}(\lambda)$ denotes the logarithm of $\lambda$ with cut along $\C{C}$, see Appendix~\ref{sec:app}.

We also define the function $L_{\mu}$ by
\[
L_{\mu}(s,k)=\ln(k-s),\quad s\in\gamma_{(\mu,\beta)},\ k\in D_{\epsilon}(\beta)\setminus\gamma_{(\mu,\beta)},
\]
where the branch is fixed by requiring that:
\begin{enumerate}[(i)]
\item
$L_{\mu}(s,k)$ is a continuous function of $s\in\gamma_{(\mu,\beta)}$ for each $k\in D_{\epsilon}(\beta)\setminus\gamma_{(\mu,\beta)}$.
\item
For $s=\beta$, we have $L_{\mu}(\beta,k)=\ln_{\mu}(k-\beta)$. 
\end{enumerate}
Integrating by parts, we find
\begin{align*}
\int_{\gamma_{(\mu,\beta)}}\frac{\ln(1+r(s)r^*(s))}{s-k}\,\dd s
&=\ln_{\mu}(k-\beta)\ln(1+r(\beta)r^*(\beta))-L_{\mu}(\mu,k)\ln(1+\abs{r(\mu)}^2)\\
&\quad-\int_{\gamma_{(\mu,\beta)}}L_{\mu}(s,k)\,\dd\ln(1+r(s)r^*(s)).
\end{align*}
Hence we can write
\begin{equation}   \label{eq:tdeltanu}
\tilde\delta(\xi,k)=\eul^{-\ii\tilde\nu(\xi)\ln_{\mu}(k-\beta)+\tilde\chi(\xi,k)},\quad k\in\D{C}\setminus\gamma_{(\bar\beta,\beta)},\ (x,t)\in\D{S},
\end{equation}
where $\tilde\nu(\xi)$ is defined in \eqref{eq:nutilde} and the function $\tilde\chi(\xi,k)$ is defined by 
\begin{align}    \label{eq:chitilde}
\tilde\chi(\xi,k)&\coloneqq-\frac{1}{2\pi\ii}L_{\mu}(\mu,k)\ln(1+\abs{r(\mu)}^2)-\frac{1}{2\pi\ii}\int_{\gamma_{(\mu,\beta)}}L_{\mu}(s,k)\,\dd\ln(1+r(s)r^*(s))\notag\\
&\quad-\frac{1}{2\pi\ii}\int_{\gamma_{(\bar\beta,\mu)}}\frac{\ln(1+r(s)r^*(s))}{s-k}\,\dd s,\qquad\qquad k\in\D{C}\setminus\gamma_{(\bar\beta,\beta)},\ (x,t)\in\D{S}.
\end{align}

\begin{lemma}   \label{lem:tdelta}
We have
\[
\tilde\delta(\xi,k)=\lambda^{-\ii\tilde\nu(\xi)}t^{\frac{\ii\tilde\nu(\xi)}{2}}\delta_0(\xi)\delta_1(\xi,k),\quad k\in D_{\epsilon}(\beta)\setminus\gamma_{(\mu,\beta)},\ (x,t)\in\D{S},
\]
where the functions $\delta_0$ and $\delta_1$ are defined by 
\begin{alignat*}{3}
\delta_0(\xi)&\coloneqq\eul^{\ii\tilde\nu(\xi)\ln\psi_{\beta}(\xi)}\eul^{\tilde\chi(\xi,\beta)},&\quad&(x,t)\in\D{S},&&\\
\delta_1(\xi,k)&\coloneqq\eul^{\tilde\chi(\xi,k)-\tilde\chi(\xi,\beta)},&&(x,t)\in\D{S},&\quad&k\in D_{\epsilon}(\beta)\setminus\gamma_{(\mu,\beta)}.
\end{alignat*} 
\end{lemma}

\begin{proof}
Immediate from \eqref{eq:lnmu} and \eqref{eq:tdeltanu}.
\end{proof}

We also note that $\tilde\chi(\xi,k)$ is uniformly bounded, i.e.,
\[
\abs{\tilde\chi(\xi,k)}\leq C,\quad k\in D_{\epsilon}(\beta)\setminus\gamma_{(\mu,\beta)},\ (x,t)\in\D{S}.
\]

\subsubsection{Third transformation}

Let $m^{(\beta2)}(x,t,k)=m^{(\beta1)}(x,t,k)D(x,t,k)$, where $D$ is defined by
\begin{equation}   \label{eq:D}
D(x,t,k)=\begin{cases}
t^{\frac{\ii\tilde\nu(\xi)\sigma_3}{2}}\delta_0(\xi)^{\sigma_3},&k\in\C{S}_1\cup\C{S}_2\cup\C{S}_3\cup\C{S}_4\cup\C{S}_5,\\
I,&\text{else},
\end{cases}\quad(x,t)\in\D{S}.
\end{equation}
Then
\begin{align}   \label{eq:vbeta2}
&v_{\C{Y}_1}^{(\beta2)}=\begin{pmatrix}1&0\\\frac{rr^*}{1+rr^*}\lambda^{2\ii\tilde\nu}\delta_1^{-2}\eul^{2\ii tg_{\beta}}&1\end{pmatrix},\qquad\qquad\quad v_{\C{Y}_2}^{(\beta2)}=\begin{pmatrix}1&-\lambda^{-2\ii\tilde\nu}\delta_1^2\eul^{-\ii t(g_{\beta+}+g_{\beta-})}\\0&1\end{pmatrix},\notag\\
&v_{\C{Y}_3}^{(\beta2)}=\begin{pmatrix}1&0\\rr^*\lambda^{2\ii\tilde\nu}\delta_1^{-2}\eul^{2\ii tg_{\beta}}&1\end{pmatrix},\qquad\qquad\qquad v_{\C{Y}_4}^{(\beta2)}=\begin{pmatrix}1&-(1+rr^*)\lambda^{-2\ii\tilde\nu}\delta_1^2\eul^{-2\ii tg_{\beta}}\\0&1\end{pmatrix},\notag\\
&v_{\C{Y}_5}^{(\beta2)}=\lambda^{-\ii\tilde\nu\sigma_3}\delta_1^{\sigma_3},\qquad\qquad\qquad\qquad\qquad\qquad v_{\C{Y}_6}^{(\beta2)}=\lambda^{-\ii\tilde\nu\sigma_3}\delta_1^{\sigma_3}\begin{pmatrix}1&0\\-\frac{rr^*}{1+rr^*}\eul^{2\ii tg_{\beta}}&1\end{pmatrix},\notag\\
&v_{\C{Y}_7}^{(\beta2)}=\lambda^{-\ii\tilde\nu\sigma_3}\delta_1^{\sigma_3}\begin{pmatrix}\frac{1}{1+rr^*}&\eul^{-2\ii tg_{\beta}}\\-\frac{rr^*}{1+rr^*}\eul^{2\ii tg_{\beta}}&1\end{pmatrix},\quad v_{\C{Y}_8}^{(\beta2)}=\lambda^{-\ii\tilde\nu\sigma_3}\delta_1^{\sigma_3}\begin{pmatrix}\frac{1}{1+rr^*}&\eul^{-2\ii tg_{\beta}}\\0&1+rr^*\end{pmatrix},\notag\\
&v_{\C{Y}_9}^{(\beta2)}=\lambda^{-\ii\tilde\nu\sigma_3}\delta_1^{\sigma_3}\begin{pmatrix}1&(1+rr^*)\eul^{-2\ii tg_{\beta}}\\0&1\end{pmatrix},\notag\\
&v_{\C{Y}_{11}}^{(\beta2)}=\begin{pmatrix}\frac{1+rr^*\eul^{\ii t(g_{\beta+}-g_{\beta-})}}{1+rr^*}&\eul^{-2\ii tg_{\beta+}}(1-\eul^{\ii t(g_{\beta+}-g_{\beta-})})\\\frac{rr^*\eul^{2\ii tg_{\beta+}}(\eul^{-2\ii t(g_{\beta+}-g_{\beta-})}-1+rr^*(\eul^{-\ii t(g_{\beta+}-g_{\beta-})}-1))}{(1+rr^*)^2}&1-\frac{rr^*\eul^{-\ii t(g_{\beta+}-g_{\beta-})}(1-\eul^{-\ii t(g_{\beta+}-g_{\beta-})})}{1+rr^*}\end{pmatrix},\notag\\
&v_{\C{Y}_{14}}^{(\beta2)}=\begin{pmatrix}\eul^{\ii t(g_{\beta+}-g_{\beta-})}&0\\0&\eul^{-\ii t(g_{\beta+}-g_{\beta-})}\end{pmatrix},\qquad\quad v_{\C{Y}_{10}}^{(\beta2)}=v_{\C{Y}_{12}}^{(\beta2)}=v_{\C{Y}_{13}}^{(\beta2)}=v_{\C{Y}_{15}}^{(\beta2)}=v_{\C{Y}_{16}}^{(\beta2)}=I.
\end{align}

Let $\tilde v^Y$ denote the jump matrix $v^Y$ defined in \eqref{eq:vY} pulled back to $\C{Y}$ via the map $k\mapsto\lambda$, that is,
\begin{equation}   \label{eq:vtildeY}
\tilde v^Y(x,t,k)=v^Y(y(x,t),\tilde\nu(\xi),\lambda(x,t,k)),\quad k\in\C{Y},\ (x,t)\in\D{S},
\end{equation}
where we assume that $v^Y(y,\tilde\nu,\lambda)$ is the identity matrix whenever $\lambda\notin Y$ and where $\lambda(x,t,k)$ and $y(x,t)$ are defined in \eqref{eq:lambda} and \eqref{eq:y}, respectively. Define $\rho\equiv\rho(\xi)$ and $\rho^*\equiv\rho^*(\xi)$ by
\[
\rho=r(\beta)\quad\text{and}\quad\rho^*=r^*(\beta).
\]
For a fixed $\lambda$, $r(k)\to\rho$ and $\delta_1(\xi,k)\to 1$ as $t\to+\infty$. This suggests that $v^{(\beta2)}$ tends to $\tilde v^Y$ for large $t$. In the following subsection, we make this statement precise.
\subsubsection{Estimates of $v^{(\beta 2)}-\tilde v^Y$}

First an auxiliary lemma.

\begin{lemma}   \label{lem:egb}
Shrinking $\epsilon>0$ if necessary, we have 
\begin{equation}   \label{eq:egbestim}
\begin{cases}
\abs{\eul^{2\ii tg_{\beta}(k)}}\leq C\eul^{-ct\abs{k-\beta}^2},&k\in\C{Y}_1\cup\C{Y}_3,\\
\abs{\eul^{-2\ii tg_{\beta}(k)}}\leq C\eul^{-ct\abs{k-\beta}^2},&k\in\C{Y}_2\cup\C{Y}_4,
\end{cases}\quad(x,t)\in\D{S},
\end{equation}
for some constants $c,\,C>0$ independent of $k$, $x$, $t$.
\end{lemma}

\begin{proof}
It follows from \eqref{eq:gbestim1} that
\[
\abs{\eul^{2\ii tg_{\beta}}}\leq\abs{\eul^{-\lambda^2-2y\lambda}}\eul^{Ct\abs{\alpha-\beta}^2\abs{\ln\abs{\alpha-\beta}}}\eul^{Ct\abs{k-\beta}^3},\quad k\in\C{Y}_1\cup\C{Y}_3,\ (x,t)\in\D{S}.
\]
Since $y=\ord(\sqrt{t}\abs{\alpha-\beta})=\ord(t^{-\delta})$ tends uniformly to zero as $t\to+\infty$, we have (increasing $T$ in \eqref{eq:wedge} if necessary)
\begin{equation}   \label{eq:elambdaestim}
\begin{cases}
\abs{\eul^{-\lambda^2-2y\lambda}}\leq C\eul^{-\frac{\abs{\lambda}^2}{2}},&\lambda\in Y_1\cup Y_3,\\
\abs{\eul^{\lambda^2+2y\lambda}}\leq C\eul^{-\frac{\abs{\lambda}^2}{2}},&\lambda\in Y_2\cup Y_4,
\end{cases}\quad(x,t)\in\D{S}.
\end{equation}
Moreover, for every $d>0$ sufficiently small, we have
\begin{align}   \label{eq:dsmallestim}
t\abs{\alpha-\beta}^2\abs{\ln\abs{\alpha-\beta}}
&=t\abs{\alpha-\beta}^{2-d}\times\abs{\alpha-\beta}^d\abs{\ln\abs{\alpha-\beta}}\notag\\
&\leq t\abs{\alpha-\beta}^{2-d}\leq Ct^{-c},\quad(x,t)\in\D{S}.
\end{align}
Hence
\[
\abs{\eul^{2\ii tg_{\beta}}}\leq C\eul^{-\frac{\abs{\lambda}^2}{2}}\eul^{Ct\abs{k-\beta}^3},\quad k\in\C{Y}_1\cup\C{Y}_3,\ (x,t)\in\D{S}.
\]
From the definition \eqref{eq:lambda}, we have
\[
\abs{\lambda}\leq C\sqrt{t}\,\abs{k-\beta}\leq C\abs{\lambda},\quad k\in D_{\epsilon}(\beta),\ (x,t)\in\D{S}.
\]
Hence there exists a $C_1>0$ such that
\[
\abs{\eul^{2\ii tg_{\beta}}}\leq C\eul^{-C_1t\abs{k-\beta}^2}\eul^{Ct\abs{k-\beta}^3}=C\eul^{-t\abs{k-\beta}^2(C_1-C\abs{k-\beta})},\quad k\in\C{Y}_1\cup\C{Y}_3,\ (x,t)\in\D{S}.
\]
Shrinking $\epsilon>0$ if necessary, we may assume that $C_1-C\epsilon> 0$. This proves the lemma in the case of $k\in\C{Y}_1\cup\C{Y}_3$; a similar proof applies when $k\in\C{Y}_2\cup\C{Y}_4$.
\end{proof}

\begin{lemma}   \label{lem:vb2Yapprox}
As $t\to+\infty$, $v^{(\beta2)}$ approaches the jump matrix $\tilde v^Y$ defined in \eqref{eq:vtildeY} in the sense that
\begin{alignat}{2}     \label{eq:vb2Y1}
&\norm{v^{(\beta2)}(x,t,\,\cdot\,)-\tilde v^Y(x,t,\,\cdot\,)}_{L^1(\C{Y})}\leq Ce_1(x,t)&\quad&(x,t)\in\D{S},\\
\label{eq:vb2Yinfty}
&\norm{v^{(\beta2)}(x,t,\,\cdot\,)-\tilde v^Y(x,t,\,\cdot\,)}_{L^{\infty}(\C{Y})}\leq Ce_{\infty}(x,t)&&(x,t)\in\D{S},
\end{alignat}
where $e_1$ and $e_{\infty}$ are given by
\begin{align*}
e_{\infty}(x,t)&\coloneqq t(\ln t)^{\abs{\Im\tilde\nu}}\abs{\alpha-\beta}^2\abs{\ln\abs{\alpha-\beta}}+t^{-1/2}(\ln t)^{\abs{\Im\tilde\nu}+\frac{3}{2}},\\
e_1(x,t)&\coloneqq e_{\infty}(x,t)\sqrt{\frac{\ln t}{t}}\,.
\end{align*}
\end{lemma}

\begin{proof}
We will show that the following estimates hold uniformly for $(x,t)\in\D{S}$:
\begin{subequations}
\begin{alignat}{2}  \label{eq:vb2Y1j}
&\norm{v^{(\beta2)}(x,t,\,\cdot\,)-\tilde v^Y(x,t,\,\cdot\,)}_{L^1(\C{Y}_j)}\leq Ce_1(x,t),&\quad&j=1,\dots,9,\\ \label{eq:vb2Y111}
&\norm{v^{(\beta2)}(x,t,\,\cdot\,)-I}_{L^1(\C{Y}_{11})}\leq C\sqrt{t}\,\abs{\alpha-\beta}^2,\\ \label{eq:vb2Y114}
&\norm{v^{(\beta2)}(x,t,\,\cdot\,)-I}_{L^1(\C{Y}_{14})}\leq Ct\abs{\alpha-\beta}^3,
\end{alignat}
\end{subequations}
and
\begin{subequations}
\begin{alignat}{2}  \label{eq:vb2Yinftyj}
&\norm{v^{(\beta2)}(x,t,\,\cdot\,)-\tilde v^Y(x,t,\,\cdot\,)}_{L^{\infty}(\C{Y}_j)}\leq Ce_{\infty}(x,t),&\quad&j=1,\dots,9,\\ \label{eq:vb2Yinfty114}
&\norm{v^{(\beta2)}(x,t,\,\cdot\,)-I}_{L^{\infty}(\C{Y}_{11}\cup\C{Y}_{14})}\leq Ct\abs{\alpha-\beta}^2.
\end{alignat}
\end{subequations}
This will complete the proof.

We begin by proving \eqref{eq:vb2Yinftyj} for $j=1$. Only the $(21)$ entry of $v^{(\beta2)}-\tilde v^Y$ is nonzero for $k\in\C{Y}_1$, so we find
\begin{align}   \label{eq:vb2Y}
\abs{v^{(\beta2)}-\tilde v^Y}
&=\left\lvert\frac{rr^*}{1+rr^*}\lambda^{2\ii\tilde\nu}\delta_1^{-2}\eul^{2\ii tg_{\beta}}-\frac{\rho\rho^*}{1+\rho\rho^*}\lambda^{2\ii\tilde\nu}\eul^{-\lambda^2-2y\lambda}\right\rvert\notag\\
&\leq\abs{\lambda^{2\ii\tilde\nu}}\left(\left\lvert\frac{rr^*}{1+rr^*}-\frac{\rho\rho^*}{1+\rho\rho^*}\right\rvert\abs{\delta_1^{-2}\eul^{2\ii tg_{\beta}}}+\left\lvert\frac{\rho\rho^*}{1+\rho\rho^*}\right\rvert\abs{\delta_1^{-2}-1}\abs{\eul^{2\ii tg_{\beta}}}\right.\notag\\
&\qquad\qquad\left.+\left\lvert\frac{\rho\rho^*}{1+\rho\rho^*}\right\rvert\abs{\eul^{2\ii tg_{\beta}}-\eul^{-\lambda^2-2y\lambda}}\right),\quad k\in\C{Y}_1,\ (x,t)\in\D{S}.
\end{align}
Standard estimates show that
\[
\abs{\tilde\chi(\xi,k)-\tilde\chi(\xi,\beta)}\leq C\abs{k-\beta}\,\abs{\ln\abs{k-\beta}},\quad k\in D_{\epsilon}(\beta)\setminus\gamma_{(\mu,\beta)},\ (x,t)\in\D{S}.
\]
Using the general inequality
\begin{equation}   \label{eq:ew}
\abs{\eul^w-1}\leq\abs{w}\max(1,\eul^{\Re w}),\quad w\in\D{C},
\end{equation}
this yields
\begin{equation}   \label{eq:delta1}
\abs{\delta_1(\xi,k)-1}\leq C\abs{k-\beta}\,\abs{\ln\abs{k-\beta}},\quad k\in D_{\epsilon}(\beta)\setminus\gamma_{(\mu,\beta)},\ (x,t)\in\D{S}.
\end{equation}
On the other hand, the smoothness of $r(k)$ implies
\begin{equation}   \label{eq:restim}
\abs{r(k)-r(\beta)}\leq C\abs{k-\beta},\quad k\in D_{\epsilon}(\beta).
\end{equation}
Using these inequalities, we see that \eqref{eq:vb2Y} implies
\begin{equation}  \label{eq:vb2tY}
\abs{v^{(\beta2)}-\tilde v^Y}\leq C(F_1+F_2),\quad k\in\C{Y}_1,\ (x,t)\in\D{S},
\end{equation}
where the functions $F_j\equiv F_j(x,t,k)$, $j=1,2$, are defined by
\begin{align*}
F_1&=\abs{\lambda}^{-2\Im\tilde\nu}\abs{k-\beta}\,\abs{\ln\abs{k-\beta}}\abs{\eul^{2\ii tg_{\beta}}},\\
F_2&=\abs{\lambda}^{-2\Im\tilde\nu}\abs{\eul^{2\ii tg_{\beta}}-\eul^{-\lambda^2-2y\lambda}}.
\end{align*}

We first consider $F_1$. Employing \eqref{eq:egbestim} and using that $\sqrt{t}\,\abs{k-\beta}\geq C>0$ on $\C{Y}_1$ we obtain
\begin{align*}
\abs{F_1}&\leq C\left(\sqrt{t}\abs{k-\beta}\right)^{-2\Im\tilde\nu}\abs{k-\beta}\,\abs{\ln\abs{k-\beta}}\,\eul^{-ct\abs{k-\beta}^2}\\
&\leq\frac{C}{\sqrt{t}}\left(\sqrt{t}\abs{k-\beta}\right)^{1-2\Im\tilde\nu}\abs{\ln\abs{k-\beta}}\eul^{-ct\abs{k-\beta}^2}\\
&\leq\frac{C}{\sqrt{t}}\abs{\ln\abs{k-\beta}}\eul^{-\frac{c}{2}t\abs{k-\beta}^2},\quad k\in\C{Y}_1,\ (x,t)\in\D{S}.
\end{align*}
It follows that
\begin{subequations}   \label{eq:F11inftyY1}
\begin{equation}   \label{eq:F11Y1}
\norm{F_1}_{L^1(\C{Y}_1)}\leq\frac{C}{\sqrt{t}}\int_0^{\epsilon}\abs{\ln u}\eul^{-\frac{c}{2}tu^2}\dd u\leq Ct^{-1}\ln t,\quad (x,t)\in\D{S},
\end{equation}
and
\begin{equation}  \label{eq:F1inftyY1}
\norm{F_1}_{L^{\infty}(\C{Y}_1)}\leq Ct^{-1/2}\ln t,\quad (x,t)\in\D{S},
\end{equation}
\end{subequations}

We now consider $F_2$. On the one hand, the estimates \eqref{eq:egbestim} and \eqref{eq:elambdaestim} give
\begin{equation}    \label{eq:gblambda1}
\abs{\eul^{2\ii tg_{\beta}}-\eul^{-\lambda^2-2y\lambda}}\leq\abs{\eul^{2\ii tg_{\beta}}}+\abs{\eul^{-\lambda^2-2y\lambda}}\leq C\eul^{-ct\abs{k-\beta}^2},\quad k\in\C{Y}_1,\ (x,t)\in\D{S}.
\end{equation}
On the other hand, \eqref{eq:gbestim1}, \eqref{eq:dsmallestim}, and \eqref{eq:ew} yield
\begin{align}   \label{eq:gblambda2}
\abs{\eul^{2\ii tg_{\beta}}-\eul^{-\lambda^2-2y\lambda}}
&\leq\abs{\eul^{-\lambda^2-2y\lambda}}\abs{\eul^{\lambda^2+2y\lambda+2\ii tg_{\beta}}-1}\notag\\
&\leq Ct\left(\abs{\alpha-\beta}^2\abs{\ln\abs{\alpha-\beta}}+\abs{k-\beta}^3\right)\eul^{Ct\abs{k-\beta}^3},\quad k\in\C{Y}_1,\ (x,t)\in\D{S}.
\end{align}
Let $d>0$. On the part of $\C{Y}_1$ on which $\abs{k-\beta}\geq\sqrt{\frac{d\ln t}{ct}}$, we use the estimate \eqref{eq:gblambda1} to find
\begin{equation}   \label{eq:F2estim1}
\abs{F_2}\leq C(\sqrt{t}\abs{k-\beta})^{-2\Im\tilde\nu}t^{-d}\leq Ct^{\abs{\Im\tilde\nu}-d},\quad k\in\C{Y}_1,\ \abs{k-\beta}\geq\sqrt{\frac{d\ln t}{ct}}.
\end{equation}
On the part of $\C{Y}_1$ on which $\abs{k-\beta}\leq\sqrt{\frac{d\ln t}{ct}}$, we use the estimate \eqref{eq:gblambda2} to find
\begin{align}   \label{eq:F2estim2}
\abs{F_2}&\leq C\abs{\lambda}^{-2\Im\tilde\nu}t\left(\abs{\alpha-\beta}^2\abs{\ln\abs{\alpha-\beta}}+\abs{k-\beta}^3\right)\notag\\
&\leq C(\ln t)^{\abs{\Im\tilde\nu}}\left(t\abs{\alpha-\beta}^2\abs{\ln\abs{\alpha-\beta}}+\abs{k-\beta}^3\right),
\quad k\in\C{Y}_1,\ \abs{k-\beta}\leq\sqrt{\frac{d\ln t}{ct}}.
\end{align}
Choosing $d$ larger than $\sup_{(x,t)\in\D{S}}\abs{\Im\tilde\nu}+1$, it follows from \eqref{eq:F2estim1} and \eqref{eq:F2estim2} that 
\begin{subequations}   \label{eq:F21inftyY1}
\begin{align}    \label{eq:F21Y1}
&\norm{F_2}_{L^1(\C{Y}_1)}\notag\\
&\quad\leq Ct^{\abs{\Im\tilde\nu}-d}+C(\ln t)^{\abs{\Im\tilde\nu}}t\abs{\alpha-\beta}^2\abs{\ln\abs{\alpha-\beta}}\left(\sqrt{\frac{d\ln t}{ct}}-R\right)+C(\ln t)^{\abs{\Im\tilde\nu}}t\int_R^{\sqrt{\frac{d\ln t}{ct}}}u^3\dd u\notag\\
&\quad\leq Ct^{-1}+C(\ln t)^{\abs{\Im\tilde\nu}+\frac{1}{2}}\sqrt{t}\,\abs{\alpha-\beta}^2\abs{\ln\abs{\alpha-\beta}}+C(\ln t)^{\abs{\Im\tilde\nu}}t^{-1}\left((\ln t)^2+C\right)\notag\\
&\quad\leq C\sqrt{t}\,(\ln t)^{\abs{\Im\tilde\nu}+\frac{1}{2}}\,\abs{\alpha-\beta}^2\abs{\ln\abs{\alpha-\beta}}+Ct^{-1}(\ln t)^{\abs{\Im\tilde\nu}+2},\quad (x,t)\in\D{S},
\end{align}
and that
\begin{equation}     \label{eq:F2inftyY1}
\norm{F_2}_{L^{\infty}(\C{Y}_1)}\leq Ct(\ln t)^{\abs{\Im\tilde\nu}}\,\abs{\alpha-\beta}^2\abs{\ln\abs{\alpha-\beta}}+Ct^{-1/2}(\ln t)^{\abs{\Im\tilde\nu}+\frac{3}{2}},\quad (x,t)\in\D{S}.
\end{equation}
\end{subequations}
Combining equations \eqref{eq:vb2tY}, \eqref{eq:F11inftyY1}, and \eqref{eq:F21inftyY1}, we obtain \eqref{eq:vb2Y1j} and \eqref{eq:vb2Yinftyj} for $j=1$; the proofs for $j=2,3,4$ are similar.

We next prove \eqref{eq:vb2Yinftyj} and \eqref{eq:vb2Y1j} for $j=6$. We have
\begin{equation}    \label{eq:vb2Y6}
v^{(\beta2)}-\tilde v^Y=\lambda^{-\ii\tilde\nu\sigma_3}\Biggl(\delta_1^{\sigma_3}\begin{pmatrix}1&0\\-\frac{rr^*}{1+rr^*}\eul^{2\ii tg_{\beta}}&1\end{pmatrix}-\begin{pmatrix}1&0\\-\frac{\rho\rho^*}{1+\rho\rho^*}\eul^{-\lambda^2-2y\lambda}&1\end{pmatrix}\Biggr),\quad k\in\C{Y}_6,\ (x,t)\in\D{S}.
\end{equation}
Using \eqref{eq:delta1} and the facts that $\abs{\lambda}=1$ and $\abs{k-\beta}=\ord(t^{-1/2})$ on $\C{Y}_6$, we see that the absolute value of the $(11)$ element in \eqref{eq:vb2Y6} is bounded above by
\[
\abs{\lambda^{-\ii\tilde\nu}}\abs{\delta_1-1}\leq C\abs{k-\beta}\,\abs{\ln\abs{k-\beta}}\leq Ct^{-1/2}\ln t,\quad k\in\C{Y}_6,\ (x,t)\in\D{S}.
\]
The $(22)$ element satisfies a similar estimate. On the other hand, the $(21)$ element in \eqref{eq:vb2Y6} is bounded above by
\begin{align*}
&\abs{\lambda^{\ii\tilde{\nu}}}\left\lvert\frac{rr^*}{1+rr^*}\delta_1^{-1}\eul^{2\ii tg_{\beta}}-\frac{\rho\rho^*}{1+\rho\rho^*}\eul^{-\lambda^2-2y\lambda}\right\rvert\\
&\qquad\qquad\leq C\left\lvert\frac{rr^*}{1+rr^*}-\frac{\rho\rho^*}{1+\rho\rho^*}\right\rvert\abs{\delta_1^{-1}\eul^{2\ii tg_{\beta}}}
+C\left\lvert\frac{\rho\rho^*}{1+\rho\rho^*}\right\rvert\,\abs{\delta_1^{-1}-1}\,\abs{\eul^{2\ii tg_{\beta}}}\\
&\qquad\qquad\quad+C\left\lvert\frac{\rho\rho^*}{1+\rho\rho^*}\right\rvert\,\abs{\eul^{2\ii tg_{\beta}}-\eul^{-\lambda^2-2y\lambda}},\qquad k\in\C{Y}_6,\ (x,t)\in\D{S}.
\end{align*}
Recall that $y=\ord(t^{-\delta})$ and note that $\abs{\lambda}=1$ for $\lambda\in\C{Y}_6$. Hence we have the following analog of \eqref{eq:gblambda2}:
\begin{align}
\abs{\eul^{2\ii tg_{\beta}}-\eul^{-\lambda^2-2y\lambda}}
&\leq C\abs{\eul^{\lambda^2+2y\lambda+2\ii tg_{\beta}}-1}\notag\\
&\leq Ct\abs{\alpha-\beta}^2\abs{\ln\abs{\alpha-\beta}}+Ct^{-1/2},\quad k\in\C{Y}_6,\ (x,t)\in\D{S}.
\end{align}
Using \eqref{eq:delta1}, we conclude that the $(21)$ element in \eqref{eq:vb2Y6} is bounded above by
\begin{align*}
&C\abs{k-\beta}\abs{\ln\abs{k-\beta}}+Ct\abs{\alpha-\beta}^2\abs{\ln\abs{\alpha-\beta}}+Ct^{-1/2}\\
&\qquad\leq Ct^{-1/2}\ln t+Ct\abs{\alpha-\beta}^2\abs{\ln\abs{\alpha-\beta}},\quad k\in\C{Y}_6,\ (x,t)\in\D{S}.
\end{align*}
Hence, since $\C{Y}_6$ has length of order $\ord(t^{-1/2})$, we arrive at
\begin{align*}
\norm{v^{(\beta2)}-\tilde v^Y}_{L^1(\C{Y}_6)}&\leq Ct^{-1}\ln t+C\sqrt{t}\,\abs{\alpha-\beta}^2\abs{\ln\abs{\alpha-\beta}},\\
\norm{v^{(\beta2)}-\tilde v^Y}_{L^{\infty}(\C{Y}_6)}&\leq Ct^{-\frac{1}{2}}\ln t+Ct\,\abs{\alpha-\beta}^2\abs{\ln\abs{\alpha-\beta}}. 
\end{align*}
This proves \eqref{eq:vb2Y1j} and \eqref{eq:vb2Yinftyj} for $j=6$; the proofs for $j=5,7,8,9$ are similar.

Since $\C{Y}_{11}$ has length of order $\ord(t^{-1/2})$ and $\C{Y}_{14}$ has length of order $\ord(\abs{\alpha-\beta})$, equations \eqref{eq:vb2Y111} and \eqref{eq:vb2Y114} are a direct consequence of \eqref{eq:vb2Yinfty114}. It therefore only remains to prove \eqref{eq:vb2Yinfty114}.

According to \eqref{eq:gbestim2} we have
\[
t\abs{g_{\beta+}(k)-g_{\beta-}(k)}\leq Ct\abs{\alpha-\beta}^2\leq C,\quad k\in\C{Y}_{11}\cup\C{Y}_{14},\ (x,t)\in\D{S}.
\]
Thus, by \eqref{eq:ew},
\[
\abs{\eul^{\ii t(g_{\beta+}(k)-g_{\beta-}(k))}-1}\leq Ct\abs{\alpha-\beta}^2,\quad k\in\C{Y}_{11}\cup\C{Y}_{14},\ (x,t)\in\D{S}.
\]
The estimate \eqref{eq:vb2Yinfty114} now follows from the expressions for $v_{\C{Y}_{11}}^{(\beta2)}$ and $v_{\C{Y}_{14}}^{(\beta2)}$ given in \eqref{eq:vbeta2}.
\end{proof}

Let $\tilde m^Y$ denote the solution $m^Y$ of Appendix~\ref{sec:app} pulled back to $D_{\epsilon}(\beta)$ via the map $k\mapsto\lambda$, that is,
\begin{equation}     \label{eq:mtildeY}
\tilde m^Y(x,t,k)=m^Y\bigl(y(x,t),\tilde\nu(\xi),\lambda(x,t,k)\bigr),\quad k\in D_{\epsilon}(\beta)\setminus\C{Y},\ (x,t)\in\D{S},
\end{equation}
where $\tilde\nu(\xi)$, $\lambda(x,t,k)$, and $y(x,t)$ are defined in \eqref{eq:nutilde}, \eqref{eq:lambda}, and \eqref{eq:y}, respectively.

We have $m^{(\beta2)}=\hat m^{(6)}G$ where
\[
G(x,t,k)\coloneqq\eul^{-\ii\tilde h\sigma_3}ABD
\]
and the matrix-valued functions $A,B,D$ are defined in \eqref{eq:A}, \eqref{eq:B}, and \eqref{eq:D}, respectively. Lemma~\ref{lem:vb2Yapprox} shows that the jumps of $m^{(\beta2)}=\hat m^{(6)}G$ across $\C{Y}$ approach those of $\tilde m^Y$. In other words, as $t\to+\infty$, the jumps of $\hat m^{(6)}$ approach those of the function $\tilde m^YG^{-1}$. This suggests that we approximate $\hat m^{(6)}$ in $D_{\epsilon}(\beta)$ by a $2\times 2$-matrix valued function $m^{\beta}$ of the form
\begin{equation}    \label{eq:mbeta}
m^{\beta}\coloneqq Y_{\beta}t^{\frac{\ii\tilde\nu}{2}\sigma_3}\tilde m^YG^{-1},
\end{equation}
where $Y_{\beta}(x,t,k)$ is a function which is analytic for $k\in D_{\epsilon}(\beta)$ and we have included the $k$-independent factor $t^{\frac{\ii\tilde\nu}{2}\sigma_3}$ in order to make $Y_{\beta}$ of order $\ord(1)$. To ensure that $m^{\beta}$ is a good approximation of $\hat m^{(6)}$ for large $t$, we want to choose $Y_{\beta}(k)$ so that $m^{\beta}(m^{\model})^{-1}\to I$ on $\partial D_{\epsilon}(\beta)$ as $t\to+\infty$. Now
\[
G=\eul^{-\ii\tilde h\sigma_3}(\delta^2\hat a\hat b)^{\sigma_3/2}t^{\frac{\ii\tilde\nu}{2}\sigma_3}\delta_0^{\sigma_3}\times\begin{cases}\eul^{-\ii tg_+(\beta)\sigma_3}&k\in(D_{\epsilon}(\beta)\setminus D_{R}(\beta))\cap\C{S}_+,\\
\eul^{-\ii tg_-(\beta)\sigma_3}&k\in(D_{\epsilon}(\beta)\setminus D_{R}(\beta))\cap\C{S}_-.
\end{cases}
\]
We therefore define (we arbitrarily choose to define $Y_{\beta}$ using the expression involving $g_+(\beta)$; we could equally well have used the expression involving $g_-(\beta)$)
\begin{equation}    \label{eq:Ybeta}
Y_{\beta}(x,t,k)\coloneqq m^{\model}(x,t,k)\eul^{-\ii\tilde h\sigma_3}(\delta^2\hat a\hat b)^{\sigma_3/2}\delta_0^{\sigma_3}\eul^{-\ii tg_+(\beta)\sigma_3}.
\end{equation}

\begin{lemma}   \label{lem:mbeta}
For each $(x,t)\in\D{S}$, the function $m^{\beta}(x,t,k)$ defined in \eqref{eq:mbeta} is an analytic function of $k\in D_{\epsilon}(\beta)\setminus\C{Y}$ and the function $Y_{\beta}(x,t,k)$ defined in \eqref{eq:Ybeta} is an analytic function of $k\in D_{\epsilon}(\beta)$. Moreover, we have the uniform estimate
\begin{equation}
\label{eq:Ybbound}
\abs{Y_{\beta}(x,t,k)}\leq C,\quad k\in D_{\epsilon}(\beta),\quad(x,t)\in\D{S}.
\end{equation}
Across $\C{Y}$, $m^{\beta}$ obeys the jump condition $m_+^{\beta}=m_-^{\beta}v^{\beta}$, where the jump matrix $v^{\beta}$ satisfies
\begin{subequations} \label{eq:v6b}
\begin{alignat}{2}   \label{eq:v6binfty}
&\norm{m_-^{\beta}(\hat v^{(6)}-v^{\beta})(m_+^{\beta})^{-1}}_{L^{\infty}(\C{Y})}\leq CE_{\infty}(x,t),&\quad&(x,t)\in\D{S},\\
\label{eq:v6b1}
&\norm{m_-^{\beta}(\hat v^{(6)}-v^{\beta})(m_+^{\beta})^{-1}}_{L^1(\C{Y})}\leq CE_1(x,t),&&(x,t)\in\D{S},
\end{alignat}
\end{subequations} 
with the functions $E_1$ and $E_{\infty}$ defined by
\begin{align}
E_{\infty}(x,t)&\coloneqq t^{\abs{\Im\tilde\nu}}e_{\infty}(x,t)\notag\\
&=t^{\abs{\Im\tilde\nu}+1}(\ln t)^{\abs{\Im\tilde\nu}}\abs{\alpha-\beta}^2\abs{\ln\abs{\alpha-\beta}}+t^{\abs{\Im\tilde\nu}-\frac{1}{2}}(\ln t)^{\abs{\Im\tilde\nu}+\frac{3}{2}},\\
E_1(x,t)&\coloneqq E_{\infty}(x,t)\sqrt{\frac{\ln t}{t}}.
\end{align}
Furthermore, on $\partial D_{\epsilon}(\beta)$ the quotient $m^{\model}(m^{\beta})^{-1}$ satisfies
\begin{align}    \label{eq:mmodbbound}
&\left\lvert\frac{1}{2\pi\ii}\int_{\partial D_{\epsilon}(\beta)}\left(m^{\model}(m^{\beta})^{-1}-I\right)\dd k-\frac{Y_{\beta}(x,t,\beta)t^{\frac{\ii\tilde\nu}{2}\sigma_3}m_1^Y(y,\tilde\nu)t^{-\frac{\ii\tilde\nu}{2}\sigma_3}Y_{\beta}(x,t,\beta)^{-1}}{\sqrt{t}\,\psi_{\beta}(\xi)}\right\rvert\notag\\
&\qquad\leq Ct^{\abs{\Im\tilde\nu}}\left(t^{-1}+t\abs{\alpha-\beta}^2\right),\qquad (x,t)\in\D{S},
\end{align}
and
\begin{equation}   \label{eq:mmodbinfty}
\norm{m^{\model}(m^{\beta})^{-1}-I}_{L^{\infty}(\partial D_{\epsilon}(\beta))}\leq Ct^{\abs{\Im\tilde\nu}-\frac{1}{2}}+Ct^{\abs{\Im \tilde{\nu}}+1}\abs{\alpha-\beta}^2,\quad(x,t)\in\D{S},
\end{equation}
where $m_1^Y(y,\tilde\nu)$ is defined in \eqref{eq:m1Y}.
\end{lemma}

\begin{proof}
The analyticity properties of $m^{\beta}$ and $Y_{\beta}$ are immediate. Using that $\Im g_+(\beta)=0$ (see \cite{BLS22}*{Section 3}), the bound \eqref{eq:Ybbound} on $Y_{\beta}$ follows from the definition \eqref{eq:Ybeta}. 

We next establish the estimates \eqref{eq:v6b}. We have
\[
\hat v^{(6)}-v^{\beta}=G_-\left(v^{(\beta2)}-\tilde v^Y\right)G_+^{-1},\quad k\in\C{Y},\ (x,t)\in\D{S},
\]
and so
\[
m_-^{\beta}(\hat v^{(6)}-v^{\beta})(m_+^{\beta})^{-1}=Y_{\beta}\,t^{\frac{\ii \tilde\nu}{2}\sigma_3}\tilde m_-^Y(v^{(\beta2)}-\tilde v^Y)(\tilde m_+^Y)^{-1}t^{-\frac{\ii\tilde\nu}{2}\sigma_3}Y_{\beta}^{-1}.
\]
In view of the bounds \eqref{eq:Ybbound} and \eqref{eq:mYbound} on $Y_{\beta}$ and $m^Y$, this gives
\[
\abs{m_-^{\beta}(\hat v^{(6)}-v^{\beta})(m_+^{\beta})^{-1}}\leq Ct^{\abs{\Im\tilde\nu}}\abs{v^{(\beta2)}-\tilde v^Y},\quad k\in\C{Y},\ (x,t)\in\D{S}.
\]
The estimates \eqref{eq:v6b} now follow from Lemma~\ref{lem:vb2Yapprox}.

It remains to prove \eqref{eq:mmodbbound} and \eqref{eq:mmodbinfty}. Note that 
\begin{align*}
&m^{\model}(m^{\beta})^{-1}-I=Y_{\beta}t^{\frac{\ii\tilde\nu}{2}\sigma_3}\left((\tilde m^Y)^{-1}-I\right)t^{-\frac{\ii\tilde\nu}{2}\sigma_3}Y_{\beta}^{-1}\\
&\qquad\quad+\begin{cases}
0,&k\in\partial D_{\epsilon}(\beta)\cap\C{S}_+,\\
Y_{\beta}t^{\frac{\ii\tilde\nu}{2}\sigma_3}\left(\eul^{\ii t(g_+(\beta)-g_-(\beta))\sigma_3}-I\right)(\tilde m^Y)^{-1}t^{-\frac{\ii\tilde\nu}{2}\sigma_3}Y_{\beta}^{-1},&k\in\partial D_{\epsilon}(\beta)\cap\C{S}_-.
\end{cases}
\end{align*}
Since $\inf_{(x,t)\in\D{S}}\abs{\psi_{\beta}(\xi)}>0$, the variable $\lambda=\sqrt{t}(k-\beta)\psi_{\beta}(\xi)$ goes to infinity as $t\to+\infty$ whenever $k\in\partial D_{\epsilon}(\beta)$. Thus equation \eqref{eq:mYas} yields (see \eqref{eq:mtildeY})
\[
\tilde m^Y(x,t,k)=I+\frac{m_1^Y(y,\tilde\nu)}{\sqrt{t}(k-\beta)\psi_{\beta}(\xi)}+\ord(t^{-1}),\quad t\to+\infty,\ k\in\partial D_{\epsilon}(\beta),\ (x,t)\in\D{S},
\]
uniformly with respect to $k\in\partial D_{\epsilon}(\beta)$ and $(x,t)\in\D{S}$. Consequently, using \eqref{eq:Ybbound} and \eqref{eq:mYbound},
\begin{align}   \label{eq:mmodbestim}
\left\lvert
m^{\model}(m^{\beta})^{-1}-I+\frac{Y_{\beta}t^{\frac{\ii\tilde\nu}{2}\sigma_3}m_1^Yt^{-\frac{\ii\tilde\nu}{2}\sigma_3}Y_{\beta}^{-1}}{\sqrt{t}(k-\beta)\psi_{\beta}(\xi)}\right\rvert&\leq Ct^{\abs{\Im\tilde\nu}-1}+Ct^{\abs{\Im\tilde\nu}+1}\abs{g_+(\beta)-g_-(\beta)},\notag\\
&\qquad\qquad\qquad k\in\partial D_{\epsilon}(\beta),\ (x,t)\in\D{S}.
\end{align}
But, by \eqref{eq:gbestim2},
\[
\abs{g_+(\beta)-g_-(\beta)}\leq C\abs{\alpha-\beta}^2.
\]
Thus, using that the functions $Y_{\beta}^{\pm1}$ are analytic in $D_{\epsilon}(\beta)$, equation \eqref{eq:mmodbbound} follows from \eqref{eq:mmodbestim} and Cauchy's formula. Since $\abs{\psi_{\beta}(\xi)}^{-1}\leq C$ for $(x,t)\in\D{S}$, \eqref{eq:mmodbinfty} also follows from \eqref{eq:mmodbestim}.
\end{proof}

\section{Final steps}   \label{sec:final}
\subsection{The approximate solution}  \label{sec:approx}

Define a local solution $m^{\bar\beta}$ for $k\in D_{\epsilon}(\bar\beta)$ by
\[
m^{\bar\beta}(x,t,k)=\sigma_3\sigma_1\overline{m^{\beta}(x,t,\bar k)}\sigma_1\sigma_3.
\]
We define an approximate solution $m^{\appr}(x,t,k)$ by
\[
m^{\appr}=\begin{cases}
m^{\beta},&k\in D_{\epsilon}(\beta),\\
m^{\mu},&k\in D_{\epsilon}(\mu),\\
m^{\bar\beta},&k\in D_{\epsilon}(\bar\beta),\\
m^{\model},&\text{elsewhere}.
\end{cases}
\]
\subsection{The solution $\BS{m^{\diff}}$}  \label{sec:mhat}

We will show that the function $m^{\diff}(x,t,k)$ defined by
\[
m^{\diff}=\hat m^{(6)}(m^{\appr})^{-1}
\]
is such that $m^{\diff}-I$ is small for large $t$. Let $\C{D}=D_{\epsilon}(\beta)\cup D_{\epsilon}(\mu)\cup D_{\epsilon}(\bar\beta)$ denote the union of the three open disks in \eqref{eq:threedisks}. The function $m^{\diff}$ satisfies the RH problem
\begin{equation}    \label{eq:rhphat}
\begin{cases}
m^{\diff}(x,t,\,\cdot\,)\in I+\dot E^2(\D{C}\setminus\Sigma^{\diff}),&\\
m^{\diff}_+(x,t,k)=m^{\diff}_-(x,t,k)v^{\diff}(x,t,k)&\text{for a.e. }k\in\Sigma^{\diff},
\end{cases}
\end{equation}
where the contour $\Sigma^{\diff}\coloneqq(\Sigma^{(6)}\setminus\Sigma^{\model})\cup\partial\C{D}$ is displayed in Figure~\ref{fig:hat-contour} and the jump matrix $v^{\diff}$ is given by
\[
v^{\diff}=\begin{cases}
m^{\model}\hat v^{(6)}(m^{\model})^{-1},&k\in\Sigma^{(6)}\setminus\bar{\C{D}},\\
m^{\model}(m^{\beta})^{-1},&k\in\partial D_{\epsilon}(\beta),\\
m^{\model}(m^{\mu})^{-1},&k\in\partial D_{\epsilon}(\mu),\\
m^{\model}(m^{\bar\beta})^{-1},&k\in\partial D_{\epsilon}(\bar\beta),\\
m_-^{\beta}\hat v^{(6)}(m_+^{\beta})^{-1},&k\in\C{Y},\\
m_-^{\mu}\hat v^{(6)}(m_+^{\mu})^{-1},&k\in\C{X},\\
m_-^{\bar\beta}\hat v^{(6)}(m_+^{\bar\beta})^{-1},&k\in\Sigma^{(6)}\cap D_{\epsilon}(\bar\beta)
\end{cases}
\]
with $\C{Y}\coloneqq\Sigma^{(6)}\cap D_{\epsilon}(\beta)$ (see Figure~\ref{fig:jump-contour-Y}) and $\C{X}\coloneqq\Sigma^{(6)}\cap D_{\epsilon}(\mu)$.

\begin{figure}[ht]
\centering\includegraphics[scale=.7]{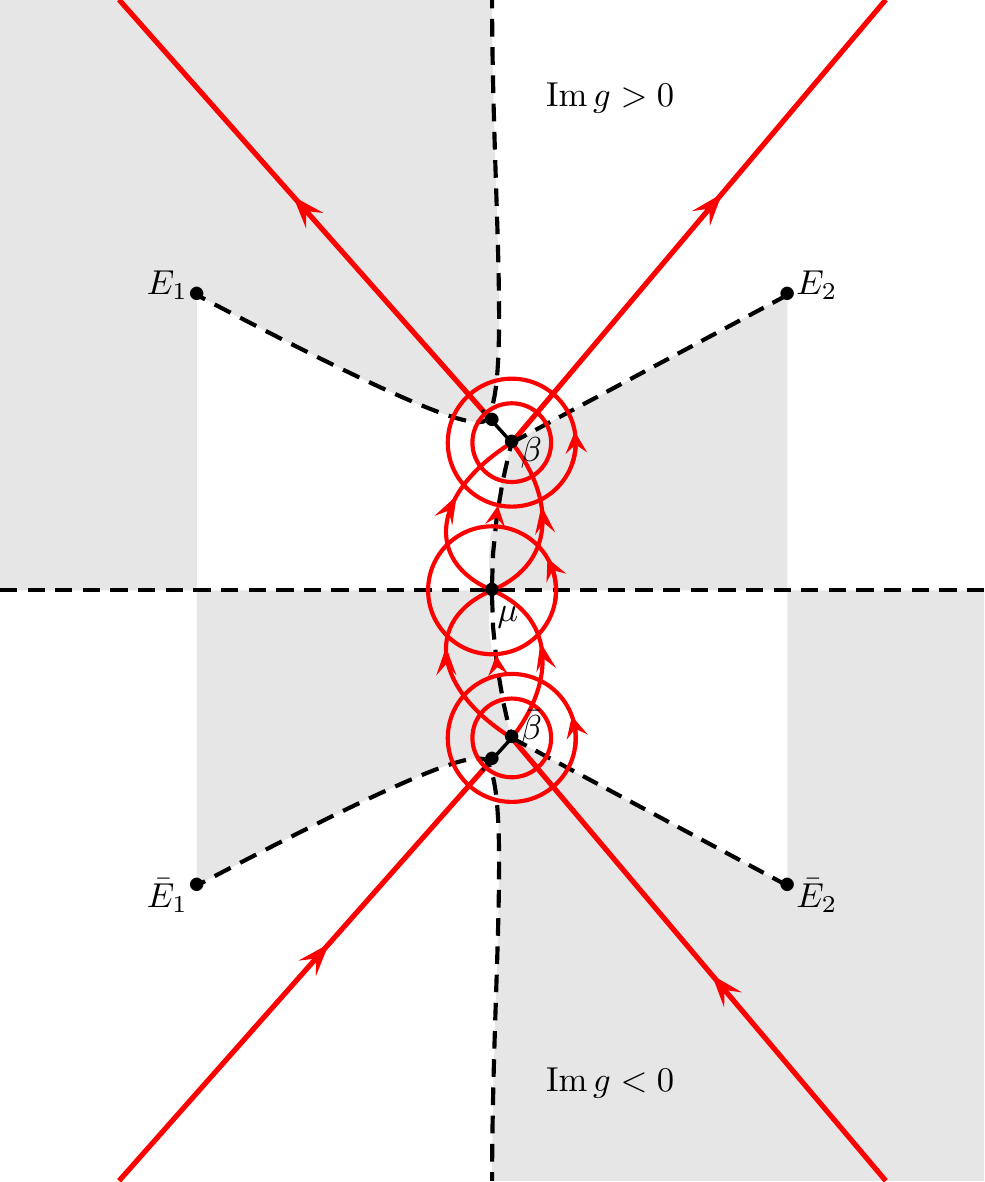}
\caption{The jump contour $\Sigma^{\diff}$.} 
\label{fig:hat-contour}
\end{figure}

Let $\hat w\coloneqq v^{\diff}-I$. The matrix $\hat v^{(6)}-I$ is exponentially small on $\Sigma'\coloneqq\Sigma^{(6)}\setminus(\bar{\C{D}}\cup\gamma_{(\bar\beta,\beta)})$, uniformly with respect to $(x,t)\in\D{S}$, that is,
\begin{subequations}   \label{eq:whatestims}
\begin{equation}  \label{eq:whatestima}
\norm{\hat w}_{(L^1\cap L^{\infty})(\Sigma')}\leq C\eul^{-ct},\quad(x,t)\in\D{S}.
\end{equation}
On the other hand, equation \eqref{eq:mmodbinfty} implies
\begin{equation}  \label{eq:whatestimb}
\norm{\hat w}_{(L^1\cap L^{\infty})(\partial D_{\epsilon}(\beta)\cup\partial D_{\epsilon}(\bar\beta))}\leq Ct^{\abs{\Im\tilde\nu}-\frac{1}{2}}+Ct^{\abs{\Im\tilde{\nu}}+1}\abs{\alpha-\beta}^2,\quad(x,t)\in\D{S}.
\end{equation}
As in \cite{BLS22}*{(7.2c)}, we have
\begin{equation}  \label{eq:whatestimc}
\norm{\hat w}_{(L^1\cap L^{\infty})(\partial D_{\epsilon}(\mu))}\leq Ct^{-1/2},\quad(x,t)\in\D{S},
\end{equation}
and
\begin{equation}  \label{eq:whatestimd}
\begin{cases}
\norm{\hat w}_{L^1(\C{X})}\leq Ct^{-1}\ln t,\\
\norm{\hat w}_{L^2(\C{X})}\leq Ct^{-3/4}\ln t,\\
\norm{\hat w}_{L^{\infty}(\C{X})}\leq Ct^{-1/2}\ln t,
\end{cases}\quad(x,t)\in\D{S}.
\end{equation}
For $k\in\C{Y}$, we have 
\[
\hat w=m_-^{\beta}(\hat v^{(6)}-v^{\beta})(m_+^{\beta})^{-1},
\]
so Lemma~\ref{lem:mbeta} yields
\begin{equation}   \label{eq:whatestime}
\begin{cases}
\norm{\hat w}_{L^1(\C{Y})}\leq CE_1(x,t),\\
\norm{\hat w}_{L^{\infty}(\C{Y})}\leq CE_{\infty}(x,t),
\end{cases}\quad(x,t)\in\D{S}.
\end{equation}
For $k\in\gamma_{(\bar\beta,\beta)}\setminus\bar{\C{D}}$ we have $\hat w=m^{\model}(\eul^{\ii t(g_+-g_-)\sigma_3}-I)(m^{\model})^{-1}$; thus, since $g_+-g_-$ is constant on $\gamma_{(\bar\beta,\beta)}$ and $g_+(\beta)-g_-(\beta)=\ord(\abs{\alpha-\beta}^2)$ (cf.\ \eqref{eq:gbestim2}),
\begin{equation}  \label{eq:whatestimf}
\norm{\hat w}_{(L^1\cap L^2\cap L^{\infty})(\gamma_{(\bar\beta,\beta)}\setminus\bar{\C{D}})}\leq Ct\abs{\alpha-\beta}^2,\quad(x,t)\in\D{S}.
\end{equation}
\end{subequations}

Equations~\eqref{eq:whatestims} show that
\begin{subequations} 
\begin{alignat}{2}
\norm{\hat w}_{L^1(\Sigma^{\diff})}&\leq CE_1'(x,t),&\quad&(x,t)\in\D{S},\\
\norm{\hat w}_{L^{\infty}(\Sigma^{\diff})}&\leq CE_{\infty}(x,t),&&(x,t)\in\D{S},
\end{alignat}
\end{subequations}
where
\begin{equation} \label{E1prim}
E_1'(x,t)\coloneqq E_1(x,t)+t^{\abs{\Im\tilde\nu}-\frac{1}{2}}+t^{\abs{\Im\tilde{\nu}}+1}\abs{\alpha-\beta}^2.
\end{equation}
Since $\abs{\Im\tilde\nu}<1/2$, $E_{\infty}(x,t)$ goes uniformly to zero as $t\to+\infty$. Hence
\begin{equation}
\norm{\hat{\C{C}}_{\hat w}}_{\C{B}(L^2(\Sigma^{\diff}))}\leq C\norm{\hat w}_{L^{\infty}(\Sigma^{\diff})}\to 0,\quad t\to+\infty,
\end{equation}
uniformly with respect to $(x,t)\in\D{S}$, where $\hat{\C{C}}_{\hat w}$ is defined by $\hat{\C{C}}_{\hat w}f=\hat{\C{C}}_-(f\hat w)$ as in \cite{BLS22}*{Section 7.2} ($\hat{\C{C}}_-h$ is the boundary value of $\hat{\C{C}}h$ from the right side of $\Sigma^{\diff}$ with $\hat{\C{C}}$ the Cauchy operator associated with $\Sigma^{\diff}$). Hence, increasing $T$ in \eqref{eq:wedge} if necessary, we have
\[
\norm{\hat{\C{C}}_{\hat w}}_{\C{B}(L^2(\Sigma^{\diff}))}\leq C<1,\quad(x,t)\in\D{S}.
\]
Then, in particular, $I-\hat{\C{C}}_{\hat w(x,t,\,\cdot\,)}\in\C{B}(L^2(\Sigma^{\diff}))$ is invertible for all $(x,t)\in\D{S}$, so we can define $\hat\mu(x,t,k)\in I+L^2(\Sigma^{\diff})$ by
\begin{equation}
\hat\mu=I+(I-\hat{\C{C}}_{\hat w})^{-1}\hat{\C{C}}_{\hat w}I,\quad(x,t)\in\D{S}.
\end{equation}
Standard estimates using the Neumann series show that 
\[
\norm{\hat\mu-I}_{L^2(\Sigma^{\diff})}\leq C\frac{\norm{\hat w}_{L^2(\Sigma^{\diff})}}{1-\norm{\hat{\C{C}}_{\hat w}}_{\C{B}(L^2(\Sigma^{\diff}))}}.
\]
Thus
\begin{equation}  \label{eq:muhatestim}
\norm{\hat\mu(x,t,\,\cdot\,)-I}_{L^2(\Sigma^{\diff})}\leq C\norm{\hat w}_{L^2(\Sigma^{\diff})}\leq C\sqrt{E_1'E_{\infty}},\quad(x,t)\in\D{S}.
 \end{equation} 
It follows that there exists a unique solution $m^{\diff}\in I+\dot E^2(\hat{\D{C}}\setminus\Sigma^{\diff})$ of the RH problem \eqref{eq:rhphat} for all sufficiently large $t$. This solution is given by
\begin{equation}
m^{\diff}(x,t,k)=I+\hat{\C{C}}(\hat\mu\hat w)=I+\frac{1}{2\pi\ii}\int_{\Sigma^{\diff}}\hat\mu(x,t,s)\hat w(x,t,s)\frac{\dd s}{s-k},\quad k\in\D{C}\setminus\Sigma^{\diff},\ (x,t)\in\D{S}.
\end{equation}
\subsection{Asymptotics of $\BS{m^{\diff}}$}

For each $(x,t)\in\D{S}$, we have
\begin{equation}   \label{eq:lim-mhat}
\lim_{k\to\infty}k\left(m^{\diff}(x,t,k)-I\right)=-\frac{1}{2\pi\ii}\int_{\Sigma^{\diff}}\hat\mu(x,t,k)\hat w(x,t,k)\dd k.
\end{equation}
Let us consider the contributions to the right-hand side of \eqref{eq:lim-mhat} from the different parts of $\Sigma^{\diff}$. All error terms in what follows will be uniform with respect to $(x,t)\in\D{S}$.

By \eqref{eq:whatestima} and \eqref{eq:muhatestim},
\begin{align*}
\int_{\Sigma'}\hat\mu(x,t,k)\hat w(x,t,k)\dd k
&=\int_{\Sigma'}\hat w(x,t,k)\dd k+\int_{\Sigma'}(\hat\mu(x,t,k)-I)\hat w(x,t,k)\dd k\\
&=\ord\left(\norm{\hat w}_{L^1(\Sigma')}\right)+\ord\left(\norm{\hat\mu-I}_{L^2(\Sigma')}\norm{\hat w}_{L^2(\Sigma')}\right)=\ord(\eul^{-ct}),\quad t\to+\infty.
\end{align*}
Hence the contribution to the integral in \eqref{eq:lim-mhat} from $\Sigma'$ is $\ord(\eul^{-ct})$. The contribution from $\partial D_{\epsilon}(\mu)$ to the right-hand side of \eqref{eq:lim-mhat} is
\begin{align*}
&-\frac{1}{2\pi\ii}\int_{\partial D_{\epsilon}(\mu)}\hat w(x,t,k)\dd k-\frac{1}{2\pi\ii}\int_{\partial D_{\epsilon}(\mu)}(\hat\mu(x,t,k)-I)\hat w(x,t,k)\dd k\\
&\quad=-\frac{1}{2\pi\ii}\int_{\partial D_{\epsilon}(\mu)}\left(m^{\model}(m^{\mu})^{-1}-I\right)\dd k+\ord\left(\norm{\hat\mu-I}_{L^2(\partial D_{\epsilon}(\mu))}\norm{\hat w}_{L^2(\partial D_{\epsilon}(\mu))}\right)\\
&\quad=\frac{T_{\mu}(x,t)}{\sqrt{t}}+\ord(t^{-1})+\ord\left(\sqrt{E_1'E_{\infty}}\,t^{-\frac{1}{2}}\right),\quad t\to+\infty,
\end{align*}
where
\begin{equation}  \label{eq:Tmu}
T_{\mu}(x,t)\coloneqq -\frac{Y_{\mu}(x,t,\mu)m_1^X(\xi)Y_{\mu}(x,t,\mu)^{-1}}{\psi_{\mu}(\xi,\mu)}\,.
\end{equation}
The function $\psi_{\mu}(k)\equiv\psi_{\mu}(\xi,k)$ and the constant $m_1^X\equiv m_1^X(\xi)$ are defined in \cite{BLS22}*{(6.17) \& (6.31)}, while $Y_{\mu}$ is defined as in \cite{BLS22}*{(6.26)} but with a different $m^{\model}$, which comes from Section~\ref{sec:model}:
\[
Y_{\mu}(x,t,k)\coloneqq m^{\model}(x,t,k)\eul^{-\ii h(k)\sigma_3}B(k)\tilde\delta(k)^{-\sigma_3}\delta_0(t)^{\sigma_3}.
\]
The contribution from $\C{X}$ to the right-hand side of \eqref{eq:lim-mhat} is
\[
\ord\left(\norm{\hat w}_{L^1(\C{X})}\right)+\ord\left(\norm{\hat\mu-I}_{L^2(\Sigma^{\diff})}\norm{\hat w}_{L^2(\C{X})}\right)=\ord\left(t^{-1}\ln t+\sqrt{E_1'E_{\infty}}\,t^{-\frac{3}{4}}\ln t\right),\quad t\to+\infty.
\]
By \eqref{eq:mmodbbound}, \eqref{eq:whatestimb}, and \eqref{eq:muhatestim}, the contribution from $\partial D_{\epsilon}(\beta)$ to the right-hand side of \eqref{eq:lim-mhat} is
\begin{align*}
&-\frac{1}{2\pi\ii}\int_{\partial D_{\epsilon}(\beta)}\hat w(x,t,k)\dd k-\frac{1}{2\pi\ii}\int_{\partial D_{\epsilon}(\beta)}(\hat\mu(x,t,k)-I)\hat w(x,t,k)\dd k\\
&\quad=-\frac{1}{2\pi\ii}\int_{\partial D_{\epsilon}(\beta)}\left(m^{\model}(m^{\beta})^{-1}-I\right)\dd k+\ord\left(\norm{\hat\mu-I}_{L^2(\partial D_{\epsilon}(\beta))}\norm{\hat w}_{L^2(\partial D_{\epsilon}(\beta))}\right)\\
&\quad=\frac{T_{\beta}(x,t)}{\sqrt{t}}+\ord\left(\sqrt{E_1'E_{\infty}}\,(t^{\abs{\Im\tilde\nu}-\frac{1}{2}}+t^{\abs{\Im\tilde\nu}+1}\abs{\alpha-\beta}^2)\right),\quad t\to+\infty,
\end{align*}
where
\begin{equation}  \label{eq:Tbeta}
T_{\beta}(x,t)\coloneqq -\frac{Y_{\beta}(x,t,\beta)t^{\frac{\ii\tilde\nu}{2}\sigma_3}m_1^Y(y,\tilde\nu)t^{-\frac{\ii\tilde\nu}{2}\sigma_3}Y_{\beta}(x,t,\beta)^{-1}}{\psi_{\beta}(\xi)}\,.
\end{equation}
The function $Y_{\beta}(x,t,k)$ is defined in \eqref{eq:Ybeta}, $m_1^Y(y,\tilde\nu)$ is defined in \eqref{eq:m1Y} with $y\equiv y(x,t)$ given by \eqref{eq:y}, and $\psi_{\beta}(\xi)$ is defined in \eqref{eq:psibeta}. The symmetries $\hat\mu(x,t,k)=\sigma_3\sigma_1\overline{\hat\mu(x,t,\bar k)}\sigma_1\sigma_3$ and $\hat w(x,t,k)=\sigma_3\sigma_1\overline{\hat w(x,t,\bar k)}\sigma_1\sigma_3$ imply that the contribution from $D_{\epsilon}(\bar\beta)$ to the right-hand side of \eqref{eq:lim-mhat} is
\begin{align*}
&-\frac{1}{2\pi\ii}\int_{\partial D_{\epsilon}(\bar\beta)}(\hat\mu\hat w)(x,t,k)\dd k=\frac{1}{2\pi\ii}\sigma_3\sigma_1\overline{\int_{\partial D_{\epsilon}(\beta)}(\hat\mu\hat w)(x,t,k)\dd k}\,\sigma_1\sigma_3\\
&\qquad\qquad\qquad=\frac{\sigma_3\sigma_1\overline{T_{\beta}(x,t)}\,\sigma_1\sigma_3}{\sqrt{t}}+\ord\left(\sqrt{E_1'E_{\infty}}\,(t^{\abs{\Im\tilde\nu}-\frac{1}{2}}+t^{\abs{\Im\tilde\nu}+1}\abs{\alpha-\beta}^2)\right),\quad t\to+\infty.
\end{align*}
The contribution from $\C{Y}=\Sigma^{(6)}\cap D_{\epsilon}(\beta)$ to the right-hand side of \eqref{eq:lim-mhat} is
\[
\ord\left(\norm{\hat w}_{L^1(\C{Y})}+\norm{\hat\mu-I}_{L^2(\C{Y})}\norm{\hat w}_{L^2(\C{Y})}\right)=\ord\left(E_1+\sqrt{E_1'E_{\infty}}\sqrt{E_1E_{\infty}}\right),\quad t\to+\infty.
\]
The contribution from $\Sigma^{(6)}\cap D_{\epsilon}(\bar\beta)$ is of the same order. By \eqref{eq:whatestimf} and \eqref{eq:muhatestim}, the contribution from $\Sigma''\coloneqq\gamma_{(\bar\beta,\beta)}\setminus\bar{\C{D}}$ is
\[
\ord\left(\norm{\hat w}_{L^1(\Sigma'')}\right)+\ord\left(\norm{\hat\mu-I}_{L^2(\Sigma'')}\norm{\hat w}_{L^2(\Sigma'')}\right)=\ord\left(t\abs{\alpha-\beta}^2+\sqrt{E_1'E_{\infty}}\,t\abs{\alpha-\beta}^2\right),\quad t\to+\infty.
\]
Collecting the above contributions, we find from \eqref{eq:lim-mhat} that
\begin{align}   \label{eq:mhatinfty}
&\lim_{k\to\infty}k(m^{\diff}(x,t,k)-I)=\frac{T_{\mu}(x,t)+T_{\beta}(x,t)+\sigma_3\sigma_1\overline{T_{\beta}(x,t)}\sigma_1\sigma_3}{\sqrt{t}}+\ord\left(\sqrt{E_1'E_{\infty}}\,t^{-\frac{1}{2}}\right)\notag\\
&\qquad+\ord\left(t^{-1}\ln t+\sqrt{E_1'E_{\infty}}\,t^{-\frac{3}{4}}\ln t\right)+\ord\left(\sqrt{E_1'E_{\infty}}\,(t^{\abs{\Im\tilde\nu}-\frac{1}{2}}+t^{\abs{\Im\tilde\nu}+1}\abs{\alpha-\beta}^2)\right)\notag\\
&\qquad+\ord\left(E_1+\sqrt{E_1'E_{\infty}}\sqrt{E_1E_{\infty}}\right)+\ord\left(t\abs{\alpha-\beta}^2+\sqrt{E_1'E_{\infty}}\,t\abs{\alpha-\beta}^2\right),\quad t\to+\infty.
\end{align}
Using that
\[
E_{\infty}\leq F,\quad E_1\leq t^{-\frac{1}{2}}F,\quad E_1'\leq F,
\]
where $F$ is given by \eqref{eq:F}, we can replace the error term with the simpler (and only slightly less sharp) expression
\[
\ord\left(F^2+t\abs{\alpha-\beta}^2\right).
\]
In particular, if $\xi=\xi_0$, then the error is $\ord\left(t^{2\abs{\Im\tilde\nu_0}-1}(\ln t)^{2\abs{\Im\tilde\nu_0}+4}\right)$.

\subsection{Asymptotics of $\BS{q}$}

Recalling the various transformations of Section~\ref{sec:transforms}, we have
\[
\hat m=\eul^{\ii tg^{(0)}\sigma_3}\eul^{\ii\tilde{h}(\xi,\infty)\sigma_3}m^{\diff}m^{\model}\eul^{-\ii\tilde{h}\sigma_3}\tilde\delta^{\sigma_3}\delta^{\sigma_3}\eul^{-\ii t(g(k)-\theta(k))\sigma_3}
\]
for all large $k$ in $\ii\,\D{R}_+$. It follows that
\begin{align*}
&\lim_{k\to\infty}k\left(\hat m(x,t,k)-I\right)_{12}\\
&\qquad=\eul^{\ii tg^{(0)}\sigma_3}\eul^{\ii\tilde h(\xi,\infty)\sigma_3}\lim_{k\to\infty}k\left(m^{\model}-I+(m^{\diff}-I)m^{\model}\right)_{12}\eul^{-\ii\tilde h(\xi,\infty)\sigma_3}\eul^{-\ii tg^{(0)}\sigma_3}\\
&\qquad=\eul^{\ii(tg^{(0)}+\tilde h(\xi,\infty))\hat\sigma_3}\Bigl(\lim_{k\to\infty}k\left(m^{\model}-I\right)_{12}+\lim_{k\to\infty}k(m^{\diff}-I)_{12}\Bigr).
\end{align*}
Hence
\begin{align*}
q(x,t)&=2\ii\lim_{k\to\infty}k\left(\hat m(x,t,k)\right)_{12}\\
&=2\ii\eul^{2\ii(tg^{(0)}(\xi)+\tilde h(\xi,\infty))}\Bigl(\lim_{k\to\infty}k\left(m^{\model}(x,t,k)\right)_{12}+\lim_{k\to\infty}k\left(m^{\diff}(x,t,k)\right)_{12}\Bigr).
\end{align*}
In view of \eqref{eq:mmodinfty} and \eqref{eq:mhatinfty}, this yields \eqref{eq:main} and completes the proof of Theorem~\ref{thm:main}.
\appendix
\section{An exactly solvable RH problem}  \label{sec:app}

We define the contour $Y\subset\D{C}$ by $Y=\cup_{j=1}^9Y_j$, where $Y_1,\dots,Y_4$ denote the four rays
\begin{equation}    \label{eq:rays}
Y_j=\accol{r\eul^{\frac{(j-1)\pi}{2}}\mid 1\leq r<\infty},\quad j=1,\dots,4,
\end{equation}
and $Y_5,\dots,Y_9$ denote the following arcs whose union is the unit circle:
\begin{alignat*}{3}
&Y_5=\left\lbrace\eul^{\ii\varphi}\,\Big\vert\,-\frac{\pi}{2}\leq\varphi\leq 0\right\rbrace,&\quad&Y_6=\left\lbrace\eul^{\ii\varphi}\,\Big\vert\, 0\leq\varphi\leq\frac{\pi}{2}\right\rbrace,&\quad&Y_7=\left\lbrace\eul^{\ii\varphi}\,\Big\vert\,\frac{\pi}{2}\leq\varphi\leq\pi\right\rbrace,\\
&Y_8=\left\lbrace\eul^{\ii\varphi}\,\Big\vert\,\pi\leq\varphi\leq\frac{5\pi}{4}\right\rbrace,&&Y_9=\left\lbrace\eul^{\ii\varphi}\,\Big\vert\,-\frac{3\pi}{4}\leq\varphi\leq-\frac{\pi}{2}\right\rbrace.&&
\end{alignat*}
We orient $Y$ as in Figure~\ref{fig:raysY}.
\begin{figure}[ht]
\centering\includegraphics[scale=.9]{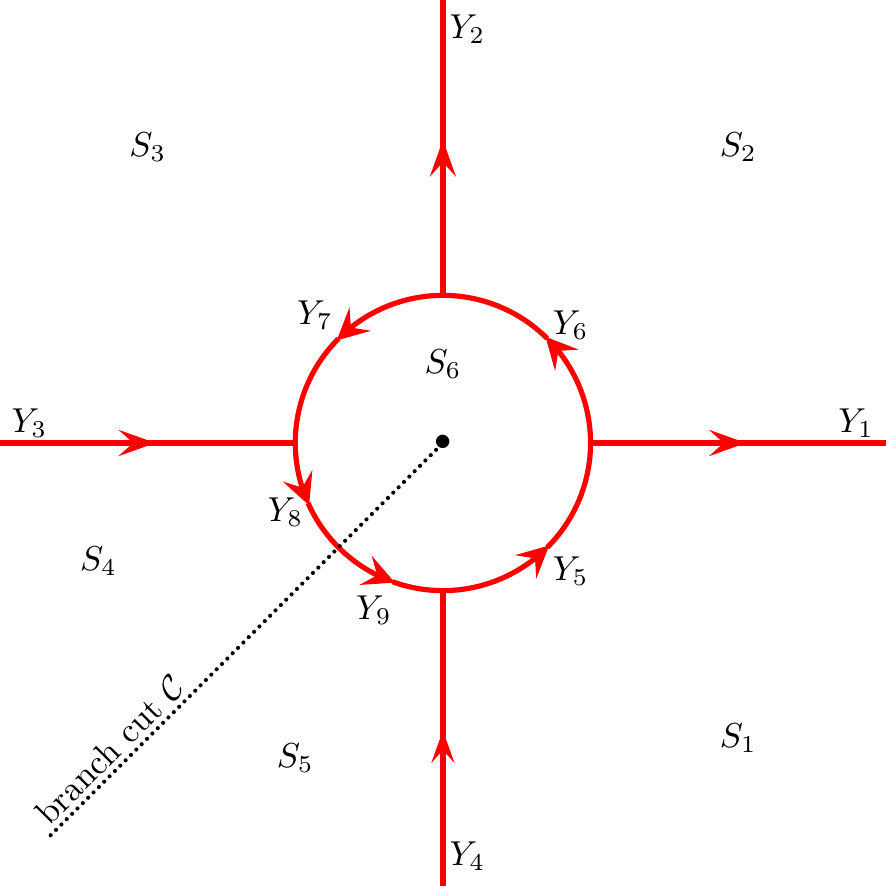}
\caption{The contour $Y=\cup_{j=1}^{9}Y_j$ and the domains $\accol{S_j}_1^6$ in the complex $\lambda$-plane.} 
\label{fig:raysY}
\end{figure}

We let $\C{C}$ denote a branch cut going from $0$ to $\infty$ in the third quadrant (see Figure~\ref{fig:raysY}). We let $\ln_{\C{C}}$ denote the function $\ln\lambda$ with cut along $\C{C}$ and branch fixed by the condition that $\ln_{\C{C}}\lambda>0$ for $\lambda>0$.

Given $y\in\D{R}$ and $\tilde\nu\in\D{C}\setminus\ii\D{Z}$, we define the jump matrix $v^Y(y,\tilde\nu,\lambda)$ for $\lambda\in Y$ by
\begin{equation}    \label{eq:vY}
v^Y(y,\tilde\nu,\lambda)=\begin{cases}
\begin{pmatrix}
1&0\\\frac{\rho\rho^*}{1+\rho\rho^*}\lambda^{2\ii\tilde\nu}\eul^{-\lambda^2-2y\lambda}&1\end{pmatrix}&\lambda\in Y_1,\\
\begin{pmatrix}
1&-\lambda^{-2\ii\tilde\nu}\eul^{\lambda^2+2y\lambda}\\0&1\end{pmatrix}&\lambda\in Y_2,\\
\begin{pmatrix}
1&0\\\rho\rho^*\lambda^{2\ii\tilde\nu}\eul^{-\lambda^2-2y\lambda}&1\end{pmatrix}&\lambda\in Y_3,\\
\begin{pmatrix}
1&-(1+\rho\rho^*)\lambda^{-2\ii\tilde\nu}\eul^{\lambda^2+2y\lambda}\\0&1\end{pmatrix}&\lambda\in Y_4,\\
\lambda^{-\ii\tilde\nu\sigma_3},&\lambda\in Y_5,\\
\lambda^{-\ii\tilde\nu\sigma_3}\begin{pmatrix}
1&0\\-\frac{\rho\rho^*}{1+\rho\rho^*}\eul^{-\lambda^2-2y\lambda}&1\end{pmatrix}&\lambda\in Y_6,\\
\lambda^{-\ii\tilde\nu\sigma_3}\begin{pmatrix}
\frac{1}{1+\rho\rho^*}&\eul^{\lambda^2+2y\lambda}\\-\frac{\rho\rho^*}{1+\rho\rho^*}\eul^{-\lambda^2-2y\lambda}&1\end{pmatrix}&\lambda\in Y_7,\\
\lambda^{-\ii\tilde\nu\sigma_3}\begin{pmatrix}
\frac{1}{1+\rho\rho^*}&\eul^{\lambda^2+2y\lambda}\\0&1+\rho\rho^*\end{pmatrix}&\lambda\in Y_8,\\
\lambda^{-\ii\tilde\nu\sigma_3}\begin{pmatrix}
1&(1+\rho\rho^*)\eul^{\lambda^2+2y\lambda}\\0&1\end{pmatrix}&\lambda\in Y_9,
\end{cases}
\end{equation}
where $\rho\rho^*\coloneqq\eul^{2\pi\tilde\nu}-1$ and the branch cut runs along $\C{C}$, so that $\lambda^a=\eul^{a\ln_{\C{C}}\lambda}$ for $a\in\D{C}$.

We consider the following family of RH problems parametrized by $y\in\D{R}$ and $\tilde\nu\in\D{C}\setminus\ii\D{Z}$:
\begin{equation}   \label{eq:rhpY}
\begin{cases}
m^Y(y,\tilde\nu,\,\cdot\,)\in I+\dot E^2(\D{C}\setminus Y),&\\
m_+^Y(y,\tilde\nu,\lambda)=m_-^Y(y,\tilde\nu,\lambda)v^Y(y,\tilde\nu,\lambda)&\text{for a.e. }\lambda\in Y.
\end{cases}
\end{equation}
Define the function $\beta^Y(\tilde\nu)$ by
\begin{equation}   \label{eq:bY}
\beta^Y(\tilde\nu)=\frac{(2\ii)^{\ii\tilde\nu+1}\sqrt{\pi}\,\eul^{\frac{5\pi\tilde\nu}{2}}}{(\eul^{2\pi\tilde\nu}-1)\Gamma(\ii\tilde\nu)},\quad\tilde\nu\in\D{C}\setminus\ii\D{Z}. 
\end{equation}
   
\begin{theorem}  \label{thm:rhpY}
For any choice of $y\in\D{R}$ and $\tilde\nu\in\D{C}\setminus\ii\D{Z}$, the RH problem \eqref{eq:rhpY} has a unique solution $m^Y(y,\tilde\nu,\lambda)$. This solution satisfies
\begin{equation}   \label{eq:mYas}
m^Y(y,\tilde\nu,\lambda)=I+\frac{m_1^Y(y,\tilde\nu)}{\lambda}+\ord(\lambda^{-2}),\quad\lambda\to\infty,
\end{equation}
where $m_1^Y$ is defined by
\begin{equation}   \label{eq:m1Y}
m_1^Y(y,\tilde\nu)\coloneqq\begin{pmatrix}
-\ii\tilde\nu y&-\frac{\beta^Y(\tilde\nu)}{2}\,\eul^{-y^2}\\-\frac{\ii\tilde\nu}{\beta^Y(\tilde\nu)}\,\eul^{y^2}&\ii\tilde\nu y\end{pmatrix},\quad y\in\D{R},\ \tilde\nu\in\D{C}\setminus\ii\D{Z},
\end{equation}
and the error term is uniform with respect to $\arg\lambda\in\croch{0,2\pi}$ and $\tilde\nu$ and $y$ in compact subsets of $\D{C}\setminus\ii\D{Z}$ and $\D{R}$, respectively. Moreover, for any compact subsets $K_1 \subset\D{R}$ and $K_2\subset\D{C}\setminus\ii\D{Z}$, we have
\begin{equation}   \label{eq:mYbound}
\sup_{y\in K_1}\sup_{\tilde\nu\in K_2}\sup_{\lambda\in\D{C}\setminus Y}\abs{m^Y(y,\tilde\nu,\lambda)}<\infty.
\end{equation}
\end{theorem}

\begin{proof}
Uniqueness of the solution $m^Y$ follows because $\det v^Y=1$. Fix $y\in\D{R}$ and $\tilde\nu\in\D{C}\setminus\ii\D{Z}$. Let the branch cut run along $\C{C}$. Define the $2\times 2$-matrix valued function $\psi(y,\tilde\nu,\lambda)$ by
\begin{equation}    \label{eq:psi}
\psi=\eul^{-\frac{y^2}{2}\sigma_3}\begin{pmatrix}
\eul^{-\frac{\pi\tilde\nu}{2}}D_{-\ii\tilde\nu}(\ii\sqrt{2}(y+\lambda))&-\frac{\beta^Y(\tilde\nu)}{\sqrt{2}}D_{\ii\tilde\nu-1}(\sqrt{2}(y+\lambda))\\
\frac{\sqrt{2}}{\beta^Y(\tilde\nu)}\eul^{-\frac{\pi\tilde\nu}{2}}\tilde\nu D_{-\ii\tilde\nu-1}(\ii\sqrt{2}(y+\lambda))&D_{\ii\tilde\nu}(\sqrt{2}(y+\lambda))\end{pmatrix}2^{\frac{\ii\tilde\nu}{2}\sigma_3},
\end{equation}
where $D_a(z)$ denotes the parabolic cylinder function. Define the matrices $R_j\equiv R_j(\tilde\nu)$, $j=1,\dots,4$, by
\begin{alignat*}{2}
R_1&=\begin{pmatrix}1&0\\\frac{\rho\rho^*}{1+\rho\rho^*}&1\end{pmatrix},&\qquad&R_2=\begin{pmatrix}1&-1\\0&1\end{pmatrix},\\
R_3&=\begin{pmatrix}1&0\\\rho\rho^*&1\end{pmatrix},&&R_4=\begin{pmatrix}1&-(1+\rho\rho^*)\\0&1\end{pmatrix}.
\end{alignat*}
We claim that the solution of \eqref{eq:rhpY} is given explicitly in terms of parabolic cylinder functions by
\begin{equation}   \label{eq:mY}
m^Y(y,\tilde\nu,\lambda)=
\begin{cases}
\Psi(y,\tilde\nu,\lambda)\eul^{-(\frac{\lambda^2}{2}+y\lambda)\sigma_3}\lambda^{\ii\tilde\nu\sigma_3},&\lambda\in S_1\cup\dots\cup S_5,\\
\Psi(y,\tilde\nu,\lambda)\eul^{-(\frac{\lambda^2}{2}+y\lambda)\sigma_3}, &\lambda\in S_6,
\end{cases}
\end{equation}
where the sectionally analytic function $\Psi$ is defined by
\begin{equation}   \label{eq:Psi}
\Psi(y,\tilde\nu,\lambda)=\begin{cases}
\psi,&\lambda\in S_1,\\
\psi R_1,&\lambda\in S_2,\\
\psi R_1R_2,&\lambda\in S_3,\\
\psi R_1R_2R_3^{-1},&\lambda\in S_4,\\
\psi R_4,&\lambda\in S_5,\\
\psi,&\lambda\in S_6.
\end{cases}
\end{equation}
Recall that $D_a(z)$ is an entire function of both $a$ and $z$. In particular, $m^Y(y,\tilde\nu,\lambda)$ is an analytic function of $\lambda\in\D{C}\setminus Y$. It is easily seen from the definition \eqref{eq:Psi} of $\Psi$ that $m^Y$ satisfies the jump condition in \eqref{eq:rhpY}.

For each $\delta>0$, the parabolic cylinder function satisfies the asymptotic formula \cite{Ol10}
\begin{align}   \label{eq:Das}
D_a(z)
&=z^a\eul^{-\frac{z^2}{4}}\left(1-\frac{a(a-1)}{2z^2}+\ord(z^{-4})\right)\notag\\
&\quad-\frac{\sqrt{2\pi}\eul^{\frac{z^2}{4}}z^{-a-1}}{\Gamma(-a)}\left(1+\frac{(a+1)(a+2)}{2z^2}+\ord(z^{-4})\right)\notag\\
&\quad\times\begin{cases}
0,&\arg z\in\left\lbrack-\frac{3\pi}{4}+\delta,\frac{3\pi}{4}-\delta\right\rbrack,\\
\eul^{\ii\pi a},&\arg z\in\left\lbrack\frac{\pi}{4}+\delta,\frac{5\pi}{4}-\delta\right\rbrack,\\
\eul^{-\ii\pi a},&\arg z\in\left\lbrack-\frac{5\pi}{4}+\delta,-\frac{\pi}{4}-\delta\right\rbrack,
\end{cases}\quad z\to\infty,\ a\in\D{C},
\end{align}
where the error terms are uniform with respect to $a$ in compact subsets and $\arg z$ in the given ranges. Tedious computations using this formula in the explicit formula \eqref{eq:mY} show that $m^Y$ satisfies \eqref{eq:mYas}, together with \eqref{eq:m1Y}, uniformly for $\tilde\nu$ and $y$ in compact subsets and for $\arg\lambda\in\croch{0,2\pi}$. It follows that $m^Y$ satisfies \eqref{eq:mYbound} and that $m^Y(y,\tilde\nu,\,\cdot\,)\in I+\dot E^2(\D{C}\setminus Y)$. This completes the proof of the theorem.
\end{proof}

\begin{remark}  \label{rem:mY}
The explicit formula for $m^Y$ given in \eqref{eq:mY} can be derived as follows. Suppose $m^Y$ is a solution of \eqref{eq:rhpY}. Let
\[
\Psi\coloneqq m^Y\eul^{\bigl(\frac{\lambda^2}{2}+y\lambda\bigr)\sigma_3}\lambda^{-\ii\tilde\nu\sigma_3}.
\]
Then $\Psi$ satisfies the jump condition
\[
\Psi_+=\Psi_-v^{\Psi},\quad\lambda\in Y\cup\left\lbrace\arg\lambda=-\tfrac{3\pi}{4}\right\rbrace,
\]
where the jump matrix $v^{\Psi}$ is given on $Y$ by (for $j=1,\dots,9$, $v_j^{\Psi}$ denotes the restriction of $v^{\Psi}$ to $Y_j$)
\begin{alignat*}{3}
&v_j^{\Psi}=R_j,\ j=1,\dots,4,&\qquad&v_5^{\Psi}=I,&\qquad&v_6^{\Psi}=\begin{pmatrix}1&0\\-\frac{\rho\rho^*}{1+\rho\rho^*}&1\end{pmatrix},\\
&v_7^{\Psi}=\begin{pmatrix}\frac{1}{1+\rho\rho^*}&1\\-\frac{\rho\rho^*}{1+\rho\rho^*}&1\end{pmatrix},&&v_8^{\Psi}=\begin{pmatrix}\frac{1}{1+\rho\rho^*}&0\\0&1+\rho\rho^*\end{pmatrix},&&v_9^{\Psi}=\begin{pmatrix}1&1+\rho\rho^*\\0&1\end{pmatrix},
\end{alignat*}
and on the branch cut (oriented towards the origin) by
\[
v^{\Psi}=\begin{pmatrix}1+\rho\rho^*&0\\0&\frac{1}{1+\rho\rho^*}\end{pmatrix},\quad\arg\lambda=-\frac{3\pi}{4}\,.
\]
Since $v^{\Psi}$ is independent of $\lambda$ and $y$, we conclude that
\[
A\coloneqq\Psi_{\lambda}\Psi^{-1}=m_{\lambda}^Y(m^Y)^{-1}+(\lambda+y)m^Y\sigma_3(m^Y)^{-1}-\ii\tilde\nu\lambda^{-1}m^Y\sigma_3(m^Y)^{-1}
\]
and
\[
U\coloneqq\Psi_y\Psi^{-1}=m_y^Y(m^Y)^{-1}+\lambda m^Y\sigma_3(m^Y)^{-1}
\]
are entire functions. Assuming that
\[
m^Y(y,\tilde\nu,\lambda)=I+\frac{m_1^Y(y,\tilde\nu)}{\lambda}+\ord(\lambda^{-2}),\quad\lambda\to\infty,
\]
we deduce that
\begin{align*}
A(y,\tilde\nu,\lambda)&=\lambda A_1(y,\tilde\nu)+A_0(y,\tilde\nu),\\
U(y,\tilde\nu,\lambda)&=\lambda U_1(y,\tilde\nu)+U_0(y,\tilde\nu).
\end{align*}
The terms of order $\ord(\lambda)$ show that $A_1=U_1=\sigma_3$ while the terms of order $\ord(1)$ show that
\[
A_0=y\sigma_3+\croch{m_1^Y,\sigma_3},\quad U_0=\croch{m_1^Y,\sigma_3}.
\]
Defining $w(y,\tilde\nu)$ and $z(y,\tilde\nu)$ by $w\coloneqq -2(m_1^Y)_{12}$ and $z\coloneqq\ii\tilde\nu+w(m_1^Y)_{21}$, we find that $\Psi$ satisfies the Lax pair equations
\begin{equation}  \label{eq:PsiLax}
\Psi_{\lambda}=A\Psi,\quad \Psi_y=U\Psi,
\end{equation}
where
\[
A=\begin{pmatrix}\lambda+y&w(y)\\2\frac{z(y)-\ii\tilde\nu}{w(y)}&-\lambda-y\end{pmatrix},\quad U=\begin{pmatrix}\lambda&w(y)\\2\frac{z(y)-\ii\tilde\nu}{w(y)}&-\lambda\end{pmatrix}.
\]
The $(12)$ element of the compatibility condition $A_y-U_{\lambda}+\croch{A,U}=0$ of this Lax pair implies that $w_y+2yw=0$. This shows that $w$ has the form
\[
w(y,\tilde\nu)=\beta^Y(\tilde\nu)\eul^{-y^2},
\]
where $\beta^Y(\tilde\nu)$ is a function which is independent of $y$. The $(21)$ element of the compatibility condition then yields $z_y=0$; hence $z=z(\tilde\nu)$. Substituting these expressions for $w$ and $z$ into the first column of the Lax pair equation $\Psi_{\lambda}=A\Psi$ yields
\[
\Psi_{11\lambda}-(\lambda+y)\Psi_{11}-\beta^Y\eul^{-y^2}\Psi_{21}=0,\quad\Psi_{21\lambda}+(\lambda+y)\Psi_{21}-\frac{2\eul^{y^2}(z-\ii\tilde\nu)}{\beta^Y}\Psi_{11}=0.
\]
This linear system has the general solution
\begin{align*}
\Psi_{11}(y,\tilde\nu,\lambda)&=C_1(y,\tilde\nu)D_{\ii\tilde\nu-1-z}(\sqrt{2}(\lambda+y))+C_2(y,\tilde\nu)D_{z-\ii\tilde\nu}(\ii\sqrt{2}(y+\lambda)),\\
\Psi_{21}(y,\tilde\nu,\lambda)&=C_3(y,\tilde\nu)D_{\ii\tilde\nu-z}(\sqrt{2}(\lambda+y))+C_4(y,\tilde\nu)D_{z-\ii\tilde\nu-1}(\ii\sqrt{2}(y+\lambda)),
\end{align*}
where $C_j(y,\tilde\nu)$, $j=1,\dots,4$, are locally independent of $\lambda$ (but, in general, the values of the $C_j$ change as $\lambda$ crosses one of the contours $Y_j$). We determine the $y$-dependence of the $C_j$ by substituting these expressions into the first column of the equation $\Psi_y=U\Psi$. The first column of $\Psi_y=U\Psi$ is satisfied provided that
\[
C_j(y,\tilde\nu)=\eul^{-\frac{y^2}{2}}B_j(\tilde\nu),\quad j=1,2;\qquad C_j(y,\tilde\nu)=\eul^{\frac{y^2}{2}}B_j(\tilde\nu),\quad j=3,4,
\]
where the functions $B_j(\tilde\nu)$ are independent of $y$.

Let us first suppose that $\lambda\in S_1$. Utilizing the asymptotic formula \eqref{eq:Das} for the range $\arg z\in\left\lbrack-\frac{3\pi}{4}+\delta,\frac{3\pi}{4}-\delta\right\rbrack$ we see that the condition
\begin{equation}   \label{eq:Psi11}
\Psi_{11}\eul^{-\frac{\lambda^2}{2}-y\lambda}\lambda^{\ii\tilde\nu}=1+\ord(\lambda^{-1}),\quad\lambda\to\infty,
\end{equation}
implies that $z(\tilde\nu)=0$ and
\[
\begin{cases}
C_1=0,&\\
C_2=\eul^{-\frac{y^2}{2}}2^{\frac{\ii\tilde\nu}{2}}\eul^{-\frac{\pi\tilde\nu}{2}},&
\end{cases}\quad\text{in }S_1.
\]
Similarly, the condition that
\begin{equation} \label{eq:Psi21}
\Psi_{21}\eul^{-\frac{\lambda^2}{2}-y\lambda}\lambda^{\ii\tilde\nu}=\frac{(m_1^Y)_{21}}{\lambda}+\ord(\lambda^{-2}),\quad\lambda\to\infty,
\end{equation}
implies that
\[
\begin{cases}
C_3=0,&\\
C_4=\eul^{\frac{y^2}{2}}\frac{2^{\frac{1+\ii\tilde\nu}{2}}\eul^{-\frac{\pi\tilde\nu}{2}}\tilde\nu}{\beta^Y(\tilde\nu)},&
\end{cases}\quad\text{in }S_1.
\]
This yields the first column of the definition \eqref{eq:psi} of $\psi$. The second column is derived in a similar way using the second columns of the Lax pair equations \eqref{eq:PsiLax}.

Let us now assume that $\lambda\in S_2$. Then we use the asymptotic formula \eqref{eq:Das} for the range $\arg z\in\croch{-\frac{3\pi}{4}+\delta,\frac{3\pi}{4}-\delta}$ to compute the asymptotics of $D_{\ii\tilde\nu-z}(\sqrt{2}(\lambda+y))$ and $D_{\ii\tilde\nu-1-z}(\sqrt{2}(\lambda+y))$, whereas we use \eqref{eq:Das} for the range $\arg z\in\croch{\frac{\pi}{4}+\delta,\frac{5\pi}{4}-\delta}$ to compute the asymptotics of $D_{z-\ii\tilde\nu}(\ii\sqrt{2}(\lambda+y))$ and $D_{z-\ii\tilde\nu-1}(\ii\sqrt{2}(\lambda+y))$. The conditions \eqref{eq:Psi11} and \eqref{eq:Psi21} then imply
\[
C_1=\eul^{-\frac{y^2}{2}}\frac{\sqrt{\pi}\,2^{\frac{1}{2}+\frac{\ii\tilde\nu}{2}}\tilde\nu}{\Gamma(\ii\tilde\nu+1)},\quad C_2=\eul^{-\frac{y^2}{2}}2^{\frac{\ii\tilde\nu}{2}}\eul^{-\frac{\pi\tilde\nu}{2}},\quad C_3=\eul^{\frac{y^2}{2}}\frac{\ii\sqrt{\pi}\,2^{1+\frac{\ii\tilde\nu}{2}}}{\beta^Y\Gamma(\ii\tilde\nu)},\quad C_4=\eul^{\frac{y^2}{2}}\frac{2^{\frac{1}{2}+\frac{\ii\tilde\nu}{2}}\eul^{-\frac{\pi\tilde\nu}{2}}\tilde\nu}{\beta^Y}
\]
in $S_2$. Similar computations apply to the second column. This gives an expression for $\Psi$ in $S_2$. We then use, for example, the $(21)$ element of the jump relation $(\Psi_-)^{-1}\Psi_+=R_1$ valid for $\lambda\in Y_1$ to determine $\beta^Y$ in terms of $\tilde\nu$. This yields the expression \eqref{eq:bY} for $\beta^Y(\tilde\nu)$. Finally, we can use the jump matrix $v^{\Psi}$ to obtain an expression for $\Psi$ also for $\lambda\in S_j$, $j=3,\dots,6$. This leads to the expression \eqref{eq:mY} for $m^Y$.
\end{remark}

\begin{remark}  \label{rem:p4}
The RH problem \eqref{eq:rhpY} is closely related to the RH problem associated with Painlev\'e~IV. More precisely, consider the Painlev\'e~IV equation
\begin{equation}    \label{eq:p4}
u_{yy}=\frac{u_y^2}{2u}+\frac{3}{2}u^3+4yu^2+(2+2y^2-4\Theta_{\infty})u-\frac{8\Theta^2}{u}\,,
\end{equation}
where $\Theta_{\infty}$ and $\Theta$ are constant parameters. The jump matrix $v^{\painl}(y,\tilde\nu,\lambda)$ of the RH problem associated with \eqref{eq:p4} has the form (see \cite{Fo06}*{p.~184}):
\begin{equation}   \label{eq:vp4}
v^{\painl}=\begin{cases}
\begin{pmatrix}1&0\\s_1\lambda^{2\Theta_{\infty}}\eul^{-\lambda^2-2y\lambda}&1\end{pmatrix},&\lambda\in Y_1,\\
\begin{pmatrix}1&s_2\lambda^{-2\Theta_{\infty}}\eul^{\lambda^2+2y\lambda}\\0&1\end{pmatrix},&\lambda\in Y_2,\\
\begin{pmatrix}1&0\\-s_3\lambda^{2\Theta_{\infty}}\eul^{-\lambda^2-2y\lambda}&1\end{pmatrix},&\lambda\in Y_3,\\
\begin{pmatrix}1&-s_4(\lambda^{-2\Theta_{\infty}})_+\eul^{\lambda^2+2y\lambda}\\0&1\end{pmatrix},&\lambda\in Y_4,\\
\lambda^{-\Theta_{\infty}\sigma_3}E^{-1}\lambda^{-\Theta\sigma_3},&\lambda\in Y_5,\\
\lambda^{-\Theta_{\infty}\sigma_3}\begin{pmatrix}1&0\\-s_1\eul^{-\lambda^2-2y\lambda}&1\end{pmatrix}E^{-1}\lambda^{-\Theta\sigma_3},&\lambda\in Y_6,\\
\lambda^{-\Theta_{\infty}\sigma_3}\begin{pmatrix}1+s_1s_2&-s_2\eul^{\lambda^2+2y\lambda}\\-s_1\eul^{-\lambda^2-2y\lambda}&1\end{pmatrix}E^{-1}\lambda^{-\Theta\sigma_3},&\lambda\in Y_7,\\
\lambda^{-\Theta_{\infty}\sigma_3}\begin{pmatrix}1+s_1s_2&-s_2\eul^{\lambda^2+2y\lambda}\\\left(-s_1-(1+s_1s_2)s_3\right)\eul^{-\lambda^2-2y\lambda}&1+s_2s_3\end{pmatrix}E^{-1}\lambda^{-\Theta\sigma_3},&\lambda\in Y_8\cup Y_9,
\end{cases}
\end{equation}
where the branch cut in \eqref{eq:vp4} runs along $\ii\D{R}_-$ (so that $\lambda^a=\eul^{a(\ln\abs{\lambda}+\ii\arg\lambda)}$ with $\arg \lambda\in(-\pi/2,3\pi/2]$ for $a\in\D{C}$) and $(\lambda^{-2\Theta_{\infty}})_+$ denotes the boundary value of $\lambda^{-2\Theta_{\infty}}$ from the left. The complex constants $\accol{s_j}_1^4$ in \eqref{eq:vp4} parametrize the solutions of \eqref{eq:p4} and must obey the relation
\[
(1+s_2s_3)\eul^{2\ii\pi\Theta_{\infty}}+\left(s_1s_4+(1+s_3s_4)(1+s_1s_2)\right)\eul^{-2\ii\pi\Theta_{\infty}}=2\cos(2\pi\Theta),
\]
while $E$ is a certain unimodular eigenmatrix. Letting $\Theta_{\infty}=\ii\tilde\nu$, $\Theta=0$, and
\begin{equation}   \label{eq:s1234}
s_1=\frac{\rho\rho^*}{1+\rho\rho^*},\quad s_2=-1,\quad s_3=-\rho\rho^*,\quad s_4=\frac{1}{1+\rho\rho^*},
\end{equation}
we can take $E$ to be the identity matrix. Then, after shifting the branch cut from $\ii\,\D{R}_-$ to $\C{C}$, the jump matrix $v^{\painl}$ defined in \eqref{eq:vp4} reduces exactly to the jump matrix $v^Y$ in \eqref{eq:vY}. Thus the RH problem \eqref{eq:rhpY} can be viewed as a special case of the RH problem associated with Painlev\'e~IV. However, the connection with Painlev\'e~IV only applies for $\Theta\neq n/2$, $n\in\D{Z}$. In our case, this connection breaks down because $\Theta=0$.
\end{remark}

\begin{acknowledgements*}
J.~Lenells acknowledges support from the G\"oran Gustafsson Foundation, the Ruth and Nils-Erik Stenb\"ack Foundation, the Swedish Research Council, Grant No.~2015-05430, and the European Research Council, Grant Agreement No.~682537.
\end{acknowledgements*}


\begin{bibdiv}
\begin{biblist}
\bib{AS}{book}{
   author={Ablowitz, M.J.},
   author={Segur, H.},
   title={Solitons and the inverse scattering transform},
     publisher={SIAM},
   date={1981},
}
\bib{BG15}{article}{
   author={Bertola, M.},
   author={Giavedoni, P.},
   title={A degeneration of two-phase solutions of the focusing
nonlinear Schr\"odinger equation via Riemann--Hilbert problems},
   journal={J. Math. Phys.},
   volume={56},
   date={2015},
   pages={061507, 17},
}
\bib{BK14}{article}{
   author={Biondini, Gino},
   author={Kova\v ci\v c, Gregor},
   title={Inverse scattering transform for the focusing nonlinear
   Schr\"odinger equation with nonzero boundary conditions},
   journal={J. Math. Phys.},
   volume={55},
   date={2014},
   number={3},
   pages={031506, 22},
}
\bib{BM16}{article}{
   author={Biondini, Gino},
   author={Mantzavinos, Dionyssios},
   title={Universal nature of the nonlinear stage of modulational
   instability},
   journal={Phys. Rev. Lett.},
   volume={116},
   date={2016},
   number={4},
   pages={043902},
}
\bib{BM17}{article}{
   author={Biondini, Gino},
   author={Mantzavinos, Dionyssios},
   title={Long-time asymptotics for the focusing nonlinear Schr\"odinger
   equation with nonzero boundary conditions at infinity and asymptotic
   stage of modulational instability},
   journal={Comm. Pure Appl. Math.},
   volume={70},
   date={2017},
   number={12},
   pages={2300--2365},
}
\bib{BLS21}{article}{
   author={Boutet de Monvel, Anne},
   author={Lenells, Jonatan},
   author={Shepelsky, Dmitry},
   title={The focusing NLS equation with step-like oscillating
   background: scenarios of long-time asymptotics},
   journal={Commun. Math. Phys.},
   volume={383},
   date={2021},
   number={2},
   pages={893--952},
}
\bib{BLS22}{article}{
   author={Boutet de Monvel, Anne},
   author={Lenells, Jonatan},
   author={Shepelsky, Dmitry},
   title={The focusing NLS equation with step-like oscillating
   background: the genus 3 sector},
   journal={Commun. Math. Phys.},
   date={2022},
   status={to appear},
   eprint={https://arXiv:2005.02822},
}
\bib{BV07}{article}{
   author={Buckingham, Robert},
   author={Venakides, Stephanos},
   title={Long-time asymptotics of the nonlinear Schr\"odinger
   equation shock problem},
   journal={Comm. Pure Appl. Math.},
   volume={60},
   date={2007},
   number={9},
   pages={1349--1414},
}
\bib{DVZ94}{article}{
   author={Deift, P.},
   author={Venakides, S.},
   author={Zhou, X.},
   title={The collisionless shock region for the long-time behavior of
solutions of the KdV equation},
   journal={Comm. Pure Appl. Math.},
   volume={47},
   date={1994},
   pages={199--206},
} 
\bib{DVZ97}{article}{
   author={Deift, P.},
   author={Venakides, S.},
   author={Zhou, X.},
   title={New results in small dispersion KdV by an extension of the
   steepest descent method for Riemann-Hilbert problems},
   journal={Internat. Math. Res. Notices},
   date={1997},
   number={6},
   pages={286--299},
}
\bib{DZ93}{article}{
   author={Deift, P.},
   author={Zhou, X.},
   title={A steepest descent method for oscillatory Riemann-Hilbert
   problems. Asymptotics for the MKdV equation},
   journal={Ann. of Math. (2)},
   volume={137},
   date={1993},
   number={2},
   pages={295--368},
}
\bib{Fo06}{book}{
   author={Fokas, Athanassios S.},
   author={Its, Alexander R.},
   author={Kapaev, Andrei A.},
   author={Novokshenov, Victor Yu.},
   title={Painlev\'e transcendents},
   series={Mathematical Surveys and Monographs},
   volume={128},
   note={The Riemann-Hilbert approach},
   publisher={American Mathematical Society, Providence, RI},
   date={2006},
   pages={xii+553},
}
\bib{KLM03}{book}{
   title={Semiclassical Soliton Ensembles for the Focusing Nonlinear 
	Schr\"odinger Equation},
   author={Kamvissis, Spyridon},
   author={McLaughlin, Kenneth D.T-R.},
   author={Miller, Peter D.},
   series={Annals of Mathematics Studies},
   volume={169},
   publisher={Princeton University Press, Cambridge},
   date={2003},
}
\bib{Le17}{article}{
   author={Lenells, Jonatan},
   title={The Nonlinear Steepest Descent Method for Riemann--Hilbert
   Problems of Low Regularity},
   journal={Indiana Math. J.},
   volume={66},
   date={2017},
   number={4},
   pages={1287--1332},
}
\bib{Le18}{article}{
   author={Lenells, Jonatan},
   title={Matrix Riemann--Hilbert problems with jumps across Carleson
   contours},
   journal={Monatsh. Math.},
   volume={186},
   date={2018},
   number={1},
   pages={111--152},
}
\bib{Ol10}{collection}{
   title={NIST handbook of mathematical functions},
   editor={Olver, Frank W. J.},
   editor={Lozier, Daniel W.},
   editor={Boisvert, Ronald F.},
   editor={Clark, Charles W.},
   publisher={Cambridge University Press, Cambridge},
   date={2010},
   pages={xvi+951},
}
\end{biblist}
\end{bibdiv}
\end{document}